\documentclass[11pt,reqno]{article}
\usepackage[utf8]{inputenc}
\usepackage{enumerate}
\usepackage{lmodern}
\usepackage[T1]{fontenc}
\usepackage{verbatim}
\usepackage{textcomp}

\usepackage[english]{babel}
\usepackage[a4paper,vmargin={3.5cm,3.5cm},hmargin={2.5cm,2.5cm}]{geometry}
\usepackage[font=sf, labelfont={sf,bf}, margin=1cm]{caption}
\usepackage[pdftex]{hyperref}
\usepackage[pdftex]{color,graphicx}
\usepackage{xcolor}

\usepackage{amsmath,amsfonts,amssymb,amsthm,mathrsfs,mathtools,bbm}

\newtheorem{theorem}{Theorem}[section]
\newtheorem{corollary}[theorem]{Corollary}
\newtheorem{lemma}[theorem]{Lemma}
\newtheorem{proposition}[theorem]{Proposition}

\theoremstyle{remark}
\newtheorem{remark}[theorem]{Remark}
\theoremstyle{definition}
\newtheorem{definition}[theorem]{Definition}

\def\FNS{\mathbf{F}(\mathbf{N}_\sigma)}

\def\RR{\mathbb{R}}
\def\R{\mathbb{R}}
\def\ZZ{\mathbb{Z}}
\def\NN{\mathbb{N}}

\def\PP{\mathbb{P}}
\def\XX{\mathbb{X}}
\def\EE{\mathbb{E}}
\def\Var{\mathbb{V}\mathrm{ar}}
\def\Cov{\mathbb{C}\mathrm{ov}}

\def\1{\mathbbm{1}}
\def\al{\alpha}
\def\be{\beta}
\def\ga{\gamma}
\def\de{\delta}

\def\ph{\varphi}
\def\la{\lambda}
\def\cP{\mathcal{P}}
\def\cX{\mathcal{X}}

\def\cH{\mathcal{H}} 
\def\sA{\mathsf{A}}
\def\sB{\mathsf{B}}
\def\sC{\mathsf{C}}

\def\sK{\mathsf{K}}
\def\sW{\mathsf{W}}
\def\sS{\mathsf{S}}
\def\sO{\mathsf{O}}
\def\sP{\mathsf{P}}
\def\bx{\mathbf{x}}
\def\by{\mathbf{y}}

\def\bz{\mathbf{z}}
\def\bF{\mathbf{F}}
\def\b0{\mathbf{0}}
\def\bM{\mathbf{M}}

\def\dk{{d_\mathrm{K}}}
\def\dw{d_{\mathrm{W}}}

\def\d2{d_2}
\def\dc{d_\mathrm{c}}

\def\wt{\widetilde}
\def\wh{\widehat}
\def\con{\xleftrightarrow}
\def\scr{\mathscr}

\numberwithin{equation}{section}

\newcommand{\norm}[1]{\lVert#1\rVert}
\newcommand{\op}[1]{\|#1\|_\mathrm{op}}

\DeclareMathOperator\supp{Supp}
\DeclareMathOperator\hess{Hess}

\DeclareMathOperator\dom{dom}

\begin{document}
\title{\bf Quantitative two-scale stabilization \\ on the Poisson space}
\author{Rapha\"el Lachi\`eze-Rey, Giovanni Peccati, Xiaochuan Yang}
\date{\small \today}

\maketitle

\begin{abstract}  We establish inequalities for assessing the distance between the distribution of a (possibly multidimensional) functional of a Poisson random measure and that of a Gaussian element. Our bounds only involve add-one cost operators at the order one -- that we evaluate and compare at two different scales -- and are specifically tailored for studying the Gaussian fluctuations of sequences of geometric functionals displaying a form of weak stabilization -- see Penrose and Yukich (2001) and Penrose (2005). Our main bounds extend the estimates recently exploited by Chatterjee and Sen (2017) in the proof of a quantitative version of the central limit theorem (CLT) for the length of the Poisson-based Euclidean minimal spanning tree (MST). We develop in full detail three applications of our bounds, namely: (i) to a quantitative multidimensional spatial CLT for functionals of the on-line nearest neighbour graph, (ii) to a quantitative multidimensional CLT involving functionals of the empirical measure associated with the edge-length of the Euclidean MST, and (iii) to a collection of multidimensional CLTs for geometric functionals of the excursion set of heavy-tailed shot noise random fields. Application (i) is based on a collection of general probabilistic approximations for strongly stabilizing functionals, that is of independent interest. \\


\end{abstract}

\tableofcontents

\section{Introduction and main results}

\subsection{Overview}\label{ss:overview}

The aim of this paper is to establish a collection of new inequalities, allowing one to prove quantitative central limit theorems (possibly multidimensional, and with respect to non-smooth probabilistic distances) for sequences of geometric functionals of a Poisson random measure displaying a form of {\bf quantitative two-scale stabilisation}. As discussed at length in the sections to follow, the concept of two-scale stabilisation promoted in our work is meant to quantitatively capture and extend the notion of {\bf weak geometric stabilisation} developed by Penrose and Yukich in the fundamental works \cite{Penrose05, PY01, PY02}, building on the ideas exploited by Kesten and Lee \cite{KL96} in order to establish a central limit theorem (CLT) for the length of the Poisson-based {\bf minimal spanning tree}. See the discussion below, as well as \cite{BY05, LRSY, LPS16, PY05, Schreiber} and the references therein, for further details. 
As demonstrated in Sections \ref{ss:ss} and \ref{ss:strongmd}, when applied to strongly stabilizing functionals (that is, to functionals possessing explicit {\bf radii of stabilization}) our inequalities provide novel quantitative bounds in any dimension, that only depend on the radii's tail probabilities. 

The idea of proving quantitative CLTs for geometric functionals of point processes by comparing the fluctuations of difference operators at two distinct scales, has recently appeared in the work by Chatterjee and Sen \cite{CS17},  {which provided the initial impetus for the present work}. In such a reference, a notion of two-scale stabilisation is implicitly used in order to prove a quantitative version of Kesten and Lee's CLT. The form of stabilization exploited in \cite{CS17} emerges in the framework of the method for  {one-dimensional}  normal approximations developed in \cite{C08, LRP}, that is applied {\it via} a discretization procedure. In comparison, our bounds  {hold in any dimension}, do not require any discretization, and are uniquely expressed in terms of {\bf single add-one cost operators} (see Section \ref{ss:frame} for definitions), evaluated over regions expanding at different speeds.

We develop three applications: (i) to the multidimensional fluctuations of edge-length statistics of the {\bf on-line nearest neighbour graph} \cite{Bergeretal, FKP, Penrose05, wade07, wade09}, (ii) to the fluctuations of edge-length functionals of the Poisson-based minimal spanning tree (thus recovering multidimensional versions of some results from \cite{CS17}), and (iii) to vectors of geometric functionals of excursion sets associated with {\bf shot-noise fields} \cite{BdBDE, BieDes12, BST, Lr19}. In the context of (ii), we are also able to prove a quantitative CLT for the number of connected components of the Boolean model with a fixed radius.

Our proofs are based on a combination of {\bf Malliavin calculus} \cite{BPSURVEY, DP, Last,  LP, LPPTRF} and {\bf Stein's method for normal approximation} \cite{CGS, NP} -- following many works that have exploited analogous tools in a geometric context (see \cite{BPSURVEY}, as well as the discussion below). In particular, it is natural to compare our two-scale bounds to the {\bf second order Poincar\'e inequalities} proved in \cite{LPS16}. As discussed in Remark \ref{r:cs} below, unlike the estimates derived in \cite{LPS16} our bounds do not involve iterated add-one cost operators, and are particularly adapted to situations in which analytically dealing with such iterated operators is unfeasible, and the techniques of \cite{LPS16} do not apply (this is the case e.g. for the study of the MST developed in Section 3). In general, it is to be expected that, if the techniques developed in the present paper and those of \cite{LPS16} are both applicable, then the bounds obtained using \cite{LPS16} are tighter --- since our estimates are derived by forcing an artificial two-scale structure potentially slowing down the rate of convergence. This phenomenon is succinctly described in the forthcoming Remark \ref{r:expo} in the special case of the nearest neighbour graph. A form of multiscale second-order Poincaré  inequalities - also inspired by the theory of stabilization - has been recently established in \cite{DG}.

We will now introduce our general framework, as well as the notational conventions that are used throughout the paper.

\subsection{Framework and basic notation}\label{ss:frame} 

Although the proofs of our main estimates --- as detailed in Appendix \ref{s:proofs} --- rely on a pervasive use of  Malliavin calculus, the statements of our results only require few notions of stochastic analysis on configuration spaces, that we recall below.
 
We fix a probability space $(\Omega, \mathcal F, \PP)$, and consider a measurable space $(\mathbb X, \mathcal X)$ endowed with a $\sigma$-finite measure $\la$. We let 
$\mathcal{X}_\lambda:=\{\sB\in\mathcal{X}\,:\,\la(\sB)<\infty\}$
and denote by $\eta=\{\eta(\sB)\,:\,\sB\in\mathcal{X}\}$ a \textbf{Poisson measure} on $(\mathbb X, \mathcal X)$ with \textbf{intensity} $\la$. We recall that the distribution of $\eta$ is fully determined by the following two facts: (i) 
{ for each finite sequence $\sB_1,\dotsc,\sB_m\in\mathcal{X}$ of pairwise disjoint sets, the random variables $\eta(\sB_1),\dotsc,\eta(\sB_m)$ are independent,}
and (ii) for every $\sB\in\mathcal{X}$, the random variable $\eta(\sB)$ has the Poisson distribution with parameter $\la(\sB)$, where the family of Poisson laws is extended to the parameter set $[0,+\infty]$ in the obvious way. From now on, we assume that $\mathcal F$ is the completed $\sigma$-field generated by $\eta$. For $\sB\in\mathcal{X}_\la$, we write $\hat{\eta}(\sB):=\eta(\sB)-\la(\sB)$ and denote by $\hat{\eta}=\{\hat{\eta}(\sB)\,:\,\sB\in\mathcal{X}_\la\}$
the \textbf{compensated Poisson measure} associated with $\eta$. 

In what follows, we will regard the Poisson measure $\eta$ as a random element {taking values} in the space $\mathbf{N}_\sigma=\mathbf{N}_\sigma(\mathbb{X})$, composed of all $\sigma$-finite point measures $\chi$ on $(\mathbb{X},\mathcal{X})$ that satisfy $\chi(\sB)\in\mathbb{N}_0\cup\{+\infty\}$ for all $\sB\in\mathcal{X}$. {Such a space is} equipped with the smallest 
$\sigma$-field {such that}, for each $\sB\in\mathcal{X}$,  the mapping $\mathbf{N}_\sigma\ni\chi\mapsto\chi(\sB)\in[0,+\infty]$ is measurable. Throughout the paper, we shall assume that the process $\eta$ is {\bf proper}, in the sense that $\eta$ can be $\PP$-a.s. represented as
$
\eta = \sum_{n=1}^{\eta(\mathbb{X})}\delta_{Y_n},
$
where $\delta_y$ is the Dirac mass at $y$, and $\{Y_n : n\geq 1\}$ stands for a countable collection of random elements with values in $\mathbb{X}$. A sufficient condition for $\eta$ to be proper is e.g. that $(\mathbb{X},\mathcal{X})$ is a Polish space endowed with its Borel $\sigma$-field, with $\la$ taken to be $\sigma$-finite as above; see \cite[Section 6.1]{LP} and \cite[p. 2-3]{Last} for more details. { We observe that Corollary 3.7 in \cite{LP} ensures that, for each Poisson measure $\eta$, there exists a proper Poisson measure $\eta^*$ which has the same distribution as $\eta$. Since our results depend solely on the law of $\eta$, assuming that $\eta$ is proper is therefore not a restriction} 

Now denote by $\FNS$ the class of all measurable functions $f :\mathbf{N}_\sigma\rightarrow\R$ and by $L^0(\Omega):=L^0(\Omega,\mathcal F)$ the class of all real-valued, measurable functions $F$ on $\Omega$. 
Note that, as $\mathcal F$ is the completion of $\sigma(\eta)$, each $F\in {L}^0(\Omega)$ can be written as $F=f(\eta)$ for some measurable function $f\in \mathbf{F}(\mathbf{N}_\sigma)$. Such a mapping $f$, called a {\bf representative} of $F$, is $\PP\circ\eta^{-1}$--a.s. uniquely defined. Using a representative $f$ of $F$, we can define the so-called \textbf{add-one cost operator} $D =(D_x)_{x\in\mathbb{X}}$ on $L^0(\Omega)$ by 
\begin{equation}\label{defdp}
 D_xF = D_x F(\eta) :=f(\eta+\delta_x)-f(\eta)\,,\quad x\in\mathbb{X}.
\end{equation}
Without further mention, we will use the fact that the mapping $\mathbb X \times \Omega \to \R : (x, \omega) \mapsto D_xF(\eta(\omega))$ is jointly measurable (many facts of a similar nature are exploited below without mention). We also stress that the definition of $DF$ is $\PP\otimes \la$-a.e. independent of the choice of the representative $f$ --- see \cite[Lemma 2.4]{LP}. In order to simplify the discussion, from now the following convention is in order: when introducing a generic random variable $F\in L^0(\Omega)$, we will once and for all (implicitly) select one of its representatives and denote such a representative mapping by the same symbol $F$. In view of such a convention, we will use capital letters $F,\, G, \, H$, and so on, both to denote generic elements of $L^0(\Omega)$ and of $\FNS$. For $p\geq 1$, we also write $L^p(\Omega) := L^p(\Omega, \mathcal F, \PP)$, and $L^p(\PP\otimes \la) := L^p(\Omega\times \mathbb X, \mathcal F\otimes \mathcal X, \PP\otimes \la)$.


Given $\sB\in \mathcal X $ and $\chi\in \mathbf{N}_\sigma$, we denote by $\chi |_\sB$ the restriction of $\chi$ to the set $\sB$. Given $F\in \FNS$ and $\sB\in \mathcal X $, we write 
\begin{equation}\label{e:fb}
F(\sB) = F(\sB)(\eta) := F ( \eta|_\sB).
\end{equation}
This yields in particular that $D_xF(\sB) = F((\eta+\delta_x)|_\sB) - F ( \eta|_\sB)$, in such a way that $D_xF(\sB) = D_xF(\eta|_\sB)$ (that is, the mapping $\chi\mapsto D_xF(\chi)$ computed at $\chi = \eta|_\sB$) if $x\in \sB$, and equals zero otherwise. Plainly, one has that $F(\mathbb X) = F(\eta)$. Given $y\in \mathbb X$ and $\sB\in \mathcal X $, we also set 
\begin{equation}\label{e:fyb}
F^y(\sB) = F^y(\sB)(\eta) := F((\eta+\delta_y)|_\sB) = D_yF(\sB) +  F ( \sB),
\end{equation}
and therefore 
\begin{equation}\label{e:claro}
D_x F^y(\sB) =F((\eta+ \delta_x+ \delta_y)|_\sB) - F((\eta+ \delta_y)|_\sB) = (D_xF(\sB))^y,\, \quad x\in \mathbb{X}. 
\end{equation}
We will often need to consider collections of sets with the form $ \mathcal A = \{\sA_x : x\in \sB\}$, where $\sB\in \mathcal X $ and $\sA_x \in \mathcal X$ for every $x$. We will say that $\mathcal A$ is {\bf functionally measurable} if the mapping $\sB \times \mathbf{N}_\sigma \to \mathbf{N}_\sigma  : (x, \chi)\mapsto \chi |_{\sA_x}$ is jointly measurable. The reader can check that, if $\mathbb X$ is a vector space and a metric space (with $\mathcal X$ the associated Borel $\sigma$- field) and if $\sA_x = \tau_x \sA := \sA+x$ (translation of $\sA$ by $x$) for some fixed open $\sA\in \cX$, then $\mathcal A$ is functionally measurable, and that the same conclusion holds if $\sA$ is closed; collections of sets of this type are the only ones that are relevant for our applications. One can check that, if each $\sA_x$ is contained in a set $\sC$ with finite measure, then functional measurability is equivalent to the requirement that the mapping $(x,y)\mapsto {\bf 1}_{ y\in \sA_x}$ is jointly measurable.

One crucial situation considered in this paper is given by $\mathbb X = \R^d$ ($d\geq 1$), $\mathcal X = \mathscr{B}(\R^d)$, and $\lambda := t \times {\rm Leb}$, where $t>0$ and ``${\rm Leb}$'' is the Lebesgue measure on $\R^d$. In this setting, one says that $\eta$ is a {\bf homogenous Poisson measure with intensity} $t$. Since the Lebesgue measure has no atoms, it is known that $\eta$ charges singletons with mass either zero or one: it follows that one can identify $\eta$ with its support, denoted from now on by $\mathcal P$. In the homogeneous framework, by a slight abuse of notation and given $F\in L^0(\Omega)$, we will indifferently use the symbols $F(\eta)$ and $F(\mathcal P)$, according to notational convenience. We will also write interchangeably $F(\eta +\delta_{ x})$ and $F(\mathcal P\cup \{ x\} )$ when $x$ is not in the support of $\eta$, and tacitly adopt similar conventions to simplify the presentation. In view of the well-known distributional properties of homogeneous Poisson random measures, in this paper we will only consider, without loss of generality, the case $t=1$ (unit intensity). 

\smallskip

\noindent{\bf Further notation.} Given an integer $m\geq1$, we write $[m] := \{1,...,m\}$. Given two positive numerical sequences $\{a_n,b_n\}$, we write $a_n\asymp b_n$ if $0<c<a_n/b_n<C<\infty$, for constants  $c,C$ independent of $n$.  Given a convex body $\sB \subset \R^d$ and $\varepsilon >0$ , we denote by $\sB^{-\epsilon}$ the collection of those $x$ in the interior of $\sB$ such that $d(x, \partial \sB)> \epsilon$ (with $d$ denoting the Euclidean distance). The {\bf diameter} of a Borel set $\sB$, written ${\rm diam}\,\sB$ is the maximal Euclidean distance between points $x,y$ in the closure of $\sB$. Given $a\in \RR^m$, we write $\| a\|$ to denote the Euclidean norm of $a$. Given $\sB\subset \R^d$, we write $|\sB|$ to indicate the Lebesgue measure of $\sB$.

\subsection{Main results in the one-dimensional case}\label{ss:introone}

\subsubsection{General estimates}

We work within the same framework and notation of the previous section. Given two real-valued random variables $F,\, G$, the  {\bf Kolmogorov distance} between the distributions of $F$ and $G$ is given by 
\begin{align*}
\dk(F,G):= \sup_{z\in\RR} |\PP[F\le z] - \PP[G\le z]|.
\end{align*} 
Given two real-valued integrable random variables $F,G$, the {\bf 1-Wasserstein} distance between the distribution of $F$ and $G$ is given by
\begin{align*}
\dw(F,G): = \sup_{h\in\mathcal H_W} |\EE[h(F)] - \EE[h(G)]|,
\end{align*}
where $\mathcal H_W$ is the set of Lipschitz mappings with Lipschitz constant at most $1$.  It is a well-known fact that $\dk$ and $\dw$ induce topologies on the class of probability measures on $\R$ that are stronger than the topology of convergence in distribution. Basic properties of $\dk$ and $\dw$ are discussed in \cite[Appendix C]{NP}.

The forthcoming Theorem \ref{t:bound_w} yields bounds in the 1-Wasserstein and Kolmogorov distances that are meant to quantitatively capture the concept of weak stabilisation evoked in Section \ref{ss:overview} --- see Remark \ref{r:ws}. In particular, we attach to this result (and to similar estimates below) the label ``abstract two-scale stabilization'', since its statement involves a reference set $\sB$ and a collection of regions $\{\sA_x\}$ that are meant to grow at two different speeds in concrete applications. For the rest of the paper, the symbol $N(0,1)$ denotes a generic centered Gaussian random variable with unit variance. Given a random variable $F$, we write $\EE[|F|^\infty]^{\frac 1 \infty} := {\rm ess}\sup F$.

\begin{theorem}[{Abstract two-scale stabilization, I}] \label{t:bound_w} Let $\eta$ be a Poisson measure on $(\mathbb X, \mathcal X)$ with intensity $\lambda$, and fix $\sB\in \mathcal X$. Consider $F,\,  G\in \FNS $ and define the random variable $F(\sB)$ according to \eqref{e:fb}. Consider a functionally measurable collection  $\{\sA_x : x\in \sB\}\subset \mathcal X$, and assume that $F(\sB), \, G(\sA_x)\in L^2(\Omega)$, for all $x\in \sB$. 
\begin{enumerate}
\item[ \bf (i)] Suppose that there exists $p\in(4,\infty]$ such that
\begin{equation}\label{e:m_con}
\sup_{x\in \sB} \left\{ \EE[|D_x F(\sB) |^{p}]^{\frac 1 p} + \EE[|D_x G(\sA_x) |^{p}]^{\frac 1 p} \right\} :=K<\infty,
\end{equation}
and let $\wh F :=(F(\sB) -\EE F(\sB) )/\sigma$, where $\sigma^2 := \Var F(\sB)>0$. Then, for the constant $c := 3\max(1,K)^3$ it holds that 
\begin{eqnarray}
\frac1c\dw(\wh F(\sB) ,N(0,1)) &\le& \frac{1 }{\sigma^2} \sqrt{\iint_{\sB^2_\Delta} \EE[|D_x F(\sB)  - D_x G(\sA_x)|]^{1-\frac{4}{p}} \lambda^2(dx, dy)}  \label{e:a1w} \\
 &&\quad\quad\quad\quad+  \frac{1}{\sigma^2} \sqrt{\lambda^2 (\{(x,y)\in \sB^2: \sA_x\cap \sA_y \ne\emptyset\})} + \frac{\lambda(\sB)}{\sigma^3},\notag
\end{eqnarray}
where $\frac 4\infty := 0$ and
\begin{equation}\label{e:bdelta}
\sB^2_\Delta := \{(x,y)\in \sB^2 : \sA_x\cap \sA_y  = \emptyset\}.
\end{equation}
\item[\bf (ii)] Suppose that condition \eqref{e:m_con} is replaced by the stronger requirement that, for some $p\in (4, \infty]$,
\begin{equation}\label{e:m_con2}
\sup_{x,y\in \sB} \left\{ \EE[|D_x F (\sB) |^{p}]^{\frac 1 p}\! +\! \EE[|D_x G(\sA_x) |^{p}]^{\frac 1 p} \!+\!\EE[|D_x F^y(\sB) |^{p}]^{\frac 1 p} + \EE[|D_x G^y(\sA_x) |^{p}]^{\frac 1 p}  \right\}\! :=K'\!<\infty,
\end{equation}
where $F^y(\sB)$ and $F^y(\sA_x)$ are defined according to \eqref{e:fyb}. Then, for $c' := 7 \max(1,K')^2$,  such that 
\begin{eqnarray}
\label{e:ak} &&\frac{1}{c'}\dk(\wh F(\sB) ,N(0,1)) \\
&&\le \frac{ 1}{\sigma^2} \sqrt{ \iint_{\sB_\Delta^2} \left ( \EE[|D_x F(\sB)  - D_x G(\sA_x)|]^{1-\frac{4}{p}} + \EE[|D_x F^y(\sB)  - D_x G^y(\sA_x)|]^{1-\frac{4}{p}} \right)\lambda^2(dx,dy)   } \notag \\
&&\quad\quad\quad\quad\quad\quad \quad\quad\quad\quad\quad\quad\quad\quad\quad+  \frac{1}{\sigma^2} \sqrt{\lambda^2 (\{(x,y)\in \sB^2: \sA_x\cap \sA_y \ne\emptyset\})} + \frac{\lambda(\sB)}{\sigma^3}. \notag
\end{eqnarray}

\end{enumerate}
\end{theorem}

Note {that the integrand on the right-hand side of \eqref{e:a1w} does not depend on $y$}, in such a way that
\begin{eqnarray*}
&&\iint_{\sB^2_\Delta} \EE[|D_x F(\sB)  - D_x G(\sA_x)|]^{1-\frac{4}{p}} \lambda^2(dx, dy) \\
&&= \int_\sB \EE[|D_x F(\sB)  - D_x G(\sA_x)|]^{1-\frac{4}{p}} \lambda\{ y\in \sB : \sA_x \cap \sA_y = \emptyset\} \lambda(dx).
\end{eqnarray*}
\noindent The most natural way of applying Theorem \ref{t:bound_w} is to select $F = G$, but some extra flexibility is sometimes required, e.g. when dealing with spatial restrictions of linear edge statistics in random graphs --- see the examples discussed in Section \ref{s:online}.

\smallskip

\begin{remark}\label{r:cs} We will see below that one can effectively bound the expectations appearing in the previous statement by using the elementary estimate
\begin{equation}\label{e:nice}
\EE[|D_x F(\sB)  - D_x G(\sA_x)|] \leq K \, \PP[ D_x F(\sB)  \neq D_x G(\sA_x)]^{1- 1/p},
\end{equation}
and similarly for other terms appearing in our main statements. Such a bound should be compared with \cite[Proposition 1.5]{LPS16}, according to which, in order to prove closeness to normality, one is required to bound probabilities of the type $\PP[D_xD_y  F(\sB) \neq 0]$, for generic $x,y\in \XX$. As already observed, one of the strengths of our approach is that it does not require to assess the action of iterated add-one cost operators.
\end{remark}

\subsubsection{Euclidean setting}\label{ss:euc1}

In geometric applications, one typically applies Theorem \ref{t:bound_w} in the following dynamical setting:
\begin{itemize}
\item[(I)] $(Z, \mathcal{Z})$ is a measurable space (called the {\bf mark space}) endowed with a probability measure $\pi$, and $\eta$ is a Poisson random measure $(\R^d\times Z, \mathscr{B}(\R^d)\otimes \mathcal{Z} )$ with intensity $\lambda = {\rm Leb}\otimes \pi$ \label{page};

\item[(II)] $\sB  =  \sB_n \times Z$, where $\sB_n := n\sB_0$, $n\geq 1$, and  $\sB_0$ is a convex body of $\R^d$ whose interior contains the origin (by the stationarity of $\eta$, this last requirement is not essential, and only used to simplify the discussion);

\item[(III)] $F = G$ and $\sigma_n^2 := \Var F(\sB_n \times Z)\geq a | \sB_n|$ for some $a\in (0,\infty)$ independent of $n$;

\item[(IV)] For $n\geq 1$ and $(\bx,z) \in \sB_n\times Z$,  $\sA_{(\bx,z)} =\sA_{(\bx,z),n}  =( \tau_\bx(b_n \sB_0) \cap \sB_n ) \times Z$ where $\tau_\bx(\sB) = \bx+\sB$, as before, and $b_n$ is a positive sequence diverging to infinity in such a way that $b_n = o(n)$. Since the definition of $\sA_{(\bx,z),n}$ is independent of $z$ we will simply write $\sA_{(\bx,z),n} = \sA_{\bx,n}$, for every $z\in Z$.

\end{itemize}

\smallskip

Since $\sB_0$ is a convex body containing a neighbourhood of the origin, one has that $|\sB_n| \asymp n^d$, and $$\left|  \{(\bx,\by) \in \sB_n^2: \tau_\bx(b_n \sB_0) \cap \sB_n \cap \tau_\by(b_n B_0) \ne\emptyset\}\right| \asymp n^d b^d_n = o(n^{2d}), $$ as $n\to \infty$. For sufficient conditions implying the lower bound at (III), see e.g. \cite{Penrose05} and \cite[Section 5]{LPS16}. The next statement is a useful direct consequence of Theorem \ref{t:bound_w}.

%
%
%
%

\begin{corollary}\label{c:dw} Let the setting of Points {\rm (I)--(IV)} prevail, and assume that, for every $n$, $F(\sB_n\times Z)$ and $\{\sA_{\bx,n} : \bx\in \sB_n\}$ verify the assumptions of Theorem \ref{t:bound_w}.

\begin{itemize}

\item[\bf (a)] For $n\geq 1$, denote by $K(n) = K(n,p)$ the constant obtained from \eqref{e:m_con} (for some $p>4$)  by taking $F=G$, $\sB = \sB_n\times Z$ and $\sA_\bx = \sA_{\bx ,n}$, and assume that $\limsup_n K(n)<+\infty$. Suppose that, as $n\to \infty$,
\begin{equation}\label{e:wstab1}
\psi(n) := \sup_{\bx \in \sB_n} \int_Z \EE[|D_{(\bx,z)} F(\sB_n\times Z) - D_{(\bx,z)} F(\sA_{\bx,n} )|]^{\frac{p-4}{p}} \, \pi(dz) \to 0.
\end{equation}
Then, for some finite constant $C$ independent of $n$,
\begin{equation}\label{e:x}
\dw(\wh F(\sB_n\times Z) ,N(0,1)) \leq C\left\{  \psi(n)^{1/2} + \frac{b^{d/2}_n}{n^{d/2}}    \right\}\to 0, \quad n\to\infty.
\end{equation}

\item[\bf (b)] For $n\geq 1$, denote by $K'(n) = K'(n,p)$ the constant obtained from \eqref{e:m_con2} ($p>4$) when $F=G$, $\sB = \sB_n\times Z$ and $\sA_x = \sA_{(\bx,z),n}$, and assume that $\limsup_n K'(n)<+\infty$. Suppose that, as $n\to \infty$, \eqref{e:wstab1} takes place and also that
\begin{equation}\label{e:wstab2}
\phi(n) :=  \sup_{\bx,\by \in (\sB_n)^2_\Delta} \iint_{Z^2} \EE[|D_{(\bx,z)} F^{(\by,u)}(\sB_n\times Z) - D_{(\bx,z)} F^{(\by,u)}(\sA_{\bx,n} )|]^{\frac{p-4}{p}} \pi^2(dz,du) \to 0.
\end{equation}
Then, for some finite constant $C$ independent of $n$,
\begin{equation}\label{e:xx}
\dk(\wh F(\sB_n\times Z) ,N(0,1)) \leq C\left\{  \psi(n)^{1/2}+\phi(n)^{1/2} + \frac{b^{d/2}_n}{n^{d/2}}    \right\}\to 0, \,\, \,\,n\to\infty,
\end{equation}
where we used the notation \eqref{e:wstab1}.

\end{itemize}
\end{corollary}

A remarkable feature of our bounds is that, if \eqref{e:wstab1} and \eqref{e:wstab2} are both verified and $\phi(n)\asymp \psi(n)$, $n\to \infty$, then the right-hand sides of \eqref{e:x} and \eqref{e:xx} converge to zero with the same rate.

\smallskip

\begin{remark} \label{r:homo1} In the previous framework, the case of a homogeneous Poisson measure on $\R^d$ is obtained by taking $Z = \{0\}$ (one-point space) and $\pi$ equal to the Dirac mass at 0. In this case, we can canonically identify $\eta$ with a homogeneous Poisson measure (with unit intensity) on $(\R^d, \mathscr{B}(\R^d))$, and simply write $F(\sB_n\times Z) = F(\sB_n\times \{0\})  = F(\sB_n)$, $\sA_{(\bx,0),n} = \sA_{\bx,n} = \tau_{\bf x}(b_n \sB_0)$, $D_{(\bx, 0) } = D_\bx$, $F^{(\by, 0)} = F^\by$, and so on. In this simplified framework, conditions \eqref{e:wstab1} and \eqref{e:wstab2} boil down, respectively, to: as $n\to\infty$,
\begin{equation}\label{e:wstab11}
 \sup_{\bx \in \sB_n}\EE[|D_\bx F(\sB_n) - D_\bx F(\sA_{\bx,n} )|]\to 0,
\end{equation}
and 
\begin{equation}\label{e:wstab22}
\sup_{\bx,\by \in (\sB_n)^2_\Delta}\EE[|D_\bx F^\by(\sB_n) - D_\bx F^\by(\sA_{\bx,n} )|]\to 0.
\end{equation}
\end{remark}

\smallskip

\begin{remark}[Connection with weak stabilization] \label{r:ws}{\rm As anticipated, conditions \eqref{e:wstab11}--\eqref{e:wstab22} can be directly connected to the notion of {\bf weak stabilization} introduced in \cite{PY01}, where it is proved that, under some mild technical assumptions, sequences of weakly stabilising functionals always verify a CLT, see \cite[Theorem 3.1]{PY01}.
This notion was not emphasised in the original paper, as the accent was put on strong stabilisation (see Section \ref{ss:ss}), but it turned out to be  very useful in situations where the existence of a radius of stabilisation cannot be established, see e.g. \cite{YSA}. In order to connect this notion with our results, define $\mathscr{B}_0$ to be the class of all subsets of $\R^d$ obtained by combining arbitrary translations and dilations of the set $\sB_0$. According to \cite[Definition 3.1]{PY01}, we say that $F\in \FNS$ is {\bf weakly stabilizing} with respect to $\mathscr{B}_0$ if there exists an a.s. finite random variable $D(\b0, \infty)$ such that, for every sequence $\{\sC_n : n\geq 1\}\subset \mathscr{B}_0$ tending to $\R^d$, one has that $D_{\bf 0}F(\sC_n) \to D({\bf 0},\infty)$, a.s.-$\PP$, that is: as the domain of the argument of $F$ diverges to $\R^d$, the add-one cost $D_{\bf 0}F$ converges towards a universal limit, which is independent of the way in which $\R^d$ is approached. {}{Assuming that $F$ is translation-invariant}, and using the fact that the distribution of $\eta$ is also translation-invariant, one sees immediately that, if $F$ is weakly stabilizing, then for every $\bx\in \R^d$ there exists an a.s. finite random variable $D(\bx, \infty)$ such that, for every $\{\sC_n : n\geq 1\}\subset \mathscr{B}_0$ tending to $\R^d$, $D_\bx(\sC_n) \to D(\bx, \infty)$, a.s.-$\PP$. This last relation implies in particular that, adopting the notation of Remark \ref{r:homo1} and for $F$ weakly stabilizing, $D_\bx F(\sB_n) - D_\bx F(\sA_{\bx,n} ) \to 0$, a.s.-$\PP$, for every $\bx\in \R^d$. It is now easily seen that, if $F$ is weakly stabilizing, then a sufficient condition for \eqref{e:wstab1} to take place is that
$$
\sup_{\bx\in \sB_n} \EE[|D_\bx F(\sB_n) - D(\bx, \infty) |], \sup_{\bx\in \sB_n} \EE[|D_\bx F(\sA_{\bx,n} ) - D(\bx,\infty) |]\to 0,
$$
corresponding to a uniform strengthening of the pointwise convergence implied by the weak stabilization condition. Note that the approach developed in the present paper does not require to identify the limits $D(\bx, \infty)$, or even to prove that these limits exist. {}{We eventually observe that --- at the cost of some technicalities and using e.g. \cite[Section 2.4]{Penrose05} --- one could naturally extend the content of the present remark to the case of weakly stabilizing functionals of {\it marked} point processes.}
}
\end{remark}

\begin{remark}\label{r:ruse} For some geometric arguments, it is easier to deal with points $\bx \in \sB_n$ such that $\tau_\bx(b_n \sB_0)\subset \sB_n$, and this might not be true for points $\bx$ that are too close to $\partial \sB_n$. In order to circumvent this difficulty, one can use the strategy described in Remark \ref{r:cs} and modify the bounds \eqref{e:x}, \eqref{e:xx}, as follows. (a) Take any $\theta> {\rm diam}\,\sB_0$ (in such a way that $\sB_0$ is contained in a ball of radius $\theta$ centred at the origin), (b) in \eqref{e:x} and \eqref{e:xx}  replace the quantiy $\psi(n)$ with
\begin{eqnarray*}
&&\psi'(n) := \sup_{\bx \in \sB^{-\theta b_n}_n} \int_Z \PP[D_{(\bx,z)} F(\sB_n\times Z) \neq D_{(\bx,z)} F(\sA_{\bx,n} )]^{(1- \frac{4}{p})(1-\frac 1 p)} \, \pi(dz);
\end{eqnarray*}
 (c) replace $\phi(n)$ in \eqref{e:xx} with
\begin{eqnarray*}
&& \phi'(n)\\
&&  :=\!\!\!  \sup_{ (\bx, \by) \in (\sB_n)^2_\Delta : \bx \in \sB^{-\theta b_n}_n} \iint_{Z^2} \PP[D_{(\bx,z)} F^{(\by,u)}(\sB_n\times Z) \neq D_{(\bx,z)} F^{(\by,u)}(\sA_{\bx,n} )]^{(1- \frac{4}{p})(1-\frac 1 p)} \pi^2(dz,du),
\end{eqnarray*}
and (d) replace in each of the bounds $b_n^{d/2}/n^{d/2}$ with $\sqrt{b_n/n}$.

\end{remark}

\begin{remark}[Optimality of rates]\label{r:optimality} The applications developed in Sections 2, 3 and 4 will demonstrate that the rates of convergence obtained by applying Corollary \ref{c:dw} (and, a fortiori, applying Remark \ref{r:ruse}) can be significantly slower than the rate $O(n^{-d/2})$ that one would heuristically expect in view of the lower bound $\Var F(\sB_n \times Z)\geq a | \sB_n| \asymp n^d$. While in some specific examples it is possible to determine the sub-optimality of two-scale stabilization rates (see e.g. Remark \ref{r:expo}), assessing in general the quality of the upper bounds derived in the present paper is a challenging open problem. 
\end{remark}

{}{

\subsubsection{Strongly stabilizing functionals}\label{ss:ss}

The next definition is a natural adaptation of the definition of ``strongly stabilizing functional'', as given e.g. in \cite[Definition 2.1]{PY01} and \cite[p. 284-285]{PY02}, to the framework of the present paper (where we do not assume that functionals are automatically translation-invariant).  The notion of strong stabilization was implicitly used in the seminal work \cite{KL96}, and then conceptualized, refined and improved in \cite{PY01, PY02}. Stabilization-related techniques have been successfully applied in many geometric problems and extended to geometric binomial input, see for instance the survey  \cite{Schreiber}, as well as the more recent contributions \cite{LRSY, LPS16}. 

\begin{definition}[Strong stabilization, I]\label{d:srs} Let $\eta$ be a Poisson measure on $\R^d \times Z$ as at Point (I) of Section \ref{ss:euc1}, and let $F\in\FNS$. We say that $F$ is {\bf strongly stabilizing} at the point $(\bx, z)\in \R^d\times Z$ if there exists a finite random variable $R = R \{ (\bx,z) ; \eta \} \geq 0$ (called {\bf radius of stabilization}) such that
 {
\begin{align*}
D_{(\bx,z)} F( (\cP\cap (\sB(\bx, R)\times Z)) \cup \mathcal {A}_0) = D_{(\bx,z)} F( \cP\cap (\sB(\bx, R)\times Z)) 
\end{align*}
for all finite sets $\mathcal {A}_0 \subset (\R^d\backslash \sB(\bx, R)) \times Z$}, with $\sB(\bx, R)$ the closed ball centered at $\bx$ with radius $R$.
\end{definition}

One can check that strong stabilization always implies weak stabilization. It is by now a classical fact (see e.g. \cite[Theorem 2.2]{Penrose05}, as well as \cite[Theorem 2.1]{PY01} and \cite[Theorem 3.1]{PY02}) that, if $F$ is strongly stabilizing and verifies some mild regularity conditions, then the normalized sequence $n \mapsto \widehat F(\sB_n\times Z)$ verifies a CLT. Our next statement provides universal upper bounds for such an asymptotic result, expressed in terms of the tail probabilities of radii of stabilization. To the best of our knowledge, the forthcoming Proposition \ref{p:rs1d} and its multidimensional counterpart Proposition \ref{p:rsmd} are the first quantitative normal approximation results for strongly stabilizing functionals, holding under virtually no assumptions on the radii of stabilization (other than such radii are assumed to be a.s. finite).

\begin{proposition}[Quantative CLTs under strong stabilization, I]\label{p:rs1d} We work in the setting of Points {\rm (I)--(IV)} of Section \ref{ss:euc1}. Let $\theta>{\rm diam} \, \sB_0$ and let $c>0$ be such that $\sB_0$ contains a ball of radius $c$ centered at the origin. Consider a functional $F\in \FNS$, that is strongly stabilizing at every $(\bx ,z) \in \R^d\times Z$ with corresponding radius denoted by $R\{(\bx,z) ; \eta\}$.  Assume that $F$ verifies the assumptions of Corollary \ref{c:dw}-{\bf (a)}, with \eqref{e:wstab1} replaced by
\begin{eqnarray}\label{e:zuz}
&&\psi''(n) := \sup_{\bx \in \sB^{-\theta b_n}_n} \int_Z \PP[R\{ (\bx,z) ; \eta\}\geq c\, b_n\}]^{(1- \frac{4}{p})(1-\frac 1 p)} \, \pi(dz) \to 0.
\end{eqnarray}
Then, for some finite constant $C$ independent of $n$,
\begin{equation}\label{e:y_0}
\dw(\wh F(\sB_n\times Z) ,N(0,1)) \leq C\left\{  \psi''(n)^{1/2} + \sqrt{ \frac{b_n}{n}}  \,  \right\}\to 0, \quad n\to\infty.
\end{equation}
Now suppose that $F\in \FNS$ is strongly stabilizing at every $(\bx,z) \in \R^d\times Z$, and also that, for every fixed $(\by,u)\in \R^d\times Z$, the functional $F^{(\by, u)}$ (defined according to \eqref{e:fyb}) is strongly stabilizing at every $(\bx,z) \in \R^d\times Z$; denote by $R \{ (\bx,z) ; (\by, u); \eta  \}$ the corresponding radius of stabilization. Assume that $F$ verifies the assumptions of Corollary \ref{c:dw}-{\bf (b)} with \eqref{e:wstab1} replaced by \eqref{e:zuz} and \eqref{e:wstab2} replaced by
\begin{eqnarray*}
&& \phi''(n) :=   \sup_{  (\bx, \by) \in (\sB_n)^2_\Delta :  \bx \in \sB^{-\theta b_n}_n} \iint_{Z^2} \PP[R \{ (\bx,z) ; (\by, u); \eta  \}\geq c\, b_n )]^{(1- \frac{4}{p})(1-\frac 1 p)} \pi^2(dz,du) \to 0.
\end{eqnarray*}
Then, for some finite constant $C$ independent of $n$,
\begin{equation}\label{e:y}
\dk(\wh F(\sB_n\times Z) ,N(0,1)) \leq C\left\{  \psi''(n)^{1/2}+ \phi''(n)^{1/2} + \sqrt{ \frac{b_n}{n}}  \,  \right\}\to 0, \quad n\to\infty.
\end{equation}

\end{proposition}

Our assumptions about the finiteness of the constants $\sup_n K(n)$ and $\sup_n K'(n)$ can be regarded as slight strengthenings of the {\bf fourth moment conditions} exploited in the proof of the already mentioned CLTs \cite[Theorem 2.2]{Penrose05}, \cite[Theorem 2.1]{PY01} and \cite[Theorem 3.1]{PY02}; on the other hand, our approach does not require to assume a priori any form of polynomial boundedness (such as e.g. \cite[formula (2.17)]{Penrose05}). Since $b_n\to\infty$, a sufficient condition for $\psi''(n)\to 0$ is that the radii $R\{ (\bx,z), \eta\}$ have a distribution independent of $\bx$ (this happens e.g. under the very general assumption that $F$ is translation-invariant). If $\psi''(n)\to 0$, a sufficient condition for $\phi''(n)\to 0$ is that the radii of stabilisation verify the monotonicity property $R\{ (\bx,z), \eta\} \geq R \{ (\bx,z) ; (\by, u); \eta  \} $, a.s.-$\PP$, for every $(\by, u)$. A detailed application of Proposition \ref{p:rs1d} to the {\it on-line nearest neighbour graph} is described in Section \ref{s:online}. 
\begin{proof}[Proof of Proposition \ref{p:rs1d}] In view of Remark \ref{r:ruse}, it is sufficient to show that, for every $n\geq 1$, one has that $\phi'(n)\leq \phi''(n)$ and $\psi'(n)\leq \psi''(n)$. To show these relations, fix $\bx\in \sB^{-\theta b_n}_n$, and observe that, on the event  $\{ R\{ (\bx,z) ; \eta\} < c\, b_n\} $, one has that $D_{(\bx,z)} F(\sB_n\times Z) = D_{(\bx,z)} F(\sA_{\bx,n} )$. Reversing the last implication yields that 
$$
\PP[D_{(\bx,z)} F(\sB_n\times Z) \neq D_{(\bx,z)} F(\sA_{\bx,n} )] \leq \PP[ R\{ (\bx,z) ; \eta\} \geq c\, b_n\} ],
$$
from which the desired bound on $\psi'(n)$ follows. The bound on $\phi'(n)$ is deduced by the same argument, after replacing $F$ with $F^{(\by, u)}$.
\end{proof}

\begin{remark}[Speed of convergence]\label{r:expo}{\rm The fastest possible speed of convergence to zero of the right-hand sides of \eqref{e:y_0} and \eqref{e:y} is $\sqrt{b_n/n}$, which is in general much slower than the presumably optimal rate $O(n^{-d/2})$. The reason of the presence of the term $\sqrt{b_n/n}$ is that in our proof (which is based on Remark \ref{r:ruse}) we bound uniformly in $n$ all quantities integrated over the region $\sB_n\backslash \sB_n^{-\theta b_n}$. It is also clear that one can in principle improve \eqref{e:y_0} and \eqref{e:y} as follows: on the right-hand side of \eqref{e:y_0}, replace the term $\sqrt{b_n/n}$ with
\begin{equation}\label{e:zuz2}
 \sqrt{ \frac{b^d_n}{n^d}} +  \sqrt{ \frac{b_n}{n}}\left(\sup_{\bx \in \sB_n \backslash \sB^{-\theta b_n}_n} \int_Z \PP[D_{(\bx,z)} F(\sB_n\times Z) \neq D_{(\bx,z)} F(\sA_{\bx,n} )]^{(1- \frac{4}{p})(1-\frac 1 p)} \, \pi(dz)\right)^{1/2} \!\!\!\!\!;\end{equation}
on the right-hand side of \eqref{e:y},  replace the term $\sqrt{b_n/n}$ with
\begin{eqnarray}\label{e:y2}
&& \sqrt{ \frac{b^d_n}{n^d}} \! + \!  \sqrt{ \frac{b_n}{n}}\left(\sup\!\! \iint_{Z^2} \!\! \PP[D_{(\bx,z)} F^{(\by,u)}(\sB_n\times Z) \!\neq \! D_{(\bx,z)} F^{(\by,u)}(\sA_{\bx,n} )]^{\gamma} \, \pi^2(dz, du) \right)^{1/2}\!\!\!\!\!,\,\,\,\,\,\,\,\,\,\,\,\,
 \end{eqnarray}
where $\gamma:= (1- \frac{4}{p})(1-\frac 1 p)$ and the sup is taken over the set $\{ (\bx, \by) \in (\sB_n)^2_\Delta : \bx \in \sB_n \backslash \sB^{-\theta b_n}_n \}$. As an example, one can apply \eqref{e:zuz2}--\eqref{e:y2} to the power-weighted lengths of the $k$-nearest neighbour graph considered e.g. in \cite[Section 7.1]{LPS16}, and deduce rates of convergence to normal of the order $\frac{(\log n)^c}{n^{d/2}}$, where $c>0$ is some explicit constant (details are omitted). Such an estimate differs from the rate provided by \cite[Theorem 7.1]{LPS16} only by the factor $(\log n)^c$. As demonstrated in Sections 2--4, the main advantage of our approach is that it applies to situations where using second-order Poincar\'e inequalities is unfeasible at the moment. Bounds \eqref{e:zuz2}--\eqref{e:y2} will be directly applied to the proof of Part {\bf (a)} of Theorem \ref{t:mainonng}.
}
\end{remark}

\begin{remark} Several geometric functionals of interest can be written as 
\begin{align*}
F(\sB) =  \sum_{\bx\in\sB} \xi(\bx , \eta|_\sB ),
\end{align*}
where $\xi:\RR^d\times\mathbf{N}_\sigma \to \RR$ is called the {\bf score function} of $F$. In this case, CLTs can be deduced by exploiting the powerful theory of score-stabilizing functionals -- see e.g. \cite{BY05}. Since our results are ramifications of the add-one-cost stabilization theory, for the sake of brevity we do not recall the precise definition of stabilizing scores and refer the reader to \cite{BY05,LRSY, PY05, Schreiber} for details.

\end{remark}

}

\subsubsection{Behind the scenes: a new parsimonious bound in the Kolmogorov distance}

The bound \eqref{e:xx} follows from a new estimate in the Kolmogorov distance, whose statement de facto removes two redundant terms in the  general bounds obtained by Schulte \cite{Schulte16} and Eichelsbacher and Th\"ale \cite{ET14}. Such an estimate --- which is of an independent interest --- is inspired by the recent work by Shao and Zhang \cite{SZ19}, and is highlighted in the next statement. In what follows, we work under the general framework and notation of Section \ref{ss:frame}, and use the symbols $\mathbb{D}^{1,2}$, $L^{-1}$ and $\dom \delta$ to indicate, respectively, the space of functionals having a square-integrable add-one cost, the pseudo-inverse of Ornstein-Uhlenbeck generator and the domain of the Kabanov-Skorohod integral $\delta$. These notions are defined in Section \ref{appendix}.

\begin{theorem} \label{l:bound_kol}Let $F\in \mathbb{D}^{1,2}$ and $\wh F=(F-\EE F)/\sigma$ with $\sigma\in(0,\infty)$. Then, if
\begin{equation}\label{e:hg}
\EE \int_\XX\int_\XX [D_y (D_x F| D_xL^{-1}F |) ]^2 \lambda^2(dx,dy)<\infty,
\end{equation}
and $F$ verifies the property that 
\begin{equation}\label{e:hhgg}
F h(F)\in \mathbb D^{1,2}\quad \mbox{for all bounded measurable $h:\RR\to \RR$},
\end{equation}
one has that $DF|DL^{-1}F| \in \dom \delta$, and
\begin{align*}
\dk\left( \wh F,  N(0,1)\right) &\le \left|1 - \frac{\Var[F]}{\sigma^2}\right|+ \frac{1}{\sigma^2}\EE\big[ |\Var[F]-\langle DF, -DL^{-1} F\rangle|\big] + \frac{2}{\sigma^2}\EE[ |\de(DF|DL^{-1}F|)|],
\end{align*}
where $\langle \cdot,\cdot \rangle$ denotes the inner product in $L^2(\lambda)$.
\end{theorem}

\begin{remark}\label{r:hhgg} Using the relation $D(Fh(F)) = h(F) DF+FDh(F) +DFDh(F)$, one sees that a sufficient condition for \eqref{e:hhgg} to hold is that $F\in \mathbb{D}^{1,2}$ is $\sigma(\eta |_\sB)$-measurable, for some $\sB$ such that $\lambda(\sB)$ is finite.
\end{remark}

\subsection{Main results in the multi-dimensional case}\label{ss:intromulti}

\subsubsection{General estimates}

We will express our multidimensional bounds in terms of two smooth distances between probability measures on $\RR^m$ ($m\geq 2$), written $\d2$ and $d_3$, as well as of the convex distance $d_c$. The three distances $\d2, d_3, \dc$ (defined below) all induce topologies on the class of probability measures on $\R^m$ that are stronger than the topology of convergence in distribution. 

For every $m\geq 2$, we denote by $\mathcal H^2_m$ the set of all $C^2$-functions $h:\RR^m\to \RR$ such that 
\begin{align*}
|h(\mathbf{x})-h(\mathbf{y})|\le\norm{\mathbf{x}-\mathbf{y}}, \,\,\mathbf{x}, \mathbf{y}\in\RR^m, \mbox{ and } \sup_{\mathbf{x}\in\RR^m} \op{\hess h(\mathbf{x})}\le 1,
\end{align*}
where $\|\cdot\|$ is the Euclidean norm in $\R^m$, $\hess h$ denotes the Hessian matrix of $h$ and $\op{\cdot}$ is the operator norm. Similarly, we write $\mathcal H^3_m$ to indicate the collection of all thrice continuously differentiable functions on $\R^m$ such that all partial derivatives of order 2 and 3 are bounded by 1. Now fix $a\in \{2,3\}$ and let $\mathbf{Y}, \mathbf{Z}$ be $\RR^m$-valued random vectors such that $\EE \| \mathbf{Y}\|^{a-1}, \EE\|\mathbf{Z}\| ^{a-1}<\infty$; following \cite{PZ10}, we define 
\begin{align*}
d_a(\mathbf{Y},\mathbf{Z}):= \sup_{h\in\mathcal H^a_m} |\EE h(\mathbf{Y}) - \EE h(\mathbf{Z})|.
\end{align*}
Finally, let $\mathcal C_m$ be the class of all convex subsets of $\R^m$. Given $\RR^m$-valued random vectors $\mathbf{Y}, \mathbf{Z}$, we define the {\bf convex distance} between the distribution of $\mathbf{Y}$ and $\mathbf{Z}$ as
 \begin{align*}
\dc(\mathbf Y,\mathbf Z) : = \sup_{C\in \mathcal C_m} \big|\PP[ \mathbf{Y}\in C] - \PP[ \mathbf{Z}\in C] \big|. 
\end{align*}
For a discussion of the properties of the distance $\dc$, see \cite{NPY, SY18} and the references therein.

\smallskip

We now fix $m\geq 2$ and adopt the general framework of Section \ref{ss:frame}, namely: $\eta$ is a Poisson measure on $(\mathbb X, \mathcal X)$ with $\sigma$-finite intensity $\lambda$. Fix $\sB\in \mathcal X$, as well as a functionally measurable collection $\{\sA_x : x\in \sB\} \subset \mathcal X$, and consider mappings $F_1,...,F_m, G_1,..., G_m\in \FNS$ such that $F_i(\sB),G_i(\sA_x) \in L^2(\Omega)$ for every $i=1,...,m$ and every $x\in \sB$. We also set $\mathbf{\wh F}(\sB) :=( \wh F_1(\sB),...,\wh F_m(\sB) )$ where $\wh F_i(\sB):= (F_i(\sB) -\EE F_i(\sB))/\sigma_i$, where $\sigma_i>0$ (the most natural choice is of course $\sigma^2_i = \Var F_i(\sB)$, but some flexibility ill help streamlining our discussion). Our aim is to compare the distribution of $\mathbf{\wh F}(\sB)$ with that of a centered $m$-dimensional Gaussian vector $N_\Sigma = (N_1,...,N_m)$ with covariance matrix $\Sigma = \{\Sigma(i,j) : i,j=1,...,m\} \geq 0$. We will use the following parameters, defined for $q,p>0$ (recall also \eqref{e:bdelta}):

\begin{align*}
\ga_1&:=  \sum_{i,j=1}^m \frac{1}{\sigma_i\sigma_j}|\Sigma(i,j)\sigma_i\sigma_j-\Cov(F_i(\sB),F_j(\sB) )|, \\
\ga_2^{q,p}&:=\left(\sum_{i=1}^m \frac{1}{\sigma_i}\right)^{2} \sqrt{\iint_{\sB_\Delta^2} \sup_{i\in[m]}\EE[|D_x F_i(\sB) - D_x G_i(\sA_x)|]^{1-\frac{q}{p}}  \lambda^2 (dx, dy)}, \\
\ga_3&:=\left(\sum_{i=1}^m \frac{1}{\sigma_i}\right)^{2}  \sqrt{\lambda^2 (\sB^2\setminus \sB_{\Delta}^2) }, \quad\quad \ga_4:= \left(\sum_{i=1}^m \frac{1}{\sigma_i}\right)^{3} \lambda(\sB).
\end{align*}

Our first statement applies to the smooth distances $d_2$ and $d_3$.

\begin{theorem}[Abstract two-scale stabilization, II]\label{t:d_2} Let the previous notation and assumptions prevail, and suppose that there exists $p\in(4,\infty]$ such that 
\begin{align}\label{e:xxx}
\sup_{i\in [m]}\sup_{x\in \sB} \left\{ \EE[|D_x F_i(\sB) |^{p}]^{\frac 1 p} + \EE[|D_x G_i(\sA_x) |^{p}]^{\frac 1 p}  \right\} :=K<\infty.
\end{align}
Then, we have
\begin{align*}
d_3\left(\mathbf{\wh F}(\sB),  N_\Sigma\right) \le 3m \times \max(1,K)^3 \, (\gamma_1+\gamma_2^{4,p}+ \gamma_3+ \gamma_4).
\end{align*}
If, in addition, $\Sigma >0$, then one has also that
\begin{align*}
d_2\left(\mathbf{\wh F}(\sB),  N_\Sigma\right) \le 3m \times \max(1,K)^3 \times b(\Sigma) \times (\gamma_1+\gamma_2^{4,p}+ \gamma_3+ \gamma_4),
\end{align*}
where 
\begin{equation}\label{e:bsigma}
b(\Sigma):= \max\left\{ \op{\Sigma^{-1}} \op{\Sigma}^{1/2} \,\, , \,\, \op{\Sigma^{-1}}^{3/2} \op{\Sigma}\right\}.
\end{equation}

\end{theorem}

\medskip

In order to deal with $\dc$, we need to further define
\begin{align*}
\ga'_2&:= \left(\sum_{i=1}^m \frac{1}{\sigma_i}\right)^{2} \sqrt{\iint_{\sB_\Delta ^2} \sup_{i\in[m]}\EE[|D_x F^y_i(\sB) - D_x G_i^y(\sA_x)|]^{1-\frac{4}{p}}  \lambda(dx, dy) }, \\
\ga_5&:= \left(\sum_{i=1}^m \frac{1}{\sigma_i}\right)^2\sqrt{\lambda(\sB)}.\\
\end{align*}

\begin{theorem}[Abstract two-scale stabilization, III]\label{t:dc} Under the previous notation and assumptions, suppose  that there exists  $p\in(6,\infty]$ such that
\begin{align}\label{e:xxxx}
\sup_{i\in [m]}\sup_{y,x\in \sB} \left\{ \EE[|D_x F_i(\sB) |^{p}]^{\frac 1 p} \! +\! \EE[|D_x G_i(\sA_x) |^{p}]^{\frac 1 p}+\EE[|D_x F^y_i(\sB) |^{p}]^{\frac 1 p} \!+\! \EE[|D_x G^y_i(\sA_x) |^{p}]^{\frac 1 p}  \right\}\! := \! K'<\infty.
\end{align}
Then, we have 
\begin{align*}
\dc(\wh{\mathbf F}(\sB),N_\Sigma) \le c   \left(\ga_1+\ga_2^{4,p}+\gamma _{2}^{5,p}+\gamma _{2}^{6,p}+\ga_2'+\ga_3+\ga_4+\ga_5+ \sum_{i=1}^m \frac{1}{\sigma_i}\right),
\end{align*}
where 
\begin{align*}
c= 60 \times \left( 20\sqrt{2}m + 168 \max(1,K')^4 \op{\Sigma^{-1}}^{\frac 3 2} \Big(m^{\frac 3 2} + m^2\op{\Sigma^{- \frac 1 2}}^{\frac 1 2} \Big)\right)^{\frac 5 2}.
\end{align*}
\end{theorem}

\subsubsection{The Euclidean setting}

For geometric applications, one typically applies Theorem \ref{t:d_2} and Theorem \ref{t:dc}  in the framework of the Points (I), (II) and (IV) presented at the end of Section \ref{ss:introone} (page \pageref{page} above), to which we add the following extension of Point (III):\label{page2}

\begin{enumerate}

\item[(III')] For every $i = 1,...,m$, the random variables $F_i(\sB) = F_{n,i}(\sB_n\times Z)$ and $G_i(\sA_{(\bx,z)} ) =G_i(\sA_{(\bx,z), n} ) = G_{n,i}(\sA_{\bx, n})$ are obtained as follows. First, consider a collection of real-valued kernels $\{h^i(\sC \,;\, \cdot) : \sC\in \mathscr{B}(\R^d) \}$, such that, for each $\sC \in \mathscr{B}(\R^d)$, the mapping $\chi \mapsto h^i( \sC;\chi)$ is an element of $\FNS$. One then fixes a measurable set $\sC_i \subseteq \sB_0$ and defines, for $n\geq 1$ and $\bx\in \R^d$, 
\begin{equation}\label{e:xiaochuan}
F_{n,i}(\sB_n\times Z) := h^i\big( n\sC_i \, ;\,  \eta|_{\sB_n\times Z}\big)\, \, \mbox{and} \,\,G_{n,i}(\sA_{\bx, n}) := h^i\big( n\sC_i  \,;\,  \eta|_{\sA_{\bx, n}}\big).
\end{equation}
We implicitly assume that the functions $h^i$ are such that the mapping $$(\bx, \omega)\mapsto G_{n,i}(\sA_{\bx, n})(\eta(\omega))$$ is jointly measurable (this can be easily verified on specific examples) and also that, as $n\to \infty$, 
\begin{equation}\label{e:sigmainfty}
\frac{1}{n^d}\Cov(F_{n,i}(\sB_n \times Z), F_{n,j}(\sB_n\times Z) )\longrightarrow \Sigma_\infty(i,j), \quad i,j=1,...,m,
\end{equation}
where $\Sigma_\infty = \{\Sigma_\infty(i,j) : i,j=1,...,m\}\geq 0$ has strictly positive diagonal elements. 
\end{enumerate}

\medskip

{}{Although it is not directly applied in the present paper, it is also natural to use the following alternate specification of the random variables $ G_{n,i}$: for $n\geq 1$ and $\bx\in \R^d$, $G_{n,i}(\sA_{\bx, n}) := h^i\big( \tau_{\bf x}( b_n\sC_i) \,;\,  \eta|_{\sA_{\bx, n}}\big)$. {The mappings $h^i$ introduced above verify in most applications the following translation-invariance property (that is not required for the validity of our results):  $h^i(\tau_{\bx} \sC\, ;\, \tau_\bx \chi) = h^i( \sC,  \chi)$ for all $\bx \in \R^d$, where $\tau_\bx \chi$ is the measure on $\mathscr{B}(\R^d)\otimes \mathcal{Z}$ obtained as follows: if $\chi = \sum_\ell v_\ell \, \delta_{\bx_\ell, z_\ell}$, then $\tau_\bx \chi = \sum_\ell v_\ell \, \delta_{\bx_\ell +\bx, z_\ell}$}. Relation \eqref{e:sigmainfty} can be of course a delicate matter --- we refer the reader e.g. to \cite{Penrose05} for a panoply of sufficient conditions implying that such a requirement is met. It is also important to observe that, in many applications, the speed of convergence in \eqref{e:sigmainfty} is difficult to assess,  {and has to be dealt with on a case-by-case basis}. In order to obtain explicit rates and keep the length of the paper within bounds, it is therefore preferable to compare the distribution of $\wh{\mathbf F} = \wh{\mathbf F}_{n}$ with the one of a Gaussian vector having the same covariance structure. This is done in the following statement, which is a direct consequence of Theorem \ref{t:d_2} and Theorem \ref{t:dc}.

\begin{corollary}\label{c:euclmulti}
Let the assumptions and notation defined by Points {\rm (I), (II), (III')} and {\rm (IV)} above prevail, and define
$\mathbf{\wh F}_n(\sB_n\times Z) :=( \wh F_{n,1}(\sB_n\times Z),...,\wh F_{n,m}(\sB_n\times Z) )$ where 
\begin{equation}\label{e:fbtilde}
\wh F_{n,i}(\sB_n\times Z):=\frac{1}{n^{d/2} }(F_{n,i}(\sB_n\times Z) -\EE F_{n,i}(\sB_n\times Z)), \quad i=1,...,m.
\end{equation}
For every $n$, let $N_{\Sigma_n}$ denote a $m$-dimensional centered Gaussian vector with the same covariance matrix $\Sigma_n$ as $\mathbf{\wh F}_n(\sB_n\times Z)$. 
\begin{itemize}
\item[\bf (i)] Denote by $K(n) = K(n,p)$ the constant obtained from \eqref{e:xxx} (for some $p>4$)  by taking $F_i=F_{n,i}$, $G_i = G_{n,i}$, $\sB = \sB_n\times Z$ and $\sA_\bx = \sA_{\bx,n}$, and assume that $\limsup_n K(n)<+\infty$. If, as $n\to\infty$,
\begin{equation}\label{e:goo}
\vartheta(n) := \sup_{i\in[m]}\sup_{\bx\in \sB_n} \int_Z \EE[|D_{(\bx,z)} F_{n,i}(\sB_n\times Z) - D_{(\bx,z)} G_{n,i}(\sA_{\bx,n})|]^{\frac{p-4}{p}} \pi(dz) \to 0,
\end{equation}
then, for some constant $C$ independent of $n$,
\begin{equation}\label{e:bbbz}
d_3(\mathbf{\wh F}_n(\sB_n\times Z), N_{\Sigma_n}) \leq C\left( \vartheta(n)^{1/2} +\frac{b_n^{d/2}}{n^{d/2}}\right) \to 0, \quad n\to \infty.
\end{equation}
\item[\bf (ii)] Under the assumptions of Point {\bf (i)}, assume moreover that $\Sigma_\infty>0$. Then, the bound \eqref{e:bbbz} continues to hold when replacing $d_3$ with $\d2$.

\item[\bf (iii)] Denote by $K'(n) = K'(n,p)$ the constant obtained from \eqref{e:xxxx} (for some $p>6$)  by taking $F_i=F_{n,i}$, $G_i = G_{n,i}$, $\sB = \sB_n\times Z$ and $\sA_\bx = \sA_{\bx,n}$, and assume that $\limsup_n K'(n)<+\infty$. Suppose that $\Sigma_\infty>0$ and also that, as $n\to\infty$,
\begin{equation}\label{e:gooo}
\tau(n) :=\! \! \! \sup_{i\in[m], \, q=4,5,6,\,  \bx\in \sB_n} \int_Z \EE[|D_{(\bx,z)} F_{n,i}(\sB_n\times Z) - D_{(\bx,z)} G_{n,i}(\sA_{\bx,n})|]^{\frac{p-q}{p}} \pi(dz) \to 0, 
\end{equation}
and 
\begin{eqnarray}\label{e:goooo}
&&\varrho(n) \\
&& :=\! \! \! \sup_{i\in[m], \bx,\by \in (\sB_n)^2_\Delta} \iint_Z \EE[|D_{(\bx,z)} F^{(\by,u)}_{i,n}(\sB_n\times Z) - D_{(\bx,z)} G^{(\by,u)}_{i,n}(\sA_{\bx,n})|]^{\frac{p-4}{p}} \pi^2(dz,du) \to 0. \notag
\end{eqnarray}
Then, for some constant $C$ independent of $n$,
\begin{equation}\label{e:bbbb}
\dc(\mathbf{\wh F}_n(\sB_n\times Z), N_{\Sigma_n}) \leq C\left( \tau(n)^{1/2} +\varrho(n)^{1/2} +\frac{b_n^{d/2}}{n^{d/2}}\right) \to 0, \quad n\to \infty.
\end{equation}
\end{itemize}

\end{corollary}

\begin{remark} Following Remark \ref{r:homo1}, we can obtain the case of a homogeneous Poisson measure on $\R^d$ by specializing the previous framework to the case where $Z = \{0\}$ is the one-point space. Arguing as in Remark \ref{r:ws}, it is then possible to directly connect relations \eqref{e:goo} and \eqref{e:gooo} to the notion of {\bf weakly stabilising set function} introduced in \cite[formulae (5.1) and (5.2)]{Penrose05}. For the sake of brevity, we leave the details of this point to the interested reader. 

\end{remark}
\smallskip
{}{ 
\begin{remark}\label{r:ruse2} We will sometimes use a more refined version of Corollary \ref{c:euclmulti}, whose proof is again an immediate consequence of Theorem \ref{t:d_2} and Theorem \ref{t:dc}. To obtain the needed statement, fix an integer $M\geq 2$ (independent of $n$) and, for every $n\geq 1$ and $i=1,...,m$,, define $ \{ \sP_1(i,n),..., \sP_M(i,n)\}$ to be an arbitrary measurable partition of the set $\sB_n$. Then the conclusions of Corollary \ref{c:euclmulti} continue to hold if the quantities $\vartheta(n), \, \tau(n)$ and $\varrho(n)$ are replaced by the constants $\vartheta'(n), \, \tau'(n)$ and $\varrho'(n)$, obtained as follows. The quantity $\vartheta'(n)$ is obtained from the definition of $\vartheta(n)$ by replacing the symbol $\sup_{\bx\in \sB_n}$ with
$$
\max_{\ell = 1,...,M} \sup_{\bx\in \sP_\ell(i, n) } {\frac{ | \sP_\ell(i, n)|}{n^d}};
$$
the quantity $\tau'(n)$ is obtained from $\tau(n)$ by replacing $\sup_{i\in[m], \, q=4,5,6,\,  \bx\in \sB_n}$ with
$$
\sup_{i\in[m], \, q=4,5,6}\max_{\ell = 1,...,M} \sup_{\bx\in \sP_\ell(i, n) } {\frac{| \sP_\ell(i, n)|}{n^d}};
$$
the quantity $\varrho'(n)$ is obtained from $\varrho(n)$ by replacing $\sup_{i\in[m], \bx,\by \in \sB_n}$ with
$$
\sup_{i\in[m] } \max_{\ell = 1,...,M} \sup_{ (\bx, \by) \in (\sB_n)^2_\Delta : \bx\in \sP_\ell(i,n)  } {\frac{| \sP_\ell(i, n)|}{n^d}}.
$$

\end{remark}
}
{}{
\subsubsection{Strongly stabilizing functionals}\label{ss:strongmd}

The next definition is a natural adaptation of the notion of {\it strongly stablizing set function}, as introduced in \cite[formulae (2.13)--(2.14)]{Penrose05}, to the framework of our paper.

\begin{definition}[Strong stabilization, II]\label{d:vrs}  Let $h  = \{h(\sC \,;\, \cdot) : \sC\in \mathscr{B}(\R^d) \}$ be a collection of real-valued kernels such that, for all $\sC \in \mathscr{B}(\R^d)$, the mapping $\chi \mapsto h( \sC;\chi)$ is an element of $\FNS$. We say that $h$ is {\bf strongly stabilizing} at $(\bx, z)\in \R^d\times Z$ if there exist a.s. finite random variables $R = R \{ (\bx,z) ; \eta \} \geq 0$ and $\delta_\infty = \delta_\infty\{ (\bx, z) ; \eta\}$ such that, for every  {finite $\mathcal A_0 \subset (\R^d\backslash \sB(\bx, R)) \times Z$,}
 {
\begin{equation}\label{e:drs1}
D_{(\bx ,z)} h\big( \sC  \,;\, (\cP\cap (\sB(\bx,R)\times Z))\cup \mathcal{ A}_0  \big) = \delta_{\infty}, \quad \mbox{if $\sB(\bx, R)\subseteq  \sC$},
\end{equation}
and 
\begin{equation}\label{e:drs2}
D_{(\bx ,z)} h\big( \sC  \,;\, (\cP\cap (\sB(\bx,R)\times Z))\cup \mathcal{ A}_0  \big) = 0,  \quad \mbox{if $\sB(\bx, R)\subseteq  \R^d\backslash \sC$}.
\end{equation}
}
The random variables $R$ and $\delta_\infty$ are called, respectively, the {\bf radius of stabilization} and the {\bf stabilization limit} of $h$. 
\end{definition}

According e.g. to \cite[Theorem 2.2]{Penrose05}, if the normalized vector $\mathbf{\wh F}_n(\sB_n\times Z)$ is defined as in \eqref{e:fbtilde}, with each $F_{n,i}$ having the form \eqref{e:xiaochuan} for some collection of strongly stabilizing kernels $h^i$, then $\mathbf{\wh F}_n(\sB_n\times Z)$ verifies a multivariate CLT, provided some mild regularity assumptions are satisfied. As discussed in \cite{Penrose05}, such a result can be used to characterize the limit in law (in the sense of finite-dimensional distributions) of a plethora of vector-valued random measures associated with geometric models, such as, e.g., edge length functionals of minimal spanning trees and of nearest neighbour graphs. Our next statement provides explicit universal bounds for such a multivariate limit result; it is a counterpart to Proposition \ref{p:rs1d}.

\begin{proposition}[Quantitative CLTs under strong stabilization, II]\label{p:rsmd} We work under the assumptions and notation of Corollary \ref{c:euclmulti}, and assume moreover that the sets $\sC_i$ have smooth boundary and that each $h^i$, $i =1,...,m$, is strongly stabilizing at every $(\bx, z) \in \R^d \times Z$, with stabilization radius denoted by $R_i \{ (\bx,z)  ; \eta\}$. Let $\theta> {\rm diam}\, \sB_0$, and $c>0$ be such that $\sB_0$ contains a ball of radius $c$ centered at the origin. For $i = 1,...,m$ and $n\geq 1$, write also $\sP_1(i,n)$ to be the collection of those $\bx\in \sB_n$ such that $d(\bx, \partial \sB_n) \leq \theta\,  b_n$ or $d(\bx, \partial (n\sC_i) ) \leq \theta\, b_n$, where $d$ stands for the Euclidean distance of a point from a set, and write $P_2(i,n): = \sB_n\backslash P_1(i,n)$. 

\begin{itemize}

\item[\bf (1)] Suppose that the assumptions of Corollary \ref{c:euclmulti}-{\bf (i)} are satisfied, with \eqref{e:goo} replaced by: as $n\to\infty$,
\begin{equation}\label{e:goooz}
\vartheta''(n) := \sup_{i\in[m]}\sup_{\bx\in \sP_2(i,n)} \int_Z \PP[R_i \{ (\bx,z) ; \eta\} > c\, b_n]^{ (1- \frac{4}{p})(1-\frac1p)} \pi(dz) \to 0;
\end{equation}
then, for some constant $C$ independent of $n$,
\begin{equation}\label{e:bbb}
d_3(\mathbf{\wh F}_n(\sB_n\times Z), N_{\Sigma_n}) \leq C\left( \vartheta''(n)^{1/2} + \sqrt{\frac{b_n }{n} }\right) \to 0, \quad n\to \infty.
\end{equation}
\item[\bf (2)] Under the assumptions of Point {\bf (1)}, assume moreover that $\Sigma_\infty>0$. Then, the bound \eqref{e:bbb} continues to hold when replacing $d_3$ with $\d2$.
\item[\bf (3)] Assume that, for every $i=1,...,m$ and every $(\by, u)\in \R^d\times Z$, the collection of kernels given by $ \chi \mapsto h^i(\sC\, ;\, \chi+\delta_{(\by, u)})$, $\sC\in \mathscr{B}( \R^d )$ is strongly stabilizing at each $(\bx, z)\in \R^d\times Z$, with corresponding radius denoted by $R_i\{ (\bx, z) ; (\by, u) ; \eta\}$. Suppose that the assumptions of Corollary \ref{c:euclmulti}-{\bf (iii)} are satisfied, with \eqref{e:gooo} and \eqref{e:goooo} replaced respectively by:  as $n\to\infty$,
\begin{equation}\label{e:goooz}
\tau''(n) :=\! \! \! \sup_{i\in[m],\,  \bx\in \sP_2(i,n)} \int_Z \PP[R_i \{ (\bx,z) ; \eta\} > c\, b_n]^{ (1- \frac{6}{p})(1-\frac1p)} \pi(dz) \to 0, 
\end{equation}
and 
\begin{eqnarray}\label{e:gooooz}
&&\varrho''(n):=\\
&& \sup_{i\in[m],\,  \bx\in \sP_2(i,n), (\bx,\by)\in (\sB_n)^2_\Delta } \int_Z\int_Z  \PP[R_i \{ (\bx,z) ; (\by, u) ; \eta\} > c\, b_n]^{ (1- \frac{4}{p})(1-\frac1p)} \pi^2(dz,du) \to 0.\notag
\end{eqnarray}
Then, for some constant $C$ independent of $n$,
\begin{equation}\label{e:bbbbzz}
\dc(\mathbf{\wh F}_n(\sB_n\times Z), N_{\Sigma_n}) \leq C\left( \tau''(n)^{1/2} +\varrho''(n)^{1/2} +\sqrt{ \frac{b_n}{n} }\right) \to 0, \quad n\to \infty.
\end{equation}

\end{itemize}

\end{proposition}

\begin{proof} Since the remaining parts of the statement follow by minimal modifications of the exact same argument, we will only prove Point {\bf (1)}. In order to do this, we apply the content of Remark \ref{r:ruse} in the special case of the sequence of partitions $\{\sP_1(i,n), \sP_2(i,n)\}$, $n\geq 1$, defined in the statement. Using the uniform moment bounds on add-one-cost operators and H\"older's inequality, one infers immediately that, for $i=1,...,m$,
\begin{eqnarray*}
&& \sup_{\bx\in \sP_1(i, n) } \frac{ | \sP_1(i, n)|}{n^d} \int_Z \EE[|D_{(\bx,z)} F_{n,i}(\sB_n\times Z) - D_{(\bx,z)} G_{n,i}(\sA_{\bx,n})|]^{\frac{p-4}{p}} \pi(dz) \\
&&= O(n^{-d} n^{d-1}b_n)=O({b_n/n})
\end{eqnarray*}
Since $\frac{ | \sP_2(i, n)|}{n^d}\to 1$, we are therefore left to evaluate the quantity
$$
\sup_{\bx\in \sP_2(i, n) } \int_Z \EE[|D_{(\bx,z)} F_{n,i}(\sB_n\times Z) - D_{(\bx,z)} G_{n,i}(\sA_{\bx,n})|]^{\frac{p-4}{p}} \pi(dz),
$$
for every $i= 1,...,m$. Using again H\"older's inequality and the uniform moment bounds on add-one cost operators, one infers that, for some constant $C$ not depending on $n$,
\begin{eqnarray*}
&&\EE[|D_{(\bx,z)} F_{n,i}(\sB_n\times Z) - D_{(\bx,z)} G_{n,i}(\sA_{\bx,n})|]^{\frac{p-4}{p}}\\
&& \leq C \PP[D_{(\bx,z)} F_{n,i}(\sB_n\times Z) \neq D_{(\bx,z)} G_{n,i}(\sA_{\bx,n})]^{(1 - \frac{4}{p})(1 - \frac{1}{p}) },
\end{eqnarray*}
and the proof is concluded if we manage to show that, for all $\bx\in \sP_2(i, n) $, one has that $\PP[D_{(\bx,z)} F_{n,i}(\sB_n\times Z) \neq D_{(\bx,z)} G_{n,i}(\sA_{\bx,n})]\leq \PP[R_i \{ (\bx,z) ; \eta\} > c\, b_n]$. We distinguish two cases: (i) $ \tau_{\bf x}(b_n \sB_0) \subseteq n \sC_i$ and $d(\bx, \partial \sB_n)>\theta  b_n$, and
(ii) $ \tau_{\bf x}(b_n \sB_0) \subseteq ( n \sC_i)^c$ and $d(\bx, \partial \sB_n)> \theta b_n$. In the case (i), the desired estimate follows as in the proof of Proposition \ref{p:rs1d}, by using property \eqref{e:drs1}. To show the desired bound in the case (ii), we first observe that -- as a consequence of property \eqref{e:drs2} and for $\bx$ verifying (ii) -- one has that $D_{(\bx,z)} G_{n,i}(\sA_{\bx,n} )=0$, a.s.-$\PP$, in such a way that $\PP[D_{(\bx,z)} F_{n,i}(\sB_n\times Z) \neq D_{(\bx,z)} G_{n,i}(\sA_{\bx,n})] = \PP[D_{(\bx,z)} F_{n,i}(\sB_n\times Z) \neq 0]$. The conclusion follows by observing that, for $\bx$ verifying (ii), $\PP[D_{(\bx,z)} F_{n,i}(\sB_n\times Z) \neq 0] \leq \PP[R_i \{(\bx, z) ; \eta\}> c\, b_n]$, where we have once again exploited property \eqref{e:drs2}. 
\end{proof}

As announced, a direct application of Proposition \ref{p:rsmd}  to functionals of on-line nearest neighbour graphs is discussed in Section \ref{s:online}.

}

\subsection{Plan of the paper}

Section 2, 3 and 4 are devoted to applications, respectively, to the on-line nearest neighbour graph, to the minimal spanning tree, and to geometric functionals of shot-noise fields. Appendix A gathers together several fundamental results related to Malliavin calculus and Stein's method. Finally, Appendix B contains the proofs of our main results.

\medskip

\noindent{\bf Acknowledgments.} We thank G\"unter Last, Mathew Penrose, Matthias Schulte, Andrew Wade and Joe Yukich for several useful remarks on some preliminary versions of our work. The research developed in the present paper has been supported by the FNR grants {\bf FoRGES (R-AGR3376-10)} at Luxembourg University,  {and {\bf MISSILe (R-AGR-3410-12-Z)} at Luxembourg and Singapore Universities}.

%

\section{Application to the On-line Nearest Neighbour Graph}\label{s:online}

Our first application concerns the fluctuations of edge length functionals associated with the {\bf on-line nearest neighbour graph} (ONNG) generated by a marked Poisson process on a compact domain. The ONNG is one of the simplest and most treatable models of evolving random spatial networks, and can be seen as a special element of the class of {\bf minimal directed spanning trees} on random point configurations, as described in the survey \cite{PWsurvey} (to which we refer the reader for a discussion of the literature up to the year 2010). Discarding an earlier and somehow unnoticed appearance in \cite{steele}, the ONNG was first described in \cite{Bergeretal}, as a special case of the so-called {\bf FKP model of preferential attachement} \cite{FKP} --- see also \cite{JW} for further details. The denomination `ONNG' was introduced in \cite{Penrose05}.

The most relevant reference for the present section is \cite[Section 3.4]{Penrose05}, where multidimensional CLTs are obtained (among others) for the power-weighted lenghts of the ONNG (see Theorem \ref{t:ponng} below). Our aim is to  deduce several new quantitative counterparts of these results, collected in the statement of Theorem \ref{t:mainonng} below. Such a task is particularly relevant for the theory developed in the present paper since (as observed in \cite[pp. 1963-1964]{Penrose05}) functionals of this type are typically not exponentially stabilising, in such a way that the theory e.g. of \cite{LRSY} cannot be directly applied. 

Laws of large numbers for the edge-length statistics considered in the present section are derived in \cite{wade07}, whereas further second-order results can be found in \cite{PW} (with specific emphasis on the ONNG on the real line) and \cite{wade09} (containing in particular upper bounds on power-lenghts variances in critical cases).  {A further example of a minimal directed spanning tree whose edge length statistics have been studied by Malliavin-Stein techniques is the {\bf radial spanning tree} --- see \cite{BaBo, SchTh}.}

\smallskip

For $d\geq 1$, consider a Poisson measure $\eta$ on the product space $\R^d \times [0,1]$, with intensity given by ${\rm Leb}\otimes dx$, where $dx$ indicates the uniform measure on $(0,1)$. Given a convex body $\sK$ of $\R^d$, we define the $\eta |_{\sK\times [0,1]}$-based ONNG --- written ${\rm ONNG}(\sK) = (V(\sK), E(\sK))$ --- as follows: (i) the set $V(\sK)$ of the vertices  of ${\rm ONNG}(\sK)$ is given by those ${\bf x} \in \sK$ such that there exists $t\in [0,1]$ such that $({\bf x},t)$ is in the support of $\eta|_{\sK\times [0,1]}$, (ii) introduce a total order on $V(\sK)$ by declaring that ${\bf x}\prec {\bf y}$ whenever ${\bf x}$ and ${\bf y}$ are first coordinates in pairs $({\bf x},t), \, ({\bf y},u) \in {\rm supp} \, \eta$ such that $ t<u$, (iii) define the set $E(\sK)$ of the edges of ${\rm ONNG}(\sK)$ by implementing the following procedure: let $N = |V(\sK)|$, and write ${\bf x}_1\prec {\bf x}_2\prec \cdots \prec{\bf x}_N$ to denote the ordering of the elements of $V(\sK)$ described at Point (ii); then, $\{{\bf x}_1, {\bf x}_2\} \in E(\sK)$ and, for every $k=3,...,N$, a further edge is added, linking ${\bf x}_k$ with its nearest neighbour in the set $\{{\bf x}_1,...,{\bf x}_{k-1}\}$. Note that Points (ii) and (iii) use the fact that, with probability one, no two pairs $({\bf x},t), ({\bf y},u)\in {\rm supp}\, \eta$ are such that $t=u$, and the nearest neighbour of ${\bf x}_k$ in $\{{\bf x}_1,...,{\bf x}_{k-1}\}$ is uniquely defined. It is easily seen that the resulting random graph is a tree. It is also natural to interpret the elements of the mark space $[0,1]$ as points in time, in such a way that the total ordering described at Point (ii) can be regarded as a `random order of arrival' for the vertices of ${\rm ONNG}(\sK)$. 

\smallskip

Now consider a measurable function $\varphi : \R_+ \to \R$. In what follows, we will use our main findings in order to study the fluctuations of random variables with the form
\begin{equation}\label{e:length}
L(\varphi; \sK; \sC) := \sum_{x\in V(\sK)\cap \sC} \,\,\,\sum_{\substack{ e\in E(\sK) :\\ x\in e} } \varphi(|e|),
\end{equation}
where $\sC \subseteq \sK $ and $| e|$ is the Euclidean length of $e$; we will also adopt the shorthand notation $L(\varphi; \sK) : = L(\varphi; \sK; \sK)$. One crucial example is that of functions with the form $\varphi(r) = r^\alpha$ ($\alpha>0$), in which case one says that $L(\varphi; \sK; \sC)$ measures the {\bf power-weighted edge length} of the graph ${\rm ONNG}(\sK)$ restricted to the edges having at least one endpoint in $\sC$. Our starting point is the following statement, collecting crucial results from \cite{Penrose05}. For simplicity, in this section we let $\sB_0$ denote the closed unit ball centered at the origin, and set $\sB_n := n\sB_0$. 

\begin{theorem}[See Section 3.4 in \cite{Penrose05}]\label{t:ponng} Assume that the function $\varphi$ verifies
\begin{equation}\label{e:phibound}
\sup_{r>0} \frac{ | \varphi(r) | }{(1+r)^\alpha}<\infty
\end{equation}
for some $\alpha\in (0, d/4)$, and consider a collection of sets $\sC_1,...,\sC_m\subseteq \sB_0$. Then, there exists a constant $ \sigma^2(\varphi)\in [0, \infty)$ such that, as $n\to \infty$,
\begin{equation}\label{e:covonng}
n^{-d}\Cov \{ L(\varphi; \sB_n; n\sC_i), L(\varphi; \sB_n; n\sC_j)  \} \longrightarrow \sigma^2(\varphi) |\sC_i\cap \sC_j| := \Sigma_\infty(i,j), \quad i,j=1,...,m,
\end{equation}
and the random vector
\begin{equation}\label{e:onngvector}
\widetilde{\bf L}_n(\varphi) := \frac{1}{n^{d/2}} \Big( L(\varphi; \sB_n; n\sC_1) - \EE[L(\varphi; \sB_n; n\sC_1),...,L(\varphi; \sB_n; n\sC_m) - \EE[L(\varphi; \sB_n; n\sC_m) ] \Big)
\end{equation}
converges in distribution to a $m$-dimensional centred Gaussian vector with covariance $\Sigma_\infty$.
\end{theorem}

The expression of the constant $\sigma^2(\varphi)$ is explicitly given in \cite[Theorem 2.2]{Penrose05} as $\sigma^2(\varphi) = \EE[\EE[\delta_\infty(\varphi)\mid \mathcal F]^2]$, where $\delta_\infty(\varphi)$ is the {stabilizing limit} of $L(\varphi; \cdot; \cdot)$, see Definition \ref{d:vrs}, and $\mathcal F$ denotes the $\sigma$-field generated by the restriction of $\eta$ to $P \times [0,1]$, where $P := \{(x_1,...,x_d) \in \R^d : x_1<0\}$. One can check that a sufficient condition for $\sigma^2(\varphi)$ to be strictly positive is that $\varphi$ takes values in $\R_+$ and the mapping $\varphi : \R_+ \to \R_+$ is injective. 

As discussed in \cite{Penrose05, wade09}, the CLT stated in Theorem \ref{t:ponng} is conjectured to hold also in the range $\alpha\in [d/4, d/2)$, whereas the case $\alpha = d/2$ is believed to yield a CLT with a logarithmic normalisation \cite[Conjecture 2.2]{wade09}. For the time being, our techniques do not allow to shed light on these open problems. The second-order behaviour of power-weighted functionals in the range $\alpha >d/2$ is studied in \cite{PW, wade09}. 

The main contribution of the present section is the following new quantitative counterpart to the content of Theorem \ref{t:ponng}. In order to state our results, we set $\widetilde{L}_n := (L(\varphi; \sB_n) - \EE[L(\varphi; \sB_n)] )/  \sqrt{ \Var[L(\varphi; \sB_n)]}$, $n\geq 1$. 


\begin{theorem}[Quantitative CLTs for the ONNG]\label{t:mainonng} Let the above notation and assumptions prevail, and suppose that the function $\varphi : \R_+\to \R$ verifies \eqref{e:phibound} for some $\alpha \in (0, (d-(p-4))/p)$ and some $p>4$, and also that $ \sigma^2 = \sigma^2(\varphi)  >0$. 
\begin{itemize}

\item[\bf (a)] There exists a constant $C$, independent of $n$ and possibly depending on $p$, such that
\begin{equation}\label{e:u1}
d_{\rm W}\left( {\widetilde{L}_n}, N\right), \, d_{\rm K}\left( {\widetilde{L}_n}, N\right) \leq C \, n^{ -\frac{d}{4}}, \quad n\geq1,
\end{equation}
where $N$ denotes a standard Gaussian random variable with unit variance. 

\item[\bf (b)] For $m\geq 2$ fix sets $\sC_1,... ,\sC_m\subset \sB_0 $ with smooth boundaries, and denote by ${\bf N}(n)$ a centered $m$-dimensional Gaussian vector with the same covariance structure as $\widetilde{\bf L}_n(\varphi)$, as defined in formula \eqref{e:onngvector}. Then, there exists a constant $C$, independent of $n$ and possibly depending on $p$, such that
\begin{equation}\label{e:uu1}
d_3(\widetilde{\bf L}_n(\varphi), {\bf N}(n))\leq C\, n^{ -\frac{d}{2d+2}}, \quad n\geq 1.
\end{equation}

\item[\bf (c)] If the matrix $\Sigma_\infty$ defined in  \eqref{e:covonng} is positive definite (in particular, if the sets $\sC_1,...,\sC_m$ are disjoint), then the right-hand side of \eqref{e:uu1} is also an upper bound for the quantity $d_2(\widetilde{\bf L}_n(\varphi), {\bf N}(n))$, for every $n\geq 1$. 
 \item[\bf (d)] If the matrix $\Sigma_\infty$ defined in formula \eqref{e:covonng} is positive definite {and $p>6$}, then the right-hand side of \eqref{e:uu1} also upper bounds the quantity $d_{\rm c}(\widetilde{\bf L}_n(\varphi), {\bf N}(n))$, for every $n\geq 1$. 
\end{itemize}

\end{theorem}

Plainly, the assumption that \eqref{e:phibound} is verified for some $\alpha \in (0, (d-(p-4))/p)$ and $p>6$ implicitly requires that $d\geq 3$. 

\begin{proof}[Proof of Theorem \ref{t:mainonng}] We apply Proposition \ref{p:rs1d} in the reinforced version put forward in Remark \ref{r:expo} (for Point {\bf (a)}), and Proposition \ref{p:rsmd} (Points {\bf (b)}--{\bf (d)}) (since $\sB_0$ is the unit ball centred at the origin, one can take $\theta = c=1$). We will use the following specifications of Points (I), (II), (III), (III') and (IV), as listed in Section \ref{page} and Section \ref{page2}: $Z = [0,1]$ and $\pi$ is the uniform measure; at Point (III), $F(\sK \times [0,1]) = G(\sK\times [0,1]) = L(\varphi; \sK)$ (where $\sK = \sB_n$, or $\sK= \tau_{\bf x}(b_n \sB_0) \cap \sB_n $); at Point (III') $ h^i( \sC;\chi) = h( \sC;\chi) =  \sum \varphi(|e]) $, where the sum runs over all edges in the ONNG associated with the points in the support of $\chi$ having at least one endpoint in $\sC$;  $\{b_n\}$ is a sequence such that $b_n = o(n)$, that we shall specify later {\it via} optimization procedures. In order to apply the aforementioned results, we first observe that adapting the arguments contained in \cite[Lemma 3.4 and Lemma 5.1]{Penrose05} yields that, if $\varphi$ verifies \eqref{e:phibound} for some $\alpha \in (0, (d-(p-4))/p)$ and some $p>4$, then the constants $K(n,p)$ and $K'(n,p)$ defined in Corollary \ref{c:dw} and in Corollary \ref{c:euclmulti} are such that $\sup_{n} K(n,p),\, \sup_{n} K'(n,p)<\infty$. Now fix $(\bx, t) \in \R^d\times (0,1)$, and let $M = M(d)$ denote the smallest integer such that there exists a collection of (possibly overlapping) cones $\{C_1,...,C_M \}$ with angular radius equal to $\pi/6$ and point at the origin, such that $\R^d = \cup_{j=1}^M C_j$. Write $C_j({\bf x}) := \tau_{\bf x} C_j$, $j=1,...,M$. For every $j=1,...,M$, we denote by $R( j, {\bf x}, t)$ {the distance from $\bx$ to its nearest neighbour} in the random set
$
\{{\bf z}\in C_j({\bf x}) : ({\bf z}, s)\in \supp\, \eta, \,\mbox{for some}\,\, s<t\};
$
one can prove that $R(j, {\bf x}, t)$ is almost surely finite, and set $R({\bf x}, t) := {2} \max_{j=1,...,M} R( j, {\bf x}, t)$.
According to the discussion contained in \cite[Proof of Theorem 3.6]{Penrose05}, the quantity $R({\bf x}, t)$ is a radius of stabilisation for the collection of kernels $h = \{h (\sC \, ; \, \cdot) : \sC\in \mathscr{B}(\R^d)$\} specified above (and therefore for $L(\varphi; \cdot)$) at $(\bx, t)$ --- see Definitions \ref{d:srs} and \ref{d:vrs}. Since $R({\bf x}, t)$ is defined as a maximum of minima, it is immediately seen that $R({\bf x}, t)$ is also a radius of stabilization for the collection of kernels $\chi \mapsto h (\sC \, ; \, \chi +\delta_{(\by, u)})$, $\sC\in \mathscr{B}(\R^d)$, for every $(\by, u) \in \R^d\times (0,1)$.  Exploiting the isotropy of $\eta$, we see that
$
\PP[ R({\bf x}, t) > b_n ]\leq M\PP[2 R(1, {\bf x}, t) > b_n].
$
Since $\PP[2R_1({\bf x}, t) \geq b_n] $ equals the probability that there are no Poisson points with time mark less than $t$ in the set $C_1({\bf x})\cap \sB({\bf x}, b_n/2)$, we deduce that, for $w=4,6$ and for some finite positive constants $a, k,C$ depending on $d, p$ (where $p>w$),
$$
\int_0^1 \PP[R({\bf x}, t) > b_n]^{(1-1/p)(1- w/p)} dt \leq M
^a \int_0^1 \exp\{ - tkb_n^d\} dt \leq \frac{C}{b_n^d}. 
$$
We can now directly plug such an estimate into the bounds appearing in Proposition \ref{p:rsmd}, and obtain the conclusions of Points {\bf (b)}, {\bf (c)} and {\bf (d)} by optimizing the mapping $x\mapsto \sqrt{\frac{x}{n}} + \sqrt{\frac{1}{x^d}}$ on $\R_+$.  To prove Point {\bf (a)} by using Remark \ref{r:expo}, we have now to evaluate the integrals over $Z = [0,1]$ and $Z^2=[0,1]^2$ appearing on the right-hand side of \eqref{e:zuz2} and \eqref{e:y2}, respectively. In order to do this, for every ${\bf x}\in \sB_n \backslash \sB_n^{-b_n}$ we define a localised version of the quantity $R(\bx, t)$ as follows: keep the notation of the first part of the proof, and set $R(j , \bx, t; n)$ ($j\in [M]$, $t\in (0,1)$, $n\geq 1$) to be the {the distance from $\bx$ to its nearest neighbour} in the random set
$
\{{\bf z}\in C_j({\bf x})\cap \sB_n : ({\bf z}, s)\in \supp\, \eta, \,\mbox{for some}\,\, s<t\}$, if such a set is not empty, and to be the length of the shortest segment contained in $C_j({\bf x})$ connecting $\bx$ to the boundary of $\sB_n$ otherwise; then, define $R({\bf x}, t;n) := {2} \max_{j=1,...,M} R( j, {\bf x}, t;n)$. Reasoning again as in \cite[Proof of Theorem 3.6]{Penrose05}, one sees that, if $R({\bf x}, t;n)\leq b_n$, then $D_{(\bx,t)} F_{n,i}(\sB_n\times Z) = D_{(\bx,t)} G_{n,i}(\sA_{\bx,n})$ and $D_{(\bx,t)} F^{(\by,u)}(\sB_n\times Z) \! = \! D_{(\bx,t)} F^{(\by,u)}(\sA_{\bx,n} )$ for every $(\by, u) \in \R^d\times (0,1)$. Since $\sB_0$ is the unit circle, one has also that $\PP[ R({\bf x}, t;n) > b_n] \leq c' \exp\{ - tk' b_n^d\}$ for some absolute constants $c',k'>0$. After integration in $t$, one infers that both \eqref{e:zuz2} and \eqref{e:y2} are bounded by some multiple of $ \sqrt{\frac{1}{b_n^d}}+\sqrt{\frac{b_n^d}{n^d}}$. Optimizing  $x\mapsto \sqrt{\frac{x^d}{n^d}} + \sqrt{\frac{1}{x^d}}$ over $\R_+$ yields the desired bound.
\end{proof}

\begin{remark} One can in principle prove multidimensional bounds of the same order as \eqref{e:u1}, by formulating an appropriate equivalent of Remark \ref{r:expo}. We leave this task to the motivated reader.
\end{remark}

\section{Applications involving connectivity functionals}

\subsection{ Edge length statistics of the Minimal Spanning Tree}

Our next application concerns the fluctuations of weighted edge length functionals of Euclidean minimal spanning trees (thereafter denoted by MSTs). As one of the most fundamental structures in combinatorial optimisation, the large sample behaviour of random MSTs has attracted much attention in the literature. For an overview of the probability theory of Euclidean combinatorial optimisation problems, we refer to the monograph \cite{Yukich}. The law of large numbers for the total  power-weighted edge length of MST was proved in \cite{Steele88, AS92}. The fact that the total weighted edge length of the MST exhibits Gaussian fluctuations had been a conjecture for several years, until Kesten and Lee \cite{KL96} proved it for general dimensions, and Alexander \cite{Alexander96} for dimension two (with different techniques). As already discussed, Kesten and Lee's proof \cite{KL96} is considered to be the starting point of the geometric theory of stabilization put forward in \cite{PY01, PY02}.  

The most relevant reference of this section is the recent paper by Chatterjee and Sen \cite{CS17} (also containing a detailed review of recent literature), where quantitative univariate central limit theorems were obtained for the total edge length of MST. In order to write precise statements, let $\mathcal{U}$ be a finite subset of $\RR^d$.  The MST of $\mathcal{U}$ is defined by
\begin{align*}
\mathrm{MST}(\mathcal{U}) = \mathrm{Argmin} \sum_{e\in T} |e|,
\end{align*}
where the argmin is taken over all connected graph $T$ with vertex set $\mathcal{U}$, and $|e|$ denotes the length of an edge $e$ in $T$. The obtained graph is necessarily a spanning tree of $\mathcal{U}$. It is classically known that the minimizing graph is unique almost surely when the input $\mathcal{U}$ is provided by a stationary Poisson point process restricted to an arbitrary bounded set (see e.g. \cite{Yukich, KL96}). Consider the total weighted edge-length functional 
\begin{align*}
M(\ph; \sK) := \sum_{e\in \mathrm{MST}(\cP|_{\sK})} \ph(|e|)
\end{align*}   
where $\ph:\RR_+\to\RR$ is a measurable function, $\sK$ is a bounded measurable subset of $\RR^d$ and $\mathcal{P}_\sK := \mathcal{P}\cap \sK$. 
In this section $\sB_0$ is the unit hypercube in $\RR^d$ centered at the origin.  Set as usual $\sB_n=n\sB_0$.  We aim at proving a multivariate central limit theorem extending the following quantitative statement.

\begin{theorem}[See Theorem 2.1 in \cite{CS17}]\label{t:CS17}  Let $M(\sB_n)=M(\ph; \sB_n)$ with $\ph(x)=x$. Consider $\wt M(\sB_n) =(M(\sB_n) - \EE[M(\sB_n)])/\sqrt{\Var[M(\sB_n)]}$. Then
\begin{align*}
\max( \dw(\wt M(\sB_n), N), \dk(\wt M(\sB_n), N)\le 
\begin{cases}
c n^{-\theta} &\mbox{ if } d=2, \\
c' \log(n)^{- \frac d {4p}}  & \mbox{ if } d\ge 3,
\end{cases}
\end{align*}
for some $\theta\in (0,1)$ and all $p>1$, where $c$ is a universal constant and $c'$ is  a finite positive constant that depends only on $p$ and $d$.
\end{theorem}

Our goal in this section is to assess the proximity between a vector of weighted edge-length functionals of the MST and a multivariate normal distribution, by considering simultaneously several weight functions, both in the smooth metrics $d_2, d_3$ and the  convex metric $\dc$. 
%
%
%
 Throughout the section, we assume that $\ph$ is given by $\ph(x)= \psi(x)\1(x\le r)$ for some non-decreasing function $\psi$ and some truncation level $r\in(0,\infty]$. If (and only if) $r= \infty$, we further assume that $\psi$  satisfies the growth condition 
\begin{align}\label{e:growth}
\exists k\in\NN, \quad  \psi(x)\le (1+x)^k \mbox{ and }  \int_0^\infty e^{-c_2 u^d}  d\psi(\sqrt{d} u)<\infty,
\end{align}
where $c_2$ is the constant in the upcoming Lemma \ref{l:walltail}.
One choice of particular interest is $\psi(x)\equiv 1$, in which case, as $r$ varies, the functional $M(\ph,\sB_n)$ corresponds to the empirical distribution function of the edge length of the MST. Another choice of interest is $\psi(x)=x^\al$ and $r=\infty$, in which case $M(\ph,\sB_n)$ becomes a power-weighted edge sum. Hence, one recovers Theorem \ref{t:CS17} by choosing $\psi(x)=x$ and $r=\infty$.  The functional of interest in this section is the $m$-dimensional vector 
 \begin{align*}
 \mathbf{M}(\sB_n):=\big( M(\ph_1;\sB_n), ..., M(\ph_m;\sB_n) \big),
 \end{align*}
 where each $\ph_i$ satisfies the aforementioned conditions.  We also observe that each $M(\ph_j;\sB_n)$ has a variance commensurate to $|\sB_n|$, (the upper bound being a consequence of the Poincar\'e inequality and Proposition \ref{p:moment}, while the lower bound follows from Proposition \ref{p:varlow}), in such a way that the correct normalization is $n^{-d/2}$.  The main result of this section is the following multidimensional extension of Theorem \ref{t:CS17}, providing a quantitative counterpart to the multidimensional limit theorem proved in \cite[Theorem 3.3]{Penrose05}.

\begin{theorem}\label{t:mst} For $n\geq 1$, let ${\bf N} = {\bf N}(n)$ be a centered Gaussian vector with the same covariance matrix as $$n^{-d/2}\mathbf{M}(\sB_n).$$ Then, one has that
\begin{align*}
d_3 (n^{-d/2}(\mathbf{M}(\sB_n)-\EE[\mathbf{M}(\sB_n)]), \mathbf{N}) \le 
\begin{cases}
c n^{-\theta} &\mbox{ if } d=2, \\
c \exp( - c \log\log(n)) & \mbox{ if } d\ge 3,
\end{cases}
\end{align*}
for some $0<\theta<1$, and some finite $c>0$ independent of $n$. The above bound continues to hold for the distances $d_2, d_c$, if we assume that the covariance of $n^{-d/2}\mathbf{M}(\sB_n)$ converges to a positive definitive matrix $\Sigma_\infty$. 
\end{theorem}

We stress that \cite[Theorem 3.3]{Penrose05},  {which is only a qualitative statement}, has however a larger scope than Theorem \ref{t:mst}, since it also applies to restrictions of the MST to subsets of $\sB_n$.  {Such a generalisation can in principle be analysed by using our techniques, and will be investigated elsewhere}.  In order to prove Theorem \ref{t:mst}, we will apply the refined version of Corollary \ref{c:euclmulti} described in Remark \ref{r:ruse2} in the setting $Z=\{0\}, \pi(dx)=\de_0$ and $M(\ph_i,\sB_n)= h^i(\cP|_{\sB_n})$. There are three issues that we need to deal with, namely
\begin{itemize}
\item[(a)] The add-one-cost operators of each coordinate, as well as their localized version, have uniform $p$-th moment bounds with $p>4,6$,  namely,  $\limsup_{n\to\infty} K'(n,p)<\infty$.
\item[(b)] Each coordinate has volume-order variance lower bounds, namely, there exists a universal constant $c$ such that $\Var[M(\ph_j ;\sB_n)]\ge c|\sB_n|$ for every $j\in [m]$. Further, the covariance matrix of $n^{-d/2}\bM(\sB_n)$ converges as $n\to\infty$.
\item[(c)] Determine the convergence rate of the quantities $\vartheta'(n), \tau'(n)$ and $\varrho'(n)$ towards 0 as $n\to\infty$. 
\end{itemize}
The result then follows immediately. We check (a) in Section \ref{s:moment_MST}, (b) in Section \ref{s:variance_MST}, and compute (c) in Section \ref{s:twoscale_MST}.  Our strategy of proof can be regarded as a non-trivial adaptation of the arguments used in \cite[Proof of Theorem 2.1]{CS17} to make them compatible with the use of add-one cost operators used in our bounds, and with the fact that we consider statistics that are more general than the graph length. For instance, the proof of Proposition \ref{p:2scale} below, which is one of our main statements, is obtained by adapting and expanding the arguments exploited in the proof of \cite[Proposition 10.1]{CS17}, that involve in particular a connection with two-arm events in Poisson-Boolean percolation. 

\smallskip

The next subsection puts forward a fundamental property of MSTs which is systematically applied in the proofs.

\subsubsection{Minimax property of the MST}

Let $\cal U$ be a finite subset of $\RR^d$ with distinct inter-point distances. It is known that two points $\mathbf x, \mathbf y\in \cal U$ form an edge in $\mathrm{MST}(\cal U)$ if and only if $\bf x$ and $\bf y$ are in two different connected components of the geometric graph of radius $|\mathbf x - \mathbf y|$, that is, a  graph with vertex set $\cal X$ and the edge set being the largest $E$ such that every $e\in E$ satisfies $|e|< |\mathbf x - \mathbf y|$.  One immediate consequence of this fact is that the MST paths are minimax in the following sense:  the path $\ga_0$ that connects $\bf x$ to $\bf y$ in MST minimizes the maximal weight (i.e. the maximal edge length) among all path $\ga$ that connects $\bf x$ to $\bf y$ in the complete graph of $\cal X$. This claim is clearly true when $\bf x$ and $\bf y$ are neighbors in MST by the aforementioned characterization of MST edges. More generally, the maximal edge $e'$ of $\ga_0$ is formed by some $\bf x', \bf y' \in \mathcal{U}$ with $\bf x'$ connecting to $\bf x$ and $\bf y'$ to $\bf y$ in the geometric graph of radius $|\bf x'-y'|$. If there is a path $\ga$ with maximal weight less than $|\bf x'-y'|$ that connects $\bf x$ to $\bf y$, then one can find a path connecting $\bf x'$ to $\bf y'$ with maximal weight less than $|\bf x'-y'|$, contradicting the fact that $e'$ is an edge of the MST.  Necessarily, $\ga_0$ is minimax. 

The minimax property is extremely important in the study of MST. We record here one consequence on the degree of the MST. 

\begin{lemma}[See \cite{AS92}]\label{l:AS}
The vertex degrees of an arbitrary Euclidean MST are uniformly bounded by a finite constant $D_\mathrm{max}$ depending only on the dimension $d$.
\end{lemma}

To see why this is true, we notice that, the minimax property prevents MST from forming a $k$-star centered at any $x\in\cal U$ with $k\ge c(d)$, where $c(d)$ is the minimal number of $60^\circ$ cones at the origin required to cover $\RR^d$. 

\subsubsection{Uniform moment bounds}\label{s:moment_MST}
In what follows, we write $\cP_n=\cP\cap \sB_n$. 
To establish the uniform moment bound for the add-one-cost, one needs to estimate the length of the  edges attached to a given deterministic vertex of the MST. Indeed, attaching a vertex $\bx\in\sB_n$ together with an edge $e$ to the MST$(\cP_n)$ creates a spanning tree of $\cP_n \cup\{\bx\}$, hence,
\begin{align*}
M^\bx(\ph;\sB_n)\le M(\ph;\sB_n) + \ph(d(\bx, \cP_n)),
\end{align*} 
where $d(\bx,\mathcal{U})$ is the Euclidean distance of $\bx$ to $\mathcal{U}$. On the other hand, removing from  MST$(\cP_n \cup\{\bx\})$ the vertex $\bx$ together with all the edges $e_1,...,e_\ell$ attached to $\bx$ creates a forest ($\ell$ disjoint trees) on $\cP_n$ so that one obtains a spanning tree of $\cP_n$ by connecting these disjoint trees with $\ell-1$ edges $e'_1,...,e'_{\ell-1}$. As a result, one has 
\begin{align*}
M(\ph;\sB_n) \le M^\bx(\ph;\sB_n)  +  \sum_{i=1}^{\ell-1} \ph(|e'_i|)
\end{align*}
with the convention that an empty sum is zero if $\ell=1$. Combining the two bounds, we arrive at the following lemma which is deterministic in nature,  see also \cite[Lemma 8.4]{CS17}.  Recall that $\ph(x)= \psi(x)\1(x\le r)$ for some $r\in(0,\infty]$ and increasing $\psi$.

\begin{lemma}\label{l:detbd}
For any $\bx\in\sB_n$, let $R=R(\cP_n,\bx)$ be the minimal positive number such that $\sB_R(\bx)$ {}{ (where $\sB_R(\bx)$ indicates the cube of side $R$ centered at ${\bf x}$)} contains all the edges attached to $\bx$ in $\mathrm{MST}(\cP_n\cup\{\bx\})$. Then we have almost surely
\begin{align*}
|D_\bx M(\ph;\sB_n)| \le \begin{cases}
(D_\mathrm{max}-1)\psi(\sqrt{d} r) & \mbox{ if } r<\infty, \\
 (D_\mathrm{max}-1)\psi(\sqrt{d} R) & \mbox{ if } r=\infty.
 \end{cases}
\end{align*}
\end{lemma}
\begin{proof}
By Lemma \ref{l:AS}, $\ell\le D_\mathrm{max}$, the claim follows immediately by noticing that  $|e'_i|\le\sqrt{d}R$, $d(\bx,\mathcal{U})\le \frac{\sqrt{d}R}{2}$ and the monotonicity of $\psi$. 
\end{proof}

From now on we focus on  the case $r=\infty$ since otherwise the uniform moment bound is trivial. We describe the tail event $\{R>u\}$ in terms of the wall event introduced in \cite{CS17}. Let $\mathcal{U}$ be a finite subset of $\RR^d$. We say that \emph{$\bx$ is surrounded by a $\mathcal{U}$-wall in $\sB_n$ at scale $u<2n$} if for any $\bx'\in \sB_n\cap \partial \sB_u(\bx)$, the restriction to $\sB_n$  of the lens-shaped intersection $\sS_{\frac{3}{4}|\bx-\bx'|}(\bx)\cap \sS_{\frac{3}{4}|\bx-\bx'|}(\bx')$ contains at least one element of $\mathcal{U}$, where $\sS_a(\bx)$ is a ball centered at $\bx$ with radius $a$.  Such an event is denoted by $\scr A(\mathcal{U},\bx,u)$. The key observation is the following inclusion relation, see also \cite[Lemma 8.2]{CS17}. 

\begin{lemma}\label{l:AsubR}
We have $\scr A(\mathcal{U},\bx,u)\subset \{ R(\mathcal{U},\bx)\le u \}$.
\end{lemma}
\begin{proof}
Suppose that $\scr A(\mathcal{U},\bx,u)$ occurs. We claim that no $\by\in \mathcal{U}\cap \sB_n\cap \sB_u(\bx)^c$ can reach $\bx$ through only one edge $\{\bx,\by\}$ in the MST$(\mathcal{U}\cup\{\bx\})$, the conclusion then follows by the definition of $R$. Now we prove the claim. Let $\bx'$ be the unique point on the segment $\bx\by$ that lies on $\partial\sB_u(\bx)$ and denote by $\bz$ an arbitrary element of $\mathcal{U}$ belonging to the aforementioned lens-shaped intersection. The existence of such an element is guaranteed by the event $\scr A(\mathcal{U},\bx,u)$. It is clear that $|\bx-\by|\ge |\bx-\bx'|>\frac{3}{4}|\bx-\bx'| \ge |\bx-\bz|$. On the other hand,
\begin{align*}
|\bx-\by|= |\bx-\bx'| + |\bx'-\by|> \frac{3}{4}  |\bx-\bx'| + |\bx'-\by| \ge |\bz-\bx'| + |\bx'-\by|\ge |\by-\bz|.
\end{align*}
In either case, $|\bx-\by|>\max(|\bx-\bz|, |\by-\bz|)$. By the minimax property, the segment $\bx\by$ cannot be an edge of MST$(\mathcal{U}\cup\{\bx\})$, as desired.  
\end{proof}

The next lemma quantifies the idea that a wall at distance $u$ is very likely to exist  as $u$ grows, see also .

\begin{lemma}\label{l:walltail}
There exist $c_1=c_1(d), c_2=c_2(d)$ such that for all $n\in\NN, \bx,\by \in\sB_n, 0<u<2n$, 
\begin{align*}
\PP( \scr A(\cP_n \cup\{\by\},\bx,u)^c) \le \PP( \scr A(\cP_n,\bx,u)^c) \le c_1 e^{-c_2 u^d}.
\end{align*} 
\end{lemma}
\begin{proof}
Without loss of generality we assume that $u\in\NN$. The first inequality follows from the fact that $\scr A(\mathcal{U},\bx,u)$ is increasing with respect to $\mathcal{U}$. The second estimate is proved in \cite[Lemma 8.3]{CS17}.

\end{proof}

%
%

We are ready to check the uniform moment bounds -- see also \cite[Lemma 8.6]{CS17}.

\begin{proposition} \label{p:moment}
For any $q>0$, there exits a finite positive constant $C_q$ such that uniformly for all $\bx, \by\in\sB_n$, $0<r<n$ and  $n\in\NN$, one has
\begin{align*}
\max( \EE[|D_\bx M(\ph;\sB_n)|^q], \EE[|D_\bx M^\by(\ph;\sB_n)|^q])\le C_q. 
\end{align*}
and
\begin{align*}
\max( \EE[|D_\bx M(\ph;\sB_n\cap\sB_r(\bx))|^q], \EE[|D_\bx M^\by(\ph;\sB_n\cap\sB_r(\bx))|^q])\le C_q. 
\end{align*}
\end{proposition}
\begin{proof}
The two inequalities have the same content, where the role of $\sB_n$ in the first inequality is replaced by $\sB_r(\bx)\cap\sB_n$ in the second, so we only consider the first.  Applying Lemmas \ref{l:detbd}, \ref{l:AsubR}, \ref{l:walltail} as well as the monotonicity of $\psi$ implies that for all $0<u<2n$,
\begin{align*}
 \PP[|D_\bx M(\ph;\sB_n)| > (D_\mathrm{max}-1)\psi(\sqrt{d} u)]\le \PP[R(\cP_n,\bx)>u]\le \PP( \scr A(\cP_n,\bx,u)^c) \le c_1 e^{-c_2 u^d},
\end{align*}
and this probability is 0 if $u>2n$ by the definition of $R$.  By Fubini's Theorem, one has
\begin{align*}
\EE[|D_\bx M(\ph; \sB_n)|^q] &= \int_0^\infty q v^{q-1} \PP[ |D_\bx M(\ph; \sB_n)|>v ]   dv.
\end{align*}
A change of variable $v=(D_\mathrm{max}-1)\psi(\sqrt{d}u)$ together with the tail bound gives
\begin{align*}
\EE[|D_\bx M(\ph;\sB_n)|^q] \le q(D_\mathrm{max}-1)^q  c_1\int_0^\infty e^{-c_2u^d} d\psi(\sqrt{d}u),
\end{align*}
yielding the moment bound for $|D_\bx M(\ph;\sB_n)|$ by the growth condition on $\psi$. The moment bound for $|D_\bx M^\by(\ph;\sB_n)|$ is straightforward in view of the first part of Lemma \ref{l:walltail}.
\end{proof}

\begin{remark}
One may also use Kesten and Lee's separation set argument \cite{KL96} to show the moment bound, which is in the same spirit as the wall event argument. 
\end{remark}

\subsubsection{Variance lower bounds}\label{s:variance_MST}

The claim about the covariance was checked in \cite{Penrose05}. 
The volume-order variance lower bound for $M(\ph;\sB_n)$ with $\ph(x)=x$,  among other things, was proved by Kesten and Lee \cite{KL96}. Here we use a general result of Penrose and Yukich \cite[Theorem 2.1]{PY01} to show the nontriviality of the limiting variances of each coordinate. 


\begin{proposition}\label{p:varlow}
There exists a positive constant $c>0$ such that for all $j\in[m]$ and $n$ large, 
\begin{align*}
\Var[M(\ph_j;\sB_n)] \ge c |\sB_n|.
\end{align*}
\end{proposition}

\begin{proof}
In order to apply \cite[Theorem 2.1]{PY01}, we need to prove the following results:  
\begin{itemize}
\item[(i)] A uniform moment condition for the add-one-cost over binomial point process with uniform points in $\sB_n$ for all $n\in\NN$.  
\item[(ii)] A deterministic growth condition: there exist finite positive $\be_1,\be_2$ such that $|M(\ph_j;\mathcal{U})|\le \be_2(\mathrm{diam}(\mathcal{U})+ |\mathcal{U}|)^{\be_1}$. 
\item[(iii)] The add-one-cost strong stabilization condition. 
\item[(iv)] $D(\b0,\infty)$ is non-degenerate. 
\end{itemize} 
Items (i)-(iii) are checked in \cite[Theorem 3.3]{Penrose05}. As for (iv), we construct two events with positive probability under which $D(\b0,\infty)$ differs by a non-trivial quantity. To this end, we borrow a construction from Kesten and Lee \cite{KL96}. Let $r=  \min_{j\in [m]} r_j /13$ if the right-hand side is finite, otherwise $r=1$. Decompose $\RR^d$ into cubes of side length $r$ centering at $r\ZZ^d$.  We call the annulus $\sB_{12r}(\b0)\setminus\sB_{r}(\b0)$ the {\it moat}. An {\it island} is an arbitrary $r$-cube sitting in the middle of the interior of the moat.
Denote by  $\scr E_0$ the event that each of the $r$-cubes in the annulus $\sB_{13r}(\b0)\setminus \sB_{12r}(\b0)$ contains at least one point in $\cP$, the island contains exactly one point of $\cP$ and no points in other cubes of the moat.  Let $\scr E_1$ be the intersection of $\scr E_0$ with the event that $\sB_r(\b0)\cap\cP=1$ and $\scr E_2$ be the intersection of $\scr E_0$ with the event that $\sB_r(\b0)\cap\cP=0$.  It is clear that both $\scr E_1$ and $\scr E_2$ occur with positive probability. Let us call the stabilizing radius $S$. By the minimax property, one sees that MST$(\cP \cap \sB_S(\b0))$ under $\scr E_1$ is identical to that under $\scr E_2$ except that the former contains exactly one more edge than the latter which connects the point in the island and  the point in $\sB_r(\b0)$. Therefore, under $\scr E_2$,  adding $\b0$ to $\cP$ increase the total weighted edge length in $\sB_S(\b0)$ by at least $\psi(2r)$. On the other hand, under $\scr E_1$, adding $\b0$ to $\cP$ creates one more edge connecting $\b0$ and the point in $\sB_r(\b0)$ and no other effect, increasing the total weighted edge length by at most $\psi(r)$. Thus, $D(\b0,\infty)$ under $\scr E_1$  and $\scr E_2$ differs by at least $\psi(2r)-\psi(r)>0$, ending the proof. 
\end{proof}

\subsubsection{Two-scale stabilisation and Proof of Theorem \ref{t:mst}}\label{s:twoscale_MST}
All three quantities $\varrho'(n), \tau'(n), \vartheta'(n)$ involve the $L^1$ difference between add-one cost operators at two different scales, that we call two-scale discrepancy. As one can see from the proof, adding \emph{one deterministic point}  to $\cP$ that is  not too close to $\bx$ does not change the decaying rate of the two-scale discrepancy at $\bx$. Hence, we obtain the same rate for $\tau'(n), \varrho'(n), \vartheta'(n)$, see Proposition \ref{p:2scale}.  In this section, the scale of $\sA_{\bx,n}$ is chosen to be $b_n=n^\al$ for any $\al\in (0,1)$ and we write for simplicity $\sA_{\bx}=\sA_{\bx,n}$.

We start with a review of the add-and-delete algorithm for constructing the MST, devised by Kesten and Lee \cite{KL96}.   There are several algorithms for building the MST on a weighted graph. Kesten and Lee's algorithm is not the most efficient but it has the advantage that if a MST has already been constructed on a certain weighted graph and one attaches a new edge (or a new vertex together with an edge) to the underlying graph, the one does not have to rebuild the MST from scratch:  one only needs to "add" and "delete" edges from the existing tree to produce the new MST of the new weighted graph.   

The algorithm goes as follows. Suppose that $\mathrm{MST}(\mathcal{U})$ of the complete graph of $\mathcal{U}$, weighted by the Euclidean distance, has been constructed. Our goal is to construct $\mathrm{MST}(\mathcal{U}\cup\{\bx\})$. There is one new vertex $\bx$, and $|\mathcal{U}|$ new edges $\{e_1,...,e_{|\mathcal{U}|}\}$, each connecting $\bx$ and one point in $\mathcal{U}$, attached to the complete graph of $\mathcal{U}$, producing the complete graph of $\mathcal{U}\cup\{\bx\}$. Instead of growing the underlying graph at once, we do so gradually by first attaching $\bx$ together with a new edge, say $e_1$, to the complete graph of $\mathcal{U}$, then by attaching one by one the new edges $e_2, e_3,..., e_{|\mathcal{U}|}$.  Clearly the first step produces a (not necessarily minimal) spanning tree of $\mathcal{U}\cup\{\bx\}$, and the later steps create each time a cycle.  The algorithm gives instructions on what to do in each step to turn these intermediate graphs into a MST.  When the algorithm finishes, we obtain $\mathrm{MST}(\mathcal{U}\cup\{\bx\})$.

\begin{itemize}
\item Step 1. Attach $\bx$ together with $e_1$ to $\mathrm{MST}(\mathcal{U})$, call the obtained spanning tree $T_1$.  
\item Step 2. For $i$ from $2$ to $|\mathcal{U}|$, 
\begin{itemize}
\item Attach edge $e_i$ to $T_{i-1}$, inducing necessarily a cycle $C_{i-1}$.
\item Remove the longest edge $f_{i-1}$ in $C_{i-1}$ to create a spanning tree $T_i$.
\end{itemize}
\end{itemize}  

\begin{lemma}\label{l:algo} The add-and-delete algorithm produces the MST of the intermediate graphs in each step of its iterations, in particular, $T_{|\mathcal{U}|}=\mathrm{MST}(\mathcal{U}\cup\{\bx\})$. 
\end{lemma}

The proof of the lemma uses the edge characterization of MST, we refer to Kesten an Lee \cite{KL96} for more details. We move to bounding the add-one-cost discrepancy. The idea, which is close to the strategy of proof of \cite[Proposition 10.1]{CS17}, is to run the add-and-delete algorithm simultaneously for $\mathrm{MST}(\cP_n)$ and $\mathrm{MST}(\cP_n\cap {\sA_{\bx}} )$ after adding $\bx$ and $|\cP_n|$ new edges,  or $\bx$ and $|\cP_n\cap {\sA_{\bx}}|$ edges, to the underlying graphs, respectively. The algorithm says that the add-one-cost is the sum of edge weights with alternating signs.  
Comparing each weight change in each intermediate step, then taking the sum over the difference between the weight change in the smaller scale $\sA_{\bx}$ and that in the larger scale $\sB_n$, gives an upper bound for the two-scale discrepancy. We then show that this upper bound goes to zero as $n$ grows to infinity.  

In order to simplify the notation, we write $\ph=\ph_j$ and $s=r_j$ for $\ph_j(x) = \psi_j(x)\1(x<r_j)$. 

\begin{proposition}\label{p:2scale} There exist $\be', c$ positive finite constants such that for all $n\in\NN$ and $\bx\in \sB_n$ with $d(\bx,\partial \sB_n)>n^\al$, 
\begin{align*}
\EE[ |D_\bx M(\ph; \sB_n\cap \sA_{\bx})-D_\bx M(\ph;\sB_n)|] \le  \begin{cases}
c n^{-\beta'} & \mbox{ if } d=2, \\
c \exp( - c \log\log(n)) & \mbox{ if } d\ge 3.
\end{cases}
\end{align*}
The same bound holds with $D_\bx M$ replaced by $D_\bx M^\by$ for all $\by\in \sB_n$ with $\|\by-\bx\|\ge n^\al$. Consequently, there exists finite positive constant $c'$ such that for any $p>6$,
\begin{align*}
\tau'(n), \varrho'(n), \vartheta'(n) \le c' n^{(\al-1)} + \begin{cases}
c' n^{-\beta' (1-\frac 6 p)} & \mbox{ if } d=2, \\
c' \exp( - c (1-\frac 6 p)\log\log(n)) & \mbox{ if } d\ge 3.
\end{cases}
\end{align*}
\end{proposition}

Applying this Proposition and Remark \ref{r:ruse2} proves immediately Theorem \ref{t:mst}.  Let us now end this section with the proof of Proposition \ref{p:2scale}.

\begin{proof}
For clarity of the presentation, we divide the proof into several steps. 

\medskip

\noindent\underline{\it Step 1:  Bound when no wall exists.} Let $a_n=o(n^\al)$ whose value will be chosen later.  By the Cauchy-Schwarz inequality and the uniform moment bound in Proposition \ref{p:moment}, there exists a constant $0<C<\infty$ such that
\begin{align*}
\EE[ |D_\bx M(\ph;\sB_n\cap \sA_{\bx})-D_\bx M(\ph;\sB_n)|\1(\scr A(\cP_n,\bx, a_n)^c)] \le C \PP[\scr A(\cP_n,\bx, a_n)^c]^{1/2} \le c_1 e^{-c_2 a_n^d}. 
\end{align*}

\noindent\underline{\it Step 2: Reduction to edges inside $\sB_{a_n}(\bx)$ when a wall exists.} Under the event $\scr A(\cP_n,\bx, a_n)$, no $\by$ outside of $\sB_{a_n}(\bx)$ can be connected to $\bx$ by one single edge in $\mathrm{MST}(\cP_n\cup\{\bx\})$ or $\mathrm{MST}((\cP_n \cup\{\bx\})\cap \sA_\bx)$. Thus, to produce these new MSTs after adding $\bx$, one can run the add-and-delete algorithm until one finishes adding all the new edges inside $\sB_{a_n}(\bx)$. 

\noindent\underline{\it Step 3:  Preliminary bound when a wall exists.} Suppose now that $\scr A(\cP_n,\bx, a_n)$ occurs, thanks to Step 2, it suffices to consider the edges formed by connecting $\bx$ with vertices in $$\{\by_1,...,\by_{|\cP_n\cap \sB_{a_n}(\bx)|}\}=\cP_n\cap \sB_{a_n}(\bx).$$ By Lemma \ref{l:algo}, we have 
\begin{align*}
D_\bx M(\ph;\sB_n) = \psi_j(|e_1|)\1(|e_1|\le s) + \sum_{i=2}^{|\cP_n\cap \sB_{a_n}(\bx)|} \psi_j(|e_i|)\1(|e_i|\le s) - \psi_j(|f_{i-1}|)\1(|f_{i-1}|\le s), 
\end{align*}
where $e_i$ is the edge formed by $\bx,\by_i$, and $f_i$'s are the removed edges in the second step of the add-and-delete algorithm. To make the presence of the indicators clear, we recall that an edge of the MST is counted for the functional $M(\ph;\sB_n)$ if its length is no larger than $s$.  On the other hand, denoting the removed edges by $\tilde f_i$ when we run the algorithm for $\mathrm{MST}(\cP_n\cap\sA_{\bx})$ to obtain $\mathrm{MST}((\cP_n\cup\{\bx\})\cap \sA_{\bx}  )$, we have
\begin{align*}
D_\bx M_s(\ph;\sB_n\cap \sA_{\bx}) = \psi_j(|e_1|)\1(|e_1|\le s) + \sum_{i=2}^{|\cP_n\cap \sB_{a_n}(\bx)|} \psi_j(|e_i|)\1(|e_i|\le s) - \psi_j(|\tilde f_{i-1}|)\1(|\tilde f_{i-1}|\le s). 
\end{align*}
Therefore, combining these two identities, we are led to the upper bound
\begin{multline}\label{e:summand}
|D_\bx M_s(\ph;\sB_n\cap \sA_{\bx}) - D_\bx M_s(\ph;\sB_n)  | \1(\scr A(\cP_n,\bx, a_n))\\
\le \sum_{i=2}^{|\cP_n\cap \sB_{a_n}(\bx)|} |\psi_j(|\tilde f_{i-1}|)\1(|\tilde f_{i-1}|\le s) - \psi_j(|f_{i-1}|)\1(|f_{i-1}|\le s)|. 
\end{multline}

\noindent\underline{\it Step 4: Bounding the summand in \eqref{e:summand}} Notice the following important fact: the added and removed edges while running the algorithm are all no longer  than $\sqrt{d} a_n$. This is clearly true for the added ones, because $\by_i\in\sB_{a_n}(\bx)$. To see that it also holds for the deleted ones, denote by $\bf z$ the only vertex adjacent to $\bx$ other than $\by_i$ in the cycle generated after adding the edge $e_i=\{\bx, \by_i\}$. Notice that $\bz\in\sB_{a_n}(\bx)$ under the wall event. Two cases may arise,  if $f_i = \{\by_i, \mathbf z\}$, then $|f_i|= |\mathbf{z}-\by_i|\le \sqrt{d} a_n$. Otherwise $\bf z$ and $\bf y_i$ do not form an edge in $\mathrm{MST}(\cP_n)$, nor in $\mathrm{MST}(\cP_n\cap \sA_\bx)$, because there is only one loop $C_{i-1}$ after adding the edge $e_i$. Consequently, by the minimax property of the intermediate MSTs, the path (as subset of the loop $C_{i-1}$)  that starts from $\by_i$ and ends at $\bf z$ via the removed edge $f_{i-1}$ has maximal length less than $|\mathbf{z}-\by_i|\le \sqrt{d} a_n$. 

Another important fact is that for each $i$,  we have $|f_i|\le |\tilde f_i|$ and the equality holds when $f_i = \tilde f_i$. Indeed, there are more potential paths to choose from for constructing $\mathrm{MST}(\cP_n)$ with minimax path property, so that one always has $|f_i|\le |\tilde f_i|$. By the fact that all the edges have distinct length almost surely, the second claim follows. Moreover, if $|f_i|< |\tilde f_i|$, then necessarily $f_i\not\subset \sA_\bx$, since otherwise, one can find a path connecting the endpoints of $\tilde f_i$ and having maximal weight smaller than $\tilde f_i$, which is a contradiction. 

As a consequence, the summand in \eqref{e:summand} satisfies
\begin{align*}
&|\psi_j(|\tilde f_{i-1}|)\1(|\tilde f_{i-1}|\le s) - \psi_j(|f_{i-1}|)\1(|f_{i-1}|\le s)| \\
&= (|\psi_j(\tilde f_{i-1}|) - \psi_j(|f_{i-1}|)) \1( |\tilde f_{i-1}|\le s ) + \psi_j(|f_{i-1}|) \1( |f_{i-1}|\le s < |\tilde f_{i-1}|) \\
&\le \int_0^{\sqrt{d} a_n} \psi_j'(u)\1( |f_{i-1}| <u<|\tilde f_{i-1}|) du + \psi_j(\sqrt{d}a_n) \1( |f_{i-1}|\le s < |\tilde f_{i-1}|),
\end{align*}
where we bounded the indicator $\1( |\tilde f_{i-1}|\le s )$ by $1$ and $|f_{i-1}|$ by $\sqrt{d} a_n$ in the last inequality. 

\noindent\underline{\it Step 5: Link to the two-arm event for the Poisson-Boolean model}. Understanding the event $$\{|f_{i-1}| <u<|\tilde f_{i-1}|\}$$ is the key to proceed.  One fundamental observation from \cite{CS17} is that such an event implies a sort of two-arm event for the Poisson-Boolean model that we recall now. The Poisson-Boolean model is the random set given by
\begin{align*}
\sO_u(\cP) = \bigcup_{\bf x\in\cP} \sS_u(\bx),
\end{align*}
where $u>0$ is the parameter of the model, and the notation $\sO$ stands for occupation. 
Suppose now  the event $\{|f_{i-1}| <u<|\tilde f_{i-1}|\}$ occurs.

%

Firstly, as written in Step 4, we have $f_{i-1}\not\subset \sA_\bx$. Notice that the circuit $C_{i-1}$ must contain $\by_a$ for some $1\le a<i$. Since  $|f_{i-1}|<u$, we see that $\sO_{u/2}(\cP_n)$ contains the edge $f_{i-1}$, and the connected component of $f_{i-1}$ in $\sO_{u/2}(\cP_n)$ contains the  whole path $C_{i-1}\setminus e_i\setminus e_a$ with endpoints $\by_i$ and $\by_a$, by the fact that $f_{i-1}$ is the longest edge in $C_{i-1}$. Therefore, the points $\by_i$ and $\by_{a}$ are in the same connected component of $\sO_{u/2}(\cP_n)$, denoted by $\mathsf K(i,\cP_n)$. 

On the other hand,  we claim that $\mathsf K(i,\cP_n)\cap (\sA_\bx^{-} \setminus \sB_{(1+2\sqrt{d})a_n}(\bx))$ has (at least) two disjoint connected components, each of which intersects both $\partial A_\bx^-$ and $\partial\sB_{(1+2\sqrt{d})a_n}(\bx)$, where $\sA_\bx^-$ is $\sA_\bx$ with radius shrinked by $\sqrt{d}a_n$, namely $$\sA_\bx^- = \sB_{n^\al-2\sqrt{d}a_n}(\bx).$$ To see this, we decompose the vertices in the path $C_{i-1}\setminus e_i\setminus e_a$ into three non-overlapping groups. The first group $\mathcal G_1$ contains those vertices in $\sA_\bx\setminus\sB_{a_n}(\bx)$ that one passes by as one travels starting from $\by_i$ along the path until exiting $\sA_\bx$, the second $\mathcal G_2$ contains  those vertices that one passes by as one travels from $\by_a$  upon exiting $\sA_\bx$, the rest are put into the (useless) third group.  The first two groups will be the skeleton to construct two disjoint connected components. 
By the choice of the size of $\sA_\bx^-$, only the configuration of $\cP_n$ inside $\sA_\bx\setminus\sB_{a_n}(\bx)$ matters to decide whether the desired claim holds. Hence, the claim follows if one can show that there exists no $\mathcal H\subset \cP_n\cap (\sA_\bx\setminus\sB_{a_n}(\bx))$ such that the Boolean set 
\begin{align*}
\sO_{u\over 2}(\mathcal G_1\cup \mathcal G_2 \cup \cH )
\end{align*}
is connected.  Suppose that there exists such $\cH$ that $\sO_{u/2}(\mathcal G_1\cup \mathcal G_2 \cup \cH )$ is connected. Then one can find a path $\ga$ connecting $\by_i$ and $\by_a$ through $\cH$ and part of $\mathcal G_1, \mathcal G_2$ inside $\sA_x$ with maximal edge length less than $u$ by previous reasoning about the maximal edge length of $C_{i-1}$ and the assumption on $\cH$. To arrive at a contradiction, we use the condition $\{ u<|\tilde f_{i-1}| \}$. By construction, all the edges in the path $\tilde\ga=\tilde C_{i-1}\setminus \tilde f_{i-1}$ are shorter than $|\tilde f_{i-1}|$. By concatenating $\ga, e_a, \tilde\ga\setminus e_i$, one finds a path that connects the endpoints of $\tilde f_{i-1}$ with the property that all egdes are shorter than $|\tilde f_{i-1}|$. Notice that only the edges in $K(\cP_n\cap \sA_x)\cup\{e_1,...,e_{i-1}\}$ are used in constructing the concatenated path, where $K(\cP_n\cap \sA_\bx)$ is the complete graph of $\cP_n\cap \sA_\bx$. Hence, by the minimax property of the MST paths, the existence of the aforementioned concatenated path contradicts the fact that $\tilde f_{i-1}$ is an edge of the intermediate minimal spanning tree $\tilde T_{i-1}$ when one runs the add-and-delete algorithm for $\mathrm{MST}(\cP_n\cap \sA_x)$.

To summarize, define the two-arm event
\begin{align*}
\scr D_n(u)  = \{ \partial\sB_{(1+2\sqrt{d})a_n}(\bx) \con{2} \partial \sA_\bx^- \mbox{ in } \sO_{u\over 2}(\cP_n) \}
\end{align*} 
where $\partial\mathsf{E_1}\con{k} \partial\mathsf{E_2}$ in $\sf F$ with $\mathsf{E}_1\subset \mathsf{E}_2$ means that $\mathsf F\cap (\mathsf{E}_2\setminus \mathsf E_1)$ has (at least) $k$ disjoint connected components $\mathsf{K}_1, ..., \mathsf{K}_k$,  such that for each $i\in[k]$, we have $\mathsf{K}_i\cap \partial\mathsf{E_1}\neq\emptyset$ and  $\mathsf{K}_i\cap \partial\mathsf{E}_2\neq\emptyset$. Since $\mathsf K(i,\cP_n)$ is a connected component of $\sO_{u/2}(\cP_n)$, we have proved the crucial inclusion  $$\{|f_{i-1}| <u<|\tilde f_{i-1}|\}\subset \scr D_n(u).$$   
We stress that the radius $(1+2\sqrt{d})a_n$ in place of $a_n$ in the definition of $\scr D_n(u)$ is chosen in such a way that $\scr D_n(u)$ and $\cP_n\cap \sB_{a_n}(\bx)$ are independent for all $u\le \sqrt{d}a_n$. 

\noindent\underline{\it Step 6:  Final bound when the wall event occurs.} Combining Steps 3-5 leads to 
\begin{align*}
&\EE[ |D_\bx M_s(\ph;\sB_n\cap \sA_{\bx}) - D_\bx M_s(\ph;\sB_n)  | \1(\scr A(\cP_n,\bx, a_n))] \\
&\le \int_0^{\sqrt{d}a_n} \psi_j'(u) \EE[ |\cP_n\cap \sB_{a_n}(\bx)|\1(\scr D_n(u)) ] du + \psi_j(\sqrt{d}a_n)\EE[ |\cP_n\cap \sS_{a_n}(\bx)|\1(\scr D_n(s))] \\
&\le c a_n^d \int_0^{a_n} \psi'(u) \PP(\scr D_n(u)) du + ca_n^{d} \psi_j(\sqrt{d}a_n) \PP(\scr D_n(s))
\end{align*}
where $c$ depends only on $d$. To proceed, we make use of a highly non-trivial bound due to Chatterjee and Sen \cite[Section 5]{CS17}. Set
\begin{align*}
a_n =\begin{cases}
\al\log(n) & \mbox{ if } d=2, \\
(\log\log(n))^{1/(d-1/2)} & \mbox{ if } d\ge 3,
\end{cases}
\end{align*}
then it holds that
\begin{align*}
\PP(\scr D_n(u)) \le 
\begin{cases}
c n^{-\be} & \mbox{ if } d=2,  \\
\frac{c \exp(c a_n^{d-1})}{\log(n)^{d/2}} & \mbox{ if } d\ge 3. 
\end{cases}
\end{align*}
where $\beta, c$ are all positive finite constants that depends only on $d$, whose value may change in each appearance. The striking feature of this bound is that it is true uniformly in the parameter of the Boolean model $u>0$. We arrive at the following bound
\begin{align*}
&\EE[ |D_\bx M_s(\ph;\sB_n\cap \sA_{\bx}) - D_\bx M_s(\ph;\sB_n)  | \1(\scr A(\cP_n,\bx, a_n))] \\ &\le \begin{cases}
c \log(n)^2 \psi_j(\sqrt{2}\log(n)) n^{-\beta} & \mbox{ if } d=2, \\
c \psi_j(\sqrt{d}\log\log(n)^{1/(d-1/2)})\exp( - c \log\log(n)) & \mbox{ if } d\ge 3.
\end{cases}
\end{align*}

\noindent\underline{\it Step 7: Conclusion} The growth condition of $\psi_j$ and the choice of $a_n$ in Step 6, as well as the bound in Step 1, yields
\begin{align*}
\EE[ |D_\bx M_s(\ph;\sB_n\cap \sA_{\bx}) - D_\bx M_s(\ph;\sB_n)  | ] \le  \begin{cases}
c n^{-\beta'} & \mbox{ if } d=2, \\
c \exp( - c \log\log(n)) & \mbox{ if } d\ge 3.
\end{cases}
\end{align*}
where $\be', c$ are universal constants.  Handling as usual the points near the boundary of $\sB_n$ with uniform moment bounds for the add-one-cost operators and H\"older's inequality, we obtain the upper  bound for $\tau'(n), \vartheta'(n), \varrho'(n)$ in view of  Remark \ref{r:ruse2}.
\end{proof}

\subsection{Number of connected components}
Fix $r>0$. For a finite set $\mathcal{U} \subset \mathbb{R}^{d} $, denote by $\mathscr  K_r(\mathcal{U} )$ the number of connected components of the graph $\mathscr  G_r(\mathcal{U} )$ with vertex set $\mathcal{U} $ and an edge between each pair of points of $\mathcal{U} $ at distance $r$ or less. In this subsection, we are interested in the functional 
\begin{align*}
F(\sB)=\mathscr  K_r(\cP\cap \sB)
\end{align*} for $\sB\subset \mathbb{R}^{d}$ bounded.
It has been shown in \cite[Theorem 13.27]{PenroseBook} that if $\sC_{n}$ is a cube with volume $n^{d}$, $F(\sC_{n})$ satisfies a central limit theorem, with $\Var(F(\sC_{n}))\geqslant \sigma n^{d}$ for some $\sigma >0$ (the proof is actually given for the rescaled version $\tilde F(\cdot )=F(n^{-1/d}\cdot )$, with $r_{n} = r n^{-1/d}$). This functional satisfies weak stabilisation in some sense, and we aim here to quantify this convergence with the tools we have developed.  As for the Minimal Spanning Tree, of central importance here is the probability of the two-arm event $\mathscr  A _{r}^{N}(\cP )$, where for any random point configuration $\mathcal{U },a\in \mathbb{R}$, $\mathscr  A_ {a}^{N}(\mathcal{U})$ is the event that at least two distinct connected components of $\mathscr  G_r(\mathcal{U}\cap \sC_{N}(0)\setminus \sC_{a}(0) )$ come within distance $r$ of both $\sC_{a}$ and $\sC_{N}^{c}$, for $N\geq a>0$. The context is more favorable here as the radius $r$ is fixed, hence this probability decays exponentially for every non-critical value of $r$, whereas for the MST the radius was somehow random and context-dependent.

Let $r_{c}(d)$ be the critical radius for continuum percolation in $\mathbb{R}^{d}$ driven by a Poisson process with intensity one.
According to \cite[Lemma 9.5]{CS17}, for $1/2<a\leqslant 2r, N$ sufficiently large,
\begin{align}
\label{eq:two-arm}
\mathbb{P}(\mathscr  A _{a}^{N}(\cP ) )\leq  \begin{cases}C\exp(-cN )$ if $r\neq r_{c}(d)\\
C\log(N)^{-d/2}$ if $r=r_{c}(d),d\geq 3\\
CN^{-\gamma  }$ if $r=r_{c}(d),d=2
\end{cases}
\end{align}
for some  positive constants $\gamma ,C,c.$

\begin{theorem}
\label{thm:components}For all $0<\alpha <1,\varepsilon >0$,  {there exists a constant $C\in (0,\infty)$ such that}
\begin{align}
\label{eq:rate-comp}
\dk(\wh F(\sC_n ) ,N(0,1))\leq C\left(
 {\mathbb{P}(\mathscr  A {r,n^{\alpha }})^{\frac{1}{2}-\varepsilon }}+ {n^{(\alpha-1) /2}}
\right), \quad  {n\geq 1}.
\end{align}

\end{theorem}

Hence the decay is polynomial if the radius is not critical or if the dimension is $2.$

\begin{proof}
We apply Corollary \ref{c:dw} modified along the lines of Remark \ref{r:ruse}, without the marks space, i.e. with $Z$ as a singleton, hence all mark-related notation is omitted. For the moment conditions, we must control by how much  the number of components is modified upon the removal of a ball. Luckily, this is easy as one can   show that a ball of $\mathbb{R}^{d}$ touches at most $\kappa _{d}$ distinct connected components of union of balls with the same radius, for some $\kappa _{d}\geq 2$. Hence the removal of this ball within any countable union of balls  cannot decrease the number of connected components by more than $1$ and it can increase by at most $\kappa _{d}-1$ (1 component that is split in at most $\kappa _{d}$ components),  and for any $\chi\subset \mathbb{R}^{d} ,\sB\subset \mathbb{R}^{d},\bx\in \mathbb{R}^{d}$,
\begin{align}
\label{eq:Dx-comp}
|\mathscr K_{r}((\chi \cup \{x\}) \cap \sB)-\mathscr  K_{r}(\chi \cap \sB)|\leq  \kappa _{d}-1 ,
\end{align}
hence $K'<\infty $ in \eqref{e:m_con2}.

Let the cube $\sA_{\bx}=\sC_{N}(\bx)$ where  $N=N_{n}= n^{\alpha }>(1+8\sqrt{d})r$ and $0<\alpha <1$ for $n$ sufficiently large. Let $X=\cP\cap \sC_{N}(\bx)$.
Denote by $\mathscr  A _{r}^{N}(\bx)$ the two-arm event with center $\bx$, i.e. the two-arm event for the translated configuration $\tau _{-\bx}\cP$. Let $\bx$ be such that $\sA_{x}\subset \sC_{n}$. If  $\mathscr  A _{r}^{N}(\bx) $ is not realised, $X$ is $(\bx,r,N)$-stable in the terminology of Penrose \cite[Section 13.7]{PenroseBook}, which implies that for $Y=(\cP\cap \sC_{n}\setminus \sA_{\bx})\subset \sA_{\bx}^{c}$ and $W=\cP\cap \sC_{r}(\bx)\cup \{\bx\}\subset \sC_{r}(\bx)$, we have with $X':=(X\setminus \sC_{r}(\bx))\cup  W=X\cup \{\bx\},$ 
\begin{align*}
\mathscr  K_r(X\cup Y)-\mathscr  K_r(X'\cup Y))&=\mathscr  K_r(X)-\mathscr  K_r(X')\\
\mathscr  K_r(\cP\cap \sC_{n})-\mathscr  K_r((\cP\cup \{\bx\})\cap \sC_{n})&=\mathscr  K_r(\cP\cap \sC_{n}\cap \sA_{\bx}))-\mathscr  K_r((\cP\cup \{\bx\})\cap \sC_{n}\cap \sA_{\bx} )\\
D_{\bx}F(\sC_{n})&=D_{\bx}F(\sC_{n}\cap \sA_{\bx}).
\end{align*}
Hence, using also \eqref{eq:Dx-comp}, $\psi'(n)$ as defined in Remark \ref{r:ruse} satisfies for $p>4$, for some $C>0,$
\begin{align*}
 { \psi'(n)}\leq C\mathbb{P}(\mathscr  A _{r}^{N})^{ (1-\frac{4}{p})(1-\frac{1}{p})}.
\end{align*}   

Dealing with the addition of some $\by\in \mathbb{R}^{d}$ to treat $\phi'(n)$ requires a bit more care. Let $\mathbf{N}_{\by}$ be the class of finite subsets of $\sB_{r}({\by})$, and $\mathbf{N}_{\by}'$ the class of $\mathcal  U \in  \mathbf{N}_\by$ such that $\sB_{r}(\by)\subset \sB_{r}(\mathcal  U)$, and $\cP'=\cP\setminus \sB_r({\by)}$. Let us observe that if $\mathscr  A_ {r}^{N}(\cP \cup \by)$ is realised, then either there are two arms disconnected from $\sB_{2r}({\by)}$, in which case the configuration of $\cP \cap \sB_r({\by)}$ is irrelevant, either there is one such arm and another arm that approaches $\sB_r({\by)}$, meaning   $\mathscr  A_{r}^{N}(\cP '\cup \mathcal  U) $ occurs for all $\mathcal  U \in \mathbf  N_ {\by } '$, or there are no such  arms, which means in particular that $\mathscr  A _{2r}^{N/2}(\cP-\by)$ is realised.  Then 
\begin{align*}
\mathbb{P}(\mathscr  A_ {r}^{N}(\cP\cup \by))\leqslant& \mathbb{P}(\forall \,\mathcal  U \in \mathbf  N_{\by},\mathscr  A_ {r}^{N}(\cP '\cup \mathcal  U ))+\mathbb{P}(\forall\; \mathcal  U \in \mathbf  N_{\by}':\mathscr  A_ {r}^{N}(\cP '\cup \mathcal  U ))+\mathbb{P}(\mathscr  A _{2r}^{N/2}(\cP ))\\
\leqslant &\mathbb{P}(\mathscr  A_ {r}^{N}(\cP ))+\frac{\mathbb{P}(\mathscr  A_ {r}^{N}(\cP ))}{\mathbb{P}(\mathscr  A_ {r}^{N}(\cP )\;|\;\forall\; \mathcal  U  \in \mathbf  N_{\by}':\mathscr{A}(\cP' \cup \mathcal  U  ))}+\mathbb{P}(\mathscr  A _{2r}^{N/2}(\cP )),
\end{align*}
and
\begin{align*}
\mathbb{P}(\mathscr  A_ {r}^{N}(\cP)\;|\;\forall\; \mathcal  U  \in \mathbf  N_{\by}':\mathscr{A}(\cP' \cup \mathcal  U  ))\geqslant \mathbb{P}(\cP \cap \sB_r(\by)\in \mathbf  N_{\by}')=:\delta _{r}>0.
\end{align*}
In view of the right hand side of \eqref{eq:two-arm}, this yields the existence of some $c>0$ such that $\mathbb{P}(\mathscr  A_ {r}^{N}(\cP \cup y))\leqslant c\mathbb{P}(\mathscr  A_ {r}^{N}(\cP ))$,
 which in turn immediately  yields \eqref{eq:rate-comp} using Remark \ref{r:ruse}.
   \end{proof}

%

%


\section{Application to heavy-tailed shot noise excursions}

The geometric study of the excursion sets of random fields is a classical and still active domain of research, motivated by concrete applications in engineering,  physics, biology and statistics --- see e.g. \cite{AT, AW} for an overview. Because of their tractable structure, Gaussian fields have played for a long time a privileged role, but a growing number of applications requires the use of random fields that are obtained as the convolution of a deterministic kernel with an atomic random measure --- see e.g. the monograph \cite{Bac10a} for theoretical foundations, as well as for applications to telecommunication networks. The aim of this section is to use our general bounds from Section 1 in order to study the multidimensional fluctuations of geometric functionals of {\it shot-noise processes}, that is, of random fields defined as the convolution of a deterministic kernel with a marked Poisson measure on some subset of $\R^d$. The functionals of interest will be expressed in terms of perimeters and volumes of excursion sets.

To fix the notation, select $d\geq 2$, let $g:\mathbb{R}^{d}\to \mathbb{R}$ be a mapping of class $\mathcal{C}^{1}$ that does not vanish a.e., and let $\pi$ be a probability on $\R$. We denote by $\eta$ a Poisson process on $\R^d \times \R$, with intensity ${\rm Leb}\otimes \pi$. Given a measurable $\sB\subset \R^d$, we write throughout the section $\eta_\sB = \eta |_{\sB\times \R}$ and define
\begin{align*}
X_{\eta_\sB}(\bx) := \int_\R \int_{\R^d} m\, g(\bx - \by) \eta_\sB(dm, d\by) =  \sum_{i\ge 1}M_{i}g(\bx-\bx_{i}),\quad \bx\in \mathbb{R}^{d},
\end{align*}
(where $\{ (\bx_i, M_i): i\geq 1\}\subset \sB\times \R$ is any enumeration of the support of $\eta_\sB$) to be the homogeneous Poisson {\bf shot-noise process} with kernel $g$ and marks distribution $\pi$. If $g\in L^{1}(\mathbb{R}^{d})$ and $\int |x|\pi(dx)<\infty $, the field $X_{\eta_\sB}$ is well-defined for every Borel set $\sB$ and every $\bx\in \R^d$, since in this case the Campbell-Mecke formula (see \cite[Chapter 4]{LP}) yields
\begin{align}
\label{eq:sn-integr}
\mathbb{E}\left[
\sum_{i\geq 1 }\left|{M_i}
g(\bx-\bx_{i})
\right|
\right]\leq (\mathbb{E}|M|)\|g\|_{L^{1}}<\infty,
\end{align}
where $M$ is any random variable with law $\pi$. Our main quantitative results are stated in the forthcoming Section \ref{ss:mainexc}. One remarkable feature of our findings is that they only require that $|g(x)|$ and $\|\nabla g(x)\|$ verify an polynomial decay condition of order $<-d$ at infinity, which is a minimal requirement for ensuring the integrability of \eqref{eq:sn-integr}.

Recently, Bulinski, Spodarev and Timmerman \cite{BST} have obtained results for the excursion volume of  {\bf quasi-associated} random fields, which applies to shot noise fields under an assumption of polynomial decay of the order at most $-3d$. In \cite{Lr19}, presumably optimal rates of convergence in the Kolmogorov distance are proved for the volume, the perimeter, and the Euler characteristic of the excursion sets of shot-noise processes, under a stronger assumption of polynomial decay. Central limit theorems of the type derived in this section can be useful for building inference and testing procedures to recover global characteristics of a stationary isotropic random field by using sparse information -- see e.g. \cite{BdBDE}. 

\begin{remark} If $\xi$ is a finite subset of $\R^d \times \R$, we define analogously
$
X_\xi(\bx) := \sum_{(m,\by)\in \xi} m\, g(\bx-\by).
$

\end{remark}


\subsection{Excursion functionals}

 From now on, we assume for simplicity that the marks equal $\pm 1$ with probability $1/2$, i.e. $\pi  =\frac{1}{2}(\delta _{1}+\delta _{-1})$. We now fix $\sB\subset \R$ such that $|\sB|<\infty$, and consider the random elements $\eta_\sB$ and $X_{\eta_\sB}$  introduced above; since $\eta_\sB$ has a.s. finite support, $X_{\eta_\sB}$ inherits the smoothness properties of $g$ with probability one. We define the {\bf excursion sets} of $X_{\eta_\sB}$ as follows:
\begin{align*}
E_{u, \eta_\sB }:=\left\{\bx \in \R^d: X_{\eta_\sB }(\bx) \geq u\right\}, \quad u \in \mathbb{R}.
\end{align*}
Our goal is to study geometric quantities associated with such excursion sets, in particular, the volume and the perimeter. We define the {\bf excursion volume} in the observation window $\sW$ by
\begin{align*}
V\left(u, \eta_\sB , \sW\right):=\mathcal{H}^{d}\left(E_{u, \eta_\sB } \cap \sW\right)
\end{align*}
and its smoothed version by
\begin{align*}
V(\varphi, \eta_\sB , \sW):=\int_{\mathbb{R}} \varphi(u) V(u, \eta_\sB , \sW) d u
\end{align*}
where $\mathcal{H}^{k}$ denotes the $k$ -dimensional Hausdorff measure, and $\varphi: \mathbb{R} \rightarrow \mathbb{R}$ is a smooth function with compact support. The connection between $V(u, \eta_\sB , \sW)$ and $V(\varphi, \eta_\sB , \sW)$ is made clear as follows. Let $\varphi_{n}$ be any approximation of identity shifted by some fixed $u \in \mathbb{R}$. Then, under appropriate non-degeneracy assumptions on $g$, we have $V(u, \xi , W)=\lim _{n \rightarrow \infty} V(\varphi_{n}, \eta_\sB , \sW) $. Similarly, we define the {\bf excursion perimeter} in the window $\sW$ by
\begin{align*}
L(u, \eta_\sB , \sW)=\text{\rm{Per}}(E_{u,\eta_\sB };\sW):=\mathcal{H}^{d-1}\left(\partial E_{u, \eta_\sB } \cap \sW\right)
\end{align*}
and its smoothed version by
\begin{align*}
L(\varphi, \eta_\sB, \sW)=\int_{\mathbb{R}} \varphi(u) L(u, \eta_\sB , \sW) d u.
\end{align*}
This version is more convenient to study because the coarea formula \cite{BieDes12} yields the representation
\begin{align}
\label{eq:coarea}
L(\varphi, \eta_\sB , \sW)=\int_{\sW}\varphi (X_{\eta_\sB }(\bx))\|\nabla  X_{\eta_\sB }(\bx)\|d\bx.
\end{align}

\subsection{Bounds to the normal} \label{ss:mainexc}

Let $\sB_0$ be the unit ball centered at the origin, and let $\sB_n = n\sB_0$, $n\geq 1$. Our main finding is the following quantitative multidimensional CLT.

\begin{theorem}
\label{thm:sn}
Let $m\geq 1$, and for $1\leq i\leq m$, let $\varphi_{i} $ be a non-constant $\mathcal{C}^{1}$ non-negative function with compact support. Define
\begin{align*}
F_{i}=U_{i}(\varphi _{i},\eta _{\sB_{n}},\sB_{n})
\end{align*}
with $U_{i}\in \{V,L\}$, and $\bF=(F_{i})_{1\leq i\leq m}$.
Assume that  $g(x)$ is $\mathcal{C}^{1}$, non-negative and $ | g(x) | ,\|\nabla g(x)\|\leq C_{g}(1+\|x\|)^{-\delta },x\in \mathbb{R}^{d} $, for some $C_{g}\geq 0,\,\,\delta >d$. Then, for $n$ sufficiently large the variances satisfy the   bounds 
\begin{align}
\label{eq:var-sn}
 \sigma_{i} |\sB_n| \leq \Var(F_{i} )\leq  c_{i} |\sB_n| 
\end{align}for some $c_{i}\geq \sigma_{i} >0$ independent of $n$. Moreover, let $N_{\Sigma _{n}}$ be a centered Gaussian vector with the same covariance structure as $\widehat{\bF}=n^{-d/2}( \bF - \EE \bF)$; then, for every  $0<\alpha <1$ and $\gamma \in (0, \frac12)$,
\begin{align}\label{e:boundz}
d_3\left(\mathbf{\wh F},  N_{\Sigma_{n}}\right) \le C_{\alpha ,\gamma} \left[ 
  n^{(\alpha -1)\frac{d}{2}}   +  n^{-\alpha\, \gamma \frac{(d-\delta )^{2}}{\delta }} \,\,\right]\to 0, \quad n\to\infty
,
  \end{align}
where $C_{\alpha ,\gamma}$ is a constant depending on $\alpha, \gamma, g, U_i, \varphi_i$. Furthermore, if the asymptotic covariance matrix $\Sigma_{\infty }=\lim_{n}\Sigma _{n} $ exists  {and is positive definite}, then the bound \eqref{e:boundz} continues to hold for $d_3$ and $d_c$ (possibly with a different constant $C_{\alpha, \gamma}$).  
\end{theorem}

We will actually show that the estimate \eqref{e:boundz} holds for $d_2, d_3$ with $\gamma$ replaced by $\frac12 - \frac4p$ for an arbitrary $p>4$, and for $d_c$ with $\gamma$ replaced by $\frac12 - \frac3p$ for all $p>6$. The proof of Theorem \ref{thm:sn} is presented in the forthcoming Sections \ref{sec:var-lower-sn} and \ref{ss:wsup}.


\subsection{Variance lower bound}
\label{sec:var-lower-sn}

 Fix $i$ and set $\varphi =\varphi _{i}$ and $F=F_{i}$. We prove below that $\Var(F)\geq c | \sB_{n} | $ for some $c>0.$  In the volume case, i.e. when $U_{i}=V$,  the functional can be rewritten as
\begin{align}
\label{eq:vol-rep}
& F=\int_{\sB_{n}}\int_{}\varphi (u)\mathbf{1}_{\{X_{\eta _{\sB_{n}}}(\bx)\geq u\}}du d\bx =\int_{}\Phi(X_{\eta _{\sB_{n}}}(\bx))d\bx
\end{align}
where $\Phi$ is the primitive function of $\varphi $ vanishing at $-\infty .$
We will use \cite[Th. 5.3]{LPS16} to lower bound the variance in both cases $U_i = V, L$.  It consists in finding two finite sets $\xi_1 ,\xi_2 \subset \mathbb{R}^{d}\times \{-1,1\}$ such that $|\mathbb{E}(F(\eta +\xi_1 )-F(\eta +\xi_2 )|)$ is uniformly bounded from below over translations and small perturbations of $\xi_1 ,\xi_2 .$

Let $R>1,0<\alpha  ,\lambda <1$ such that $2\alpha d+\lambda d+d<\delta , \alpha d<(\delta -d)\lambda $. 
Let $r=R^{\lambda }$, and the sets of marked points 
\begin{align*}
\xi_1 &=((R^{-\alpha }\mathbb{Z}^{d}) \cap \sB_{R})\times \{+1\}\\
\xi_2 &=((R^{-\alpha }\mathbb{Z}^{d}) \cap \sB_{R}^{r})\times \{+1\}\text{\rm{ where }}\sB_{R}^{r}=\sB_{R}\setminus \sB_{r}.
\end{align*}
Remark that the whole problem is invariant under the translation of $g$ by a fixed vector, i.e. if $g$ is replaced by $g(\bx_{0}+\cdot )$ for some $\bx_{0}\in \mathbb{R}^{d}$.
Hence we assume without loss of generality that $g(\b0)>0$, hence for some $\varepsilon _{1}\in (0,1),g\geq g(\b0)/2>0$ on $\sB_{\varepsilon_1} $. For $\bx$ belonging to  $\sB\in \{\sB_{R},\sB_{R}^{r}\}$    it also belongs to some ball   $\sB_{\varepsilon _{1}/2}(\by)\subset \sB$, and
\begin{align}
\label{eq:lower-Xxi}
\min(X_{\xi_1  }(\bx),X_{\xi_2 }(\bx))\geq \sum_{\bz\in (R^{-\alpha }\mathbb{Z}^{d}) \cap \sB_{\varepsilon _{1}/2}(\by)}g(\bz-\bx) \geq  \sum_{\bz\in (R^{-\alpha }\mathbb{Z}^{d}) \cap \sB_{\varepsilon _{1}/2}(\by)}g(\b0)/2\geq  \kappa R^{\alpha d} 
\end{align}
for some fixed $\kappa >0$. For $x$ at distance $\rho \geq  0$ from $\sB\subset \mathbb{R}^{d},\xi =(R^{-\alpha }\mathbb{Z} ^{d})\cap \sB\times \{1\},$
\begin{align}
\label{eq:bd-xi-rho}
X_{\xi }(\bx)\le  \sum_{\by\in (R^{-\alpha }\mathbb{Z} ^{d})\cap \sB_{\rho }(\bx)^{c}}\hspace{-1cm}C_{g}(1+ \| \by-\bx \| )^{-\delta }\le  c_{1}\int_{\rho }^{\infty }R^{\alpha d}(1+ \| \by-\bx \| )^{-\delta }d\by\le  c_{2}R^{\alpha d}(1+\rho) ^{d-\delta }
\end{align}
for some $c_{1},c_{2}<\infty .$
 Let $\Omega  _{R}$ be the event that 
\begin{align}
\label{eq:X-eta-bd}
\sup_{\bx\in \sB_{R}}\sum_{(\by,m)\in \eta }C_{g}(1+\|\by-\bx\|)^{-\delta }<\frac{\kappa R^{\alpha d}}{2},
\end{align}
whose complement  probability is bounded by $c_{3}R^{d}\exp(-c_{3}(R^{\alpha d})^{\nu })$ for some $c_{3},\nu >0$, see for instance \cite[Proposition A.1]{LrM19}. Define also 
\begin{align*}
X_{i}&=X_{\eta +\xi_i }=X_{\eta }+X_{\xi_i },\;i=1,2.
\end{align*}
Assume from now on up to increasing $R$ that $\frac{3\kappa }{2}R^{\alpha d}>u$ and $c_{2}R^{\alpha d}r^{d-\delta } <u/4.$

\begin{lemma}If $\Omega _{R} $ is realised, for $v\in[0,u],$
\begin{align*}
0=\text{\rm{Per}}(v,X_{1},\sB_{R})&\le  \text{\rm{Per}}(v,X_{2},\sB_{R})- \mathbf{1}_{\{v\geq  u/2\}}\text{\rm{Per}}(u/4,X_{\eta },\sB_{1})\\
 \text{\rm{Vol}}(v,X_{2},\sB_{R})+\mathbf{1}_{\{v\geq  u/2\}}\text{\rm{Vol}}(u/4,X_{\eta },\sB_{1})&\le  \text{\rm{Vol}}(v,X_{1},\sB_{R})=\text{\rm{Vol}}(\sB_{R})
\end{align*}
\end{lemma}

\begin{proof}
Let $\bx\in \sB_{R}$. We have using \eqref{eq:lower-Xxi}
 and \eqref{eq:X-eta-bd}
\begin{align*}
X_{1}(\bx)\geq  X_{\xi_1}(\bx)- | X_{\eta }(\bx) | \geq   (\kappa -\kappa /2)R^{\alpha d}\geq  u\geq  v
\end{align*}
hence $\sB_{R}\subset E_{v,X_{1}}$. Similarly, for $x\in \sB_{R}^{r}$, $
X_{2}(\bx)\geq  v $
and $\sB_{R}^{r}\subset E_{v,X_{2}}$.
If $v\geq  u/2$, for $\bx\in \sB_{1}$, \eqref{eq:bd-xi-rho} yields $
X_{ \xi _{2}  }(\bx)\le  c_{2}R^{\alpha d}r^{d-\delta } <u/4$. If furthermore $X_{\eta }(\bx)\le  u/4,X_{2}(\bx)<u/2\leqslant   v$, hence $\sB_{1}\cap E_{u/4,X_{\eta }}\subset E_{v,X_{2}}^{c}$, which allows to conclude the proof.

\end{proof} 

Hence since $\text{\rm{Supp}}(\varphi )\subset [0,u],$ with $U\in \{\text{\rm{Vol}},\text{\rm{Per}}\},$
\begin{align*}
 | \mathbb{E}(\mathbf{1}_{\{\Omega _{R} \}}\int_{\mathbb{R}}\varphi (v)&\left[
U(v,\eta +\xi _{1},\sB_{R})-U(v,\eta +\xi _{2},\sB_{R})
\right]dv) | \\
&\geq  \int_{u/2}^{u}\varphi (v)\mathbb{E}(U(u/4, {\eta },\sB_{1})\mathbf{1}_{\{\Omega _{R} \}})dv\xrightarrow[R\to \infty ]{}\;\delta :=\mathbb{E}(U(u/4, {\eta },\sB_{1}))\int_{u/2}^{u}\varphi (v) dv
\end{align*}
and $\delta >0$ since  the random excursion set $E_{u/4,X_{\eta }}$ is non-trivial and stationary.

In the perimeter case we have, using the representation \eqref{eq:coarea},
\begin{align*}
 |\text{\rm{Per}}&(\varphi ,X_{1},\sB_{R}^{c})- \text{\rm{Per}}(\varphi ,X_{2},\sB_{R}^{c})|
= \left|
\int_{\sB_{R}^{c}}\varphi (X_{1}(\bx))\|\nabla X_{1 }(\bx)\|d\bx-\int_{\sB_{R}^{c}}\varphi (X_{ 2  }(\bx))\|\nabla X_{2  }(\bx)\|d\bx
\right|\\
\leq &
\int_{\sB_{R}^{c}}\left[
\left|\varphi (X_{1}(\bx))-
\varphi (X_{ 2   }(\bx))\right|\|\nabla X_{ 1 }(\bx)\|
+|\varphi (X_{ 2  }(\bx))|\left|
\|\nabla X_{ 1 }(\bx)\|-\|\nabla X_{ 2  }(\bx)\|
\right|
\right]d\bx\\
\leq & \int_{\sB_{R}^{c}}\left[
\|\varphi '\|  
\sum_{\by\in (R^{-\alpha }\mathbb{Z} ^{d})\cap \sB_{r}}|g(\by-\bx)|
 \|\nabla X_{1  }(\bx)\| +\|\varphi \|\sum_{\by \in (R^{-\alpha }\mathbb{Z} ^{d})\cap \sB_{r}}\| \nabla g(\by-\bx)\|
\right]d\bx.
\end{align*}
For the first term we have with \eqref{eq:bd-xi-rho}, for some $c_{4},c_{5},c_{6},c_{7}<\infty ,$ for $\bx\in \sB_{R}^{c},$ with $R/2>R-r,$
\begin{align*}
\mathbb{E}\|\nabla X_{1 }(\bx)\| &\le  \mathbb{E}(\|\nabla X_{\xi_1 }(\bx)+\|\nabla X_{\eta }(\bx)\|)\le   C_{g} \sum_{\by\in( R^{-\alpha }\mathbb{Z} ^{d})\cap \sB_{R}}(1+\|\by-\bx\|)^{-\delta }+c_{4}\le  c_{5}R^{\alpha d},\\
\int_{\sB_{R}^{c}} 
\mathbb{E}\|\nabla X_{1 }(\bx)\|& \sum_{\by\in (R^{-\alpha }\mathbb{Z}^{d} )\cap \sB_{r}} (1+ \| \by-\bx \| )^{-\delta } d\bx
 \\ &\quad\quad\quad \le  c_{6}R^{\alpha d}\int_{\sB_{R}^{c}}r^{d}R^{\alpha d}  (\|\bx\|-R/2)^{-\delta }d\bx\le  c_{7}R^{2\alpha d+\lambda d+d-\delta }.
\end{align*}
Performing similar computations for the second term, and for the volume via the representation \eqref{eq:vol-rep},  {we infer that} there are $c_{8},c_{9}<\infty $ such that
\begin{align*}
|\mathbb{E} ( U(\varphi , &{\eta +\xi_1 },\sW)- U(\varphi , {\eta +\xi_2 },\sW)) |\\
 \geq & |\mathbb{E} (   U(\varphi ,\eta +\xi _{1},\sB_{R})-U(\varphi ,\eta +\xi _{2},\sB_{R}) |-\mathbb{E}(|U(\varphi ,\eta +\xi _{1},\sB_{R}^{c})-U(\varphi ,\eta +\xi _{2},\sB_{R}^{c})|)\\
\geq & |\mathbb{E}(  U(\varphi ,\eta +\xi _{1},\sB_{R})-U(\varphi ,\eta +\xi _{2},\sB_{R}) \mathbf{1}_{\{\Omega_{R} \}} )|-2\text{\rm{Vol}}(\sB_{R})\|\varphi \|_{L^{1}}\mathbb{P}(\Omega _{R}^{c})-c_{8}R^{(2\alpha+\lambda +1 )d -\delta }\\
\geq &\delta +o(1)-c_{9}c_{3}R^{2d}\exp(-c_{3}R^{\alpha d \nu })-c_{8}R^{(2\alpha+\lambda +1 )d -\delta }
\end{align*} 
which is larger than, say, $\delta /2$ for $R$ sufficiently large, fixed. Continuity arguments yield $\varepsilon >0$ such that if each point of $\xi_1 ,\xi_2 $ is perturbed by a quantity in $\sB_{\varepsilon }$, we still have the inequality with $\delta /3$ instead. The inequality still holds after translating $\sB_{R}$ and $\sB_{R}^{r}$ by the same vector $\bz$ as long as   $\sB_{R+\varepsilon }+\bz \subset \sW$. Let us define 
\begin{align*}
\sA=&\{\bz\text{\rm{ such that }}\sB_{R+\varepsilon }(\bz) \subset \sW\} \\
\mathscr   U_{i}=&\{(\bx+\bz+\by_{x})_{\bx\in \xi _{i}}: \by_{x}\in \sB_{\varepsilon }(0),\bz\in A\}\subset (\mathbb{R}^{d})^{ | \xi _{i} | }
\end{align*}
and remark that $ | \sA | \geq c'n^{d}$ for some $c'>0.$
We have for every $\xi _{i,\varepsilon }\in  \mathscr  U_{i}$,
\begin{align}
\label{eq:uniform-bound}
\mathbb{E}\left[
\left|
U(\varphi ,\eta _{| \sW}+ \xi _{1,\varepsilon },\sW)-U(\varphi ,\eta_{|\sW}+ \xi _{2,\varepsilon },\sW)
\right|
\right]\geq  \frac{\delta }{3}>0.
\end{align}
Let $t=n^{d},\eta _{t}=n^{-1}\eta ,f_{t}(\eta _{t})=F$, $\sA,\varepsilon $ like above. Then we can apply \cite[Theorem 5.3]{LPS16} with these variables, hence $ \Var(F)\geq \sigma n^{d}$ for  some $\sigma >0.$
 
  \subsection{ {Two-scale} stabilisation and variance upper bound}\label{ss:wsup}
For the rest of the proof, $c,\, c'$ denote  finite constants that might vary from line to line. We write the proof for the perimeter case ($U_{i}=L$), to treat the volume case one has to replace $\varphi $ by one of its primitive functions $\Phi $, and $\|\nabla X_{\xi }(x)\|$ by $1$, see \eqref{eq:vol-rep}. Since (I),\, (II),\, (III'),\, (IV) from Section 1 are in order with
\begin{align*}
h^{i}(\sC;\chi):=&L(\varphi _{i},\chi \cap \sC,\sC),
\end{align*}we seek to apply Corollary \ref{c:euclmulti} with $Z=\mathbb{R},\pi =\frac{1}{2}(\delta _{1}+\delta _{-1}),b_{n}=n^{\alpha },0<\alpha <1$. In view of bounding $K'$ in \eqref{e:xxxx}, let $(\by_{0},m)\in \mathbb{R}^{d}\times \{-1,1\}$ and $Y$ be either $\{\delta _{(\by_{0},m)}\}$ or $\emptyset $, and denote by $\eta '=\eta +Y,F'=F(\eta '_{\sB_{n}})$, in particular $F'\in \{F,F^{\by_{0}}\} $.
Denote by $$f_{\xi }^{\sB}=X_{(\eta '+ \xi )|\sB},\;I^{\sB}_{\xi }(\bz)=\varphi( f_{\xi }^{\sB}(\bz)) \|\nabla  f_{\xi }^{\sB}(\bz)\|,$$ with the  {shorthand notations} $I_{\sB}=I_{\sB}^{\emptyset },I=I_{\sW},I_{\bx}=I_{\{\bx\}}, ...$ so that for $\bx\in \mathbb{R}^{d}\times \{-1,1\},$
\begin{align*}
F' = \int_{\sB_{n}}I(\bz)d\bz,\;\;\;\; 
D_{\bx}F' = \int_{\sB_{n}}(I_{x}(\bz)-I(\bz))d\bz.
\end{align*}
First notice that, since $\|g(\bx)\|,\|\nabla g(\bx)\|$ are bounded and  {decrease at a rate not slower than} $\|\bx\|^{-\delta }$ for some $\delta  >d$ ( {$\|\bx\|\to\infty$}), they belong to $L^{p}$ for all $p\in (0,\infty ]$. Hence for all $\xi  ,\sB\subset \mathbb{R}^{d}$, the multi-variate Mecke formula yields
\begin{align*}
  \mathbb{E}| f _{\xi }^{\sB}(\bz) | ^{p} \leq \mathbb{E}\left(
\sum_{\by\in \eta '+ \xi } | g(\by-\bz) | 
\right)^{p}\leq C_{p}(\sup_{q\leq p}\|g\|_{L^{q}}^{p}+ (| \xi  | +1 ) \|g\|_{\infty })
\end{align*}
and a similar bound holds for $\mathbb{E}\|\nabla f_{\xi }^{B}\|^{p}.$ Then
\begin{align}
\notag
 | I_{\bx} ^{\sB}(\bz)-I^{\sB}  (\bz) | \leq & \|\nabla f_{\bx}^{B}(\bz)\| | f_{ \bx}^{\sB}(\bz)-f ^{\sB}(\bz) | + | f ^{B}(\bz) | \|\nabla f_{ \bx}^{B}(\bz)-\nabla f^{B}(\bz)\|\\
 \leq & \max\left(
\label{eq:Ixz} | f^{B}(\bz) | ,\|\nabla f_{\bx}^{B}(\bz)\|
\right)\times \underbrace{\max( | g(\bz-\bx) | ,\|\nabla g(\bz-\bx)\|)}_{=:G(\bz-\bx)}.
 \end{align}
We have for $p\in \mathbb{N},\bx\in \sB_{n}$, 
\begin{align}
\notag\mathbb{E}( | D_{\bx}F'  | ^{p})\leq & c\int_{\sB_{n}^{p}} |G(\bx-\bz_{1})\dots G(\bx-\bz_{p}) | d\bz_{1}\dots d\bz_{p}\\
\label{eq:unif-moment}\leq & c'\|G\|_{L^{1}}^{p}<\infty 
\end{align}
hence the strong moment condition \eqref{e:xxxx} is satisfied for every $p.$ For $p=2$, Poincar\'e's inequality also yields that the variance is bounded by $c|\sB_{n}|$, see for instance \cite{Last}, this concludes the proof of \eqref{eq:var-sn}. We have for $\bx\in \sB_{n},\sA_{\bx}\subset \mathbb{R}^{d}$
\begin{align*}
F' -F' (\sA_{\bx})=&\int_{\sB_{n}}(I(\bz)-I^{\sA_{\bx}}(\bz))d\bz,\\
D_{\bx}F' -D_{\bx}F '(\sA_{x})=&\int_{\sB_{n}}\left[\underbrace{
I_{\bx}(\bz)-I(\bz)-(I_{ \bx}^{\sA_{\bx}}-I^{\sA_{\bx}}(\bz))}_{=:D_{\bx}(\bz)}
\right]d\bz\\
|D_{\bx}(\bz)| \leq &\min\left(
|I_{\bx}(\bz)-I(\bz)|+|I_{x}^{\sA_{\bx}}(\bz)-I^{\sA_{x}}(\bz)|,\;|I_{\bx}(\bz)-I_{\bx}^{\sA_{\bx}}(\bz)|+|I(\bz)-I^{\sA_{\bx}}(\bz)|
\right).
\end{align*}
Typically, the first term in the $\min$ will be small when $\bz$ is far from $\bx$, and the second one will be small when $\bz$ is close from $\bx$. We define the cutoff region to be some ball $\sB_{n^{\beta }/2}(\bx)$ for some $0\leq \beta \leq \alpha .$ We have for $\bz\notin \sB_{n^{\beta }/2}(\bx)$, using \eqref{eq:Ixz}, for $\sA\in \{\mathbb{R}^{d},\sA_{\bx}\},$
\begin{align*}
\mathbb{E}(| I_{\bx}^{\sA}(\bz)-I ^{\sA}(\bz) |)\leq &\left(
\|g\|_{L^{1}}+ | g(\bx-\bz) | + | g(\bz-\by_{0}) |+\|\nabla g\|_{L^{1}}+ \| \nabla g(\bx-\bz) \| + \| \nabla g(\bz-\by_{0})  \| 
\right)\\
&\times \max( | g(\bx-\bz) | ,\|\nabla g(\bx-\bz)\|)
 \end{align*}
hence since $ | g(\bz) | ,\|\nabla g(\bz)\|$ are bounded by $c\|\bz\|^{-\delta },$ using also \eqref{eq:unif-moment},
\begin{align*}
\mathbb{E}( | D_{\bx}(\bz) | )\mathbf{1}_{\{\bz\notin \sB_{n^{\beta }/2}(\bx)\}}\leq c'  | \bx-\bz | ^{-  \delta }.
\end{align*}
For $\bz\in \sB_{n^{\beta }/2}(\bx)$, with $X\in \{\emptyset , \delta _{\bx}\},\eta ''=\eta ' + X,$
\begin{align*}
 | I_{X}(\bz)-I_{X}^{\sA_{\bx}}(\bz) | \leq & | f_{X}(\bz) | \|\nabla f_{X}^{\sA_{\bx}}(\bz)-\nabla f_{X} (\bz)\|+\|\nabla f_{X}(\bz)\| | f_{X}^{\sA_{\bx}}(\bz)-f_{X}(\bz) | \\
 \leq &  | f_{X}(\bz) | \sum_{\by\in \eta ''\setminus \sA_{\bx}}\|\nabla g(\by-\bz)\|+\|\nabla f_{X}(\bz)\|\sum_{\by\in  \eta ''\setminus \sA_{\bx}} | g(\bz-\by) |.
\end{align*}
We have, by Campbell's formula,
\begin{align*}
&\mathbb{E}\sum_{\by\in \eta ''\setminus \sA_{\bx}} | f_{X}(\bz) | \|\nabla g(\by-\bz) \| \\
& \leq \int_{\sA_{\bx}^{c}}\mathbb{E}( | f_{X}(\bz) | )\|\nabla g(\by-\bz)\|d\by+ \mathbf{1}_{\{Y=\{\by_{0}\},\by_{0}\notin \sA_{\bx}\}} \|\nabla g(\by_{0}-\bz) \|\mathbb{E} | f_{X}(\bz) |  \\
&\leq c(\int_{n^\alpha }^{\infty }(r-n^{\beta }/2)^{-\delta }r^{d-1}dr+\mathbf{1}_{\{Y=\{\by_{0}\}\}}(n^{\alpha }-n^{\beta }/2)^{-\delta })\\
&\leq  c'(n^{\alpha (d-\delta )}+\mathbf{1}_{\{Y=\{\by_{0}\}\}}n^{-\alpha\delta  })\\
&\leq 2c'n^{\alpha (d-\delta )},
\end{align*}
from whioch we deduce that a similar bound  holds for  $ \mathbb{E} | D_{\bx}(\bz) |.$
Finally,
\begin{align*}
\mathbb{E}( | D_{\bx}F' -D_{\bx}F' (\sA_{\bx}) | )\leq& c \left(
n^{\beta d} n^{\alpha  (d-\delta )} +\int_{n^{\beta }}^{\infty }r^{-\delta }r^{d-1}dr
\right)\\
\leq &c(n^{\beta d+\alpha (d-\delta )}+n^{\beta (d-\delta )}).
\end{align*}
With $\beta =\alpha (\delta -d)\delta ^{-1}  ,$ the two exponents equal $\alpha \delta ^{-1}(d-\delta )^{2}$, hence in Corollary \ref{c:euclmulti},
\begin{align*}
\vartheta(n)\leq &\,\,cn^{-\alpha (d-\delta )^{2}(1-4/p)/\delta }\\
\tau (n),\rho (n)\leq &\,\, c\sup_{q=4,5,6}cn^{-\alpha (d-\delta )^{2}(1-q/p)/\delta }
\end{align*}
which in turn concludes the proof of Theorem \ref{thm:sn}.

\appendix

\section{Ancillary results}\label{s:ancillary}

\subsection{Malliavin calculus on the Poisson space}\label{appendix}

We recall some fundamental results related to Malliavin calculus for Poisson random measures, following \cite{Last} (to which we refer the reader for further details and motivation), in the general framework of Section \ref{ss:frame}. In particular, in this section we denote by $\eta$ a Poisson random measure on a measurable space $(\mathbb X, \mathcal X)$ with $\sigma$-finite intensity $\la$. 

\smallskip

\noindent\underline{\it Wiener-It\^o chaos expansions}. The {\bf Wiener-It\^o chaos expansion} of Poisson functionals \cite[Theorem 1 and Theorem 3]{Last} determines an isomorphism between $L^2(\Omega)$ and the Fock space of sequences $\{ f_n : n\ge 0\}$ of functions satisfying $f_n\in L^2_s(\la^n)$, where (for $n\geq 1$) $L^2_s(\la^n)$ indicates the Hilbert space of a.e. symmetric square-integrable functions with respect to the product measure $\la^n$, and $L_s^2(\la^0)=\RR$. Consider a random variable $F\in L^2(\Omega)$, then $F$ admits a unique Wiener-It\^o chaos expansion of the type
\begin{align}\label{e:chaos}
F = \EE[F] + \sum_{n\ge 1} I_n (f_n),
\end{align} 
where the series converges in $L^2(\Omega)$, $I_n(\cdot)$ is the $n$-th multiple Wiener-It\^o integral with respect to the compensated Poisson measure $\hat \eta$, and $f_n\in L_s^2(\la^{n})$. 

We will now present the definitons and basic properties of several Malliavin operators.

\smallskip

\noindent\underline{\it The Malliavin derivative}. We denote by $\mathbb{D}^{1,2}$ the collection of those $F\in L^2(\Omega)$ such that the kernels of the chaos expansion \eqref{e:chaos} satisfy 
\begin{align} \label{e:d12}
\sum_{n\ge 1} n \cdot n! \|f_n\|^2_n <\infty,
\end{align}
where $\|\cdot\|_n$ is the norm associated with the inner product $\langle \cdot,\cdot\rangle_n$ of $L^2(\la^n)$. It turns out \cite[Theorem 3]{Last} that $F\in \mathbb{D}^{1,2}$ if and only if $DF\in L^2(\PP\otimes \la)$, where $D$ denotes the add-one cost operator defined in \eqref{defdp}, and in this case we have that
\begin{align*}
D_x F = \sum_{n\ge 1} n I_{n-1}(f_n(x,\cdot)), \quad \PP\otimes\la\mbox{-a.e.}.
\end{align*}
\smallskip

\noindent\underline{\it Kabanov-Skorohod integrals}. Let $u \in L^2(\PP\otimes\la)$ and observe  that, for $\la$-a.e. $x$, the random variable $u(x)$ is in $L^2(\Omega)$ and thus admits the representation 
\begin{align*}
u(x) = \EE[u(x)] + \sum_{n\ge 1}  I_n(u_n(x,\cdot))
\end{align*}
where the mapping $(y_1,..,y_n)\mapsto u_n(x,y_1,...,y_n)$ belongs to $L^2_s(\la^{n})$.  We denote by $\tilde u_n$ the symmetrization of $u_n$ in its $n+1$ arguments, that is: $\tilde u_n(x_1,...,x_{n+1}) = \frac{1}{(n+1)!}\sum_{\pi} u_n(x_{\pi(1)},...,x_{\pi(n+1)})$ where the sum runs over all permutations of $[n+1]$. Since $u\in L^2(\PP\otimes\la)$, we see that $\tilde u_n\in L_s^2(\la^{n+1})$. Denote by $\dom \delta$ the family of those $u\in L^2(\PP\otimes \la)$ such that 
\begin{align*}
\sum_{n\ge 0}  (n+1)! \|\tilde u_n\|_{n+1}^2 <\infty,
\end{align*}
and, for $u\in \dom \delta$, define {\bf the Kabanov-Skorohod integral} of $u$ as 
\begin{align*}
\delta(u):= \sum_{n\ge 0} I_{n+1}(\tilde u_n).
\end{align*}
To evaluate the variance of Kabanov-Skorohod integrals, we use the following isometric property. 
\begin{lemma}[{\cite[Theorem 5]{Last}}] \label{l:SkorohodVariance} Let $u\in L^2(\PP\otimes\la)$ be such that 
\begin{align*}
\EE \iint D_x u(y) \la^2(dx,dy) <\infty.
\end{align*}
Then, $u\in\dom\delta$ and 
\begin{align*}
\EE[\delta(u)^2] =\EE \int u^2(x) \la(dx) + \EE \iint D_x u(y) D_y u(x) \la^2(dx,dy).
\end{align*}
\end{lemma}
The next duality relation \cite[Theorem 4]{Last} between $D$ and $\delta$ is fundamental for implementing Stein's method. 
\begin{lemma}\label{l:duality}
Let $F\in \mathbb{D}^{1,2}$ and $u\in \dom\delta$, then 
\begin{align*}
\EE[\langle DF, u\rangle] =\EE[F\de(u)].
\end{align*}
 where $\langle \cdot,\cdot\rangle $ denotes the $L^2(\la)$ inner product. 
\end{lemma}
\noindent In the proofs of our main results, we will sometimes need to apply the above duality relation to an indicator $F=\1(G>z)$ where $G$ is $\sigma(\eta)$-measurable and $z\in\RR$. In this case, we will implicitly exploit the fact that, if $G$ only depends on the restriction of $\eta$ to a set $\sB$ such that $\lambda(\sB)<\infty$, then $D_xF = 0$ for $x$ outside $\sB$, and the condition $F\in \mathbb{D}^{1,2}$ is therefore trivially implied by the fact that $| DF | \leq 2$.  

\smallskip

\noindent\underline{\it The generator of the Ornstein-Uhlenbeck semigroup}. Denote by $\dom L$ the family of those $F\in L^2(\Omega)$ such that the kernels in the expansion \eqref{e:chaos} satisfy the relation
\begin{align*}
\sum_{n\ge 1} n^2 n! \|f_n\|_n^2<\infty. 
\end{align*}
We define the operator $L : \dom L\to L^2(\Omega)$ (called the {\bf generator of the Ornstein-Uhlenbeck semigroup}) as follows: for $F\in \dom L$,
\begin{align*}
LF := - \sum_{n\ge 1} n I_n(f_n). 
\end{align*}
The {\bf pseudo-inverse} of $L$ is defined for every $F\in L^2(\Omega)$ and is given by 
\begin{align*}
L^{-1} F := - \sum_{n\ge 1} \frac 1 n I_n(f_n).
\end{align*}
It is clear that $L^{-1}F\in \dom L$, for every $F\in L^2(\Omega)$. In the next sections, we will often use the fact that, if $F\in L^2(\Omega)$ is $\sigma(\eta |_\sB)$-measurable, then $L^{-1} F$ is also $\sigma(\eta |_\sB)$-measurable, and the same conclusion holds for $LF$ in case $F\in \dom L$ (to see this, one can apply e.g. \cite[Theorem 2 and formula (16)]{Last}).  The following fact \cite[Proposition 3]{Last} relates the Malliavin operators $D,\delta, L, L^{-1}$.
\begin{lemma}\label{l:L=-delta D}
For any $F\in L^2(\Omega)$, we have $L L^{-1} F = F- \EE[F]$.  
If $F\in\dom L$, then $F\in \mathbb{D}^{1,2}$, $DF\in \dom \delta$ and $LF =- \delta D F$.
\end{lemma}

\smallskip

\noindent\underline{\it The Ornstein-Uhlenbeck semigroup}. The generator $L$ is associated with the semigroup of operators $\{T_t : t\geq 0\}$ given by 
\begin{align*}
T_t : L^2(\Omega) \to L^2(\Omega): F = \sum_{n\ge 0}  I_n(f_n)  \mapsto \sum_{n\ge 0} e^{-nt} I_n(f_n):= T_t F, \quad t\ge 0.
\end{align*}
Using the content of the forthcoming Lemma \ref{l:facts}, one can extend $T_t$ to $L^1(\Omega)$ by means of the following Mehler-type construction \cite[Section 7]{Last}. Let $s\in[0,1]$ and let $\chi\in\mathbf{N}_\sigma$ have a representation $\chi = \sum_{i=1}^k \delta_{x_i}$, with $k\in\NN_0\cup\{\infty\}$. Denote by $\chi_{s}$ be the {\bf thinned} point measure obtained by removing independently points in $\chi$ (counting multiplicities) with probability $1-s$. Denote by $\Pi_{\mu}$ the law of a Poisson measure with intensity $\mu$. For $F\in L^1(\Omega)$, we define
\begin{align}\label{e:def_Pt}
T'_t F(\chi) :=  \int \EE[ F( \chi_{e^{-t}}  + \xi )] \Pi_{(1-e^{-t})\la } (d\xi),
\end{align}
where the expectation is taken with respect to the independent thinning of the points in the support of $\chi$, and the integral is over $\mathbf{N}_\sigma$.  It is easy to see that
  \begin{align}\label{e:teta}
  T'_t F(\eta) = \int \EE[F(\eta_{e^{-t}} + \xi )|\eta] \Pi_{(1-e^{-t})\la}(d\xi).
  \end{align}
  The next lemma collects some useful facts, whose full proofs can be found in \cite[Section 7]{Last}.
\begin{lemma}\label{l:facts}
\begin{itemize}
\item[\rm (i)] For $F\in L^1(\Omega)$, $\EE[T'_t F] = \EE[F]$. 
\item[\rm (ii)] If $F\in L^p(\Omega)$ with $p\ge 1$, then $\EE[|T'_t F|^p]\le \EE[|F|^p]$. 
\item[\rm (iii)] If $F\in L^2(\Omega)$, then $T_t F = T'_t F$.
\item[\rm (iv)] If $F\in L^2(\Omega)$, then $\PP\otimes\la$-a.e. one has that $D T'_t F = e^{-t} T'_t DF$.
\item[\rm (v)] If $F\in L^2(\Omega)$, then
\begin{align*}
L^{-1} F =  - \int_0^\infty T_t F dt. 
\end{align*}
\end{itemize}
\end{lemma}

\smallskip

We will now present some useful bounds obtained by combining Malliavin calculus and the so-called Stein's method for normal approximations -- see \cite{CGS, NP}.

\subsection{Bounds}

We work within the framework and notation of the previous section, and write  
\begin{eqnarray}\label{e:innp}
\langle f, g \rangle_\lambda=\langle f, g \rangle = \int_{\mathbb X} f(x) g(x) \lambda(dx),  
\end{eqnarray}
whenever this expression is well-defined (which might happen even if $f$ or $g$ are not in $L^2(\lambda)$). The next statement will be the key for dealing with one-dimensional normal approximations in the 1-Wasserstein distance.

\begin{lemma} [{\cite[Theorem 3.1]{PSTU10}}] \label{l:bound_w} Let $F\in \mathbb{D}^{1,2}$ and $\wh F=(F-\EE F)/\sigma$ with $\sigma\in(0,\infty)$. Then, with $N(0,1)$ denoting a Gaussian random variable with mean zero and unit variance,
\begin{align*}
\dw(\wh F,N(0,1)) \le \left|1 - \frac{\Var[F]}{\sigma^2}\right| + \frac{1}{\sigma^2} \EE\big[ |\Var[F]- \langle DF,-DL^{-1}F \rangle  |\big] + \frac{1}{\sigma^3} \EE[ \langle (DF)^2,|DL^{-1}F|\rangle].
\end{align*}
\end{lemma}

Lemma \ref{l:bound_w} and Theorem \ref{l:bound_kol} are the starting point of  our general results in the one-dimensional case.  The condition $F\in \mathbb{D}^{1,2}$ is a minimal requirement in order to apply the Malliavin-Stein methodology.  In the multi-dimensional case, our reference bounds for the $d_2, d_3$ distances were  obtained in \cite{PZ10}. 

\begin{lemma}[{\cite[Theorems 3.3 and 4.2]{PZ10}}]\label{l:d_2}
 Let $\bF=(F_1,...,F_m)$ with $F_i\in L^2(\Omega)$ for each $i\in [m]$. Set $\mathbf{\wh F} =( \wh F_1,...,\wh F_d )$ where $\wh F_i:= (F_i -\EE F_i)/\sigma_i$ with some $\sigma_i\in(0,\infty)$. Let $N_\Sigma$ be a centered Gaussian vector with  $m\times m$ covariance matrix $\Sigma$. Then,
 \begin{align*}
d_3\left(\mathbf{\wh F}, N_\Sigma\right) \le  \frac{m}{2} \sqrt{\sum_{i,j=1}^m \frac{1}{\sigma_i^2\sigma_j^2} \EE[(\Sigma(i,j)\sigma_i\sigma_j-\langle D F_i, -DL^{-1}F_j\rangle)^2] } \\
 + \frac{1}{4} \EE \left\langle \Big(\sum_{i=1}^m \frac{|DF_i|}{\sigma_i}\Big)^2, \sum_{j=1}^m \frac{|DL^{-1}F_j|}{\sigma_j}\right\rangle. 
\end{align*}
If $\Sigma$ is invertible, then
\begin{align*}
\d2\left(\mathbf{\wh F}, N_\Sigma\right) \le \op{\Sigma^{-1}} \op{\Sigma}^{1/2} \sqrt{\sum_{i,j=1}^m \frac{1}{\sigma_i^2\sigma_j^2} \EE[(\Sigma(i,j)\sigma_i\sigma_j-\langle D F_i, -DL^{-1}F_j\rangle)^2] } \\
 + \frac{\sqrt{2\pi}}{8} \op{\Sigma^{-1}}^{3/2} \op{\Sigma} \EE \left\langle \Big(\sum_{i=1}^m \frac{|DF_i|}{\sigma_i}\Big)^2, \sum_{j=1}^m \frac{|DL^{-1}F_j|}{\sigma_j}\right\rangle. 
\end{align*}
\end{lemma}

In order to deal with the convex distance $d_c$, we will need to exploit some results related to the multidimensional Stein's method for normal approximations, that we explain in the next section.

\subsection{Stein's equations for multivariate normal approximations}

The reader is referred e.g. to \cite[Chapter 4]{NP} for full details about the content of the present section. For $m\geq 2$, let $N_\Sigma$ be a centered Gaussian random vector in $\RR^m$ with invertible covariance $\Sigma=(\Sigma(i,j))$.  Consider Stein's equation  with unknown $f$ and given test function $h:\RR^m\to\RR$:
\begin{align}\label{e:SE_d}
\sum_{i,j=1}^m \Sigma(i,j) \partial_{ij}f(\mathbf x) - \sum_{i=1}^m x_i \partial_i f(\mathbf x) = h(\mathbf x) -\EE h(N_\Sigma),
\end{align}
or, in a more compact form, $\langle\hess f(\mathbf x), \Sigma\rangle_{\mathrm{HS}} - \langle \mathbf x, \nabla f(\mathbf x)\rangle = h(\mathbf x)- \EE[h(N_\Sigma)]$.  
The left-hand-side of \eqref{e:SE_d} is the generator of the $\RR^d$-valued diffusion process given by
\begin{align*}
dX_t = - X_t dt +  (2\Sigma)^{1/2} dB_t 
\end{align*}
where $\{ B_t : t\geq 0\}$ is a standard Brownian motion in $\RR^m$. The Markov semigroup associated with $\{X_t : t\geq 0\}$ enjoys the following Mehler-type representation:
\begin{align}\label{e:MehlerGaussian}
\EE_\mathbf x[h(X_t)] = \EE[h(e^{-t}\mathbf x+ \sqrt{1-e^{-2t}}N_\Sigma)],
\end{align}
where $\EE_\mathbf x$ is the expectation with respect to the distribution of $X_t$ with initial condition $X_0=\mathbf x$.  The generator approach to multivariate normal approximation consists in using the underlying semigroup to construct solutions to \eqref{e:SE_d}.  This approach was first devised by Barbour in his study of Poisson approximation \cite{Barbour88}, and then further developed in \cite{Barbour90, Gotze91} --- see also \cite{Reinert}.  More precisely, define 
\begin{align*}
f(\mathbf x) : =  -\int_0^\infty (\EE_\mathbf x[h(X_t)]-\EE[h(N_\Sigma)]) dt = -\int_0^\infty \EE[h(e^{-t}\mathbf x+\sqrt{1-e^{-2t}} N_\Sigma)] - \EE[h(N_\Sigma)] dt,
\end{align*}
where we used \eqref{e:MehlerGaussian}.  A change of variables gives 
\begin{align}\label{e:sol0}
f(\mathbf x)= -\frac{1}{2}\int_0^1 \frac{1}{u} (\EE[h(\sqrt{u}\mathbf x+\sqrt{1-u}N_\Sigma)] -\EE[h(N_\Sigma)])du. 
\end{align}

In the case $h\in C^2$ with bounded first and second partial derivatives, the second and third partial derivatives of $f$ are uniformly bounded \cite[Lemma 2.17]{PZ10}, which makes the estimation of bounds in $d_2$ a simpler task than the non-smooth $\dc$ estimates.
In order to obtain bounds in $\dc$, we would adopt a smoothing approach. For any bounded and measurable $h$ and $t\in(0,1)$, define its mollification at level $\sqrt{t}$ by
\begin{align}\label{e:h_t}
h_t(\mathbf x) :=  \EE[h(\sqrt{t} N_\Sigma + \sqrt{1-t} \mathbf x)].
\end{align} 
We stress that $h_t$ and $f_t$ given below both depend on $\Sigma$ which we suppress from the subscript to simplify the notation. 
Then, it is plain that $h_t$ is $C^\infty$ with bounded derivatives of all orders. Hence,  the solution to \eqref{e:SE_d} with $h=h_t$ is given by (up to a change of variables)
\begin{align}\label{e:sol_mol}
f_t(\mathbf x) := -\frac{1}{2}\int_t^1 \frac{1}{1-s} (\EE[h(\sqrt{s}N_\Sigma+\sqrt{1-s}\mathbf x)]-\EE[h(N_\Sigma)]) ds.
\end{align}
Denote by $\ph_\Sigma$ the density of $N_\Sigma$. One sees that $f_t\in C^\infty$ and for $i,j,k\in [m]:=\{1,...,m\}$,
\begin{align}
\partial_i f_t(\mathbf x)& = -\frac{1}{2}\int_t^1 \frac{1}{\sqrt{s(1-s)}}\int h(\sqrt{s}\mathbf z+\sqrt{1-s}\mathbf x) \partial_i \ph_\Sigma(\mathbf z)d\mathbf z ds, \nonumber\\
\partial^2_{ij} f_t(\mathbf x)& = \frac{1}{2}\int_t^1\frac{1}{s} \int h(\sqrt{s}\mathbf z+\sqrt{1-s}\mathbf x) \partial^2_{ij}\ph_\Sigma(\mathbf z) d\mathbf z ds, \nonumber\\
\partial^3_{ijk} f_t(\mathbf x)&=  -\frac{1}{2}\int_t^1 \frac{\sqrt{1-s}}{s^{3/2}} \int  h(\sqrt{s}\mathbf z+\sqrt{1-s}\mathbf x) \partial^3_{ijk} \ph_\Sigma(\mathbf z) d\mathbf z ds.  \label{e:der3}
\end{align}
 Thus, for $i,j,k\in[m]$ and $t\in(0,1)$,  one obtains the uniform bounds
 \begin{align}
|\partial_i f_t(\mathbf x)| &\le C(m,\Sigma) \norm{h}, \nonumber\\
|\partial^2_{ij} f_t(\mathbf x) | &\le C(m,\Sigma) \norm{h} |\log t|, \label{e:unif_2nd_partial}\\
 |\partial^3_{ijk} f_t(\mathbf x) | &\le C(m,\Sigma) \norm{h} \frac{1}{\sqrt{t}}\nonumber. 
\end{align}  

The remainder of the paper is devoted to the proof of our main results.

\section{Proof of the main estimates}\label{s:proofs}

\subsection{Proof of Theorem \ref{l:bound_kol}}
By homogeneity, it suffices to consider $F$ centered with $\EE[F^2]=1$, and $\sigma^2=1$.  We denote by $f_z$ the canonical solution (see e.g. \cite[formula (3.4.1)]{NP}) to the {\bf Stein's equation for the Kolmogorov distance} at $z\in \RR$, given by
\begin{align} \label{e:Stein_kol}
f'(x) - x f(x) = 1_{(-\infty,z]}(x) - \phi(z), \quad x\in \RR,
\end{align}
where $\phi(z) =(2\pi)^{-1/2}\int_{-\infty}^z e^{-u^2/2}du $. It is a well-known fact that $f_z$ is continuously differentiable at every $x\neq z$, and one moreover adopts the standard convention $f'_z(z) :=z f_z(z) + 1 - \phi(z)$. The proof of the following properties of $f_z$ can be found e.g. in \cite[Lemma 2.3]{CGS} and \cite[Section 3.4]{NP}: 
\begin{itemize}
\item[(i)] $f_z(x)\leq \sqrt{2\pi}/4$ and $|f'_z(x)|\le 1$ for all $z,x\in\RR$;
\item[(ii)] $|xf_z(x)|\le 1$ for all $x$ and $x\mapsto xf_z(x)$ is nondecreasing for all $z\in\RR$. 
\end{itemize}
Using Lemmas \ref{l:duality} and \ref{l:L=-delta D} together with the fact that $f_z(F) \in \mathbb{D}^{1,2}$ (since $f_z$ is Lipschitz and $F\in \mathbb{D}^{1,2}$) yields that
\begin{align*}
\EE[Ff_z(F)] = \EE[L L^{-1}Ff_z(F)] = \EE[\de (-D L^{-1}F) f_z(F)]  = \EE \left[ \langle D f_z(F),  -DL^{-1} F \rangle \right].
\end{align*}
We can now rewrite the add-one cost $D_x f_z(F)= f_z(F(\eta+\de_x)) - f_z(F(\eta)) = f_z(F+D_xF) - f_z(F)$ in integral form and deduce the identity
\begin{align*}
&\EE[f_z'(F)] -\EE[Ff_z(F)] \\
&= \EE[f_z'(F)(1-\langle DF, -L^{-1}DF\rangle)] - \EE\left[ \left\langle \int_0^{DF} \big(f_z'(F+t) -f_z'(F)\big) dt, -DL^{-1}F\right\rangle \right] =:I_1+I_2.
\end{align*}
Using Stein's equation \eqref{e:Stein_kol}, we see that
\begin{align*}
\int_0^{DF}\big( f_z'(F+t)-f_z'(F)\big) dt = \int_0^{DF} \big((F+t)f_z(F+t) - Ff_z(F)\big) dt - \int_0^{DF}  (1_{F\le z} - 1_{F+t\le z})  dt
\end{align*}
Since $x\mapsto xf_z(x)$  and $x\mapsto 1_{(z,\infty)}(x)$ are  nondecreasing, if $DF\ge 0$, then the integrands of both integrals are nonnegative and  bounded from above by $(F+DF)f_z(F+DF) - Ff_z(F)= D(Ff_z(F))$ and $1_{z-DF< F\le z}=D(1_{F>z})$, respectively. Hence,  if $DF\ge 0$, 
\begin{align}\label{e1}
\left|\int_0^{DF} f_z'(F+t)-f_z'(F) dt \right|\le DF D(Ff_z(F) + 1_{F>z}).
\end{align}
Likewise, when $DF<0$, both $D(Ff_z(F))$ and $D 1_{F>z}$ are non-positive so that the upper bound in \eqref{e1} is nonnegative regardless of the sign of $DF$, and the estimate \eqref{e1} continues to hold by an analogous monotonicity argument. We now observe the following facts: (a) since $f_z$ is bounded, $F f_z(F) \in \mathbb{D}^{1,2}$ (by virtue of \eqref{e:hhgg}), and (b) $DF D1_{F>z}\geq 0$ for every $z$. Exploiting again properties (i)-(ii) of $f_z$, we now infer that  
\begin{align*}
|I_1|&\le \EE|1-\langle DF, -L^{-1}DF\rangle|, \\
|I_2| &\le \EE[\langle DFD(Ff_z(F)+1_{F>z}), |DL^{-1}F|\rangle]\\
& = \EE[(Ff_z(F) + 1_{F>z})\de(DF|DL^{-1}F|)] \le 2 \EE |\de(DF|DL^{-1}F|)|,
\end{align*}
where in the central equality we have combined \cite[Proposition 2.3]{LPS16} both with the integration by parts formula stated in Lemma \ref{l:duality} (in order to deal with $D(F f_z(F))$) and with the generalised integration by parts relation stated in \cite[Lemma 2.2]{LPS16} (in order to deal with $D1_{F>z}$, this being the exact point at which \eqref{e:hg} is needed). Using the classical Stein's bound on the Kolmogorov distance stated e.g. in \cite[Theorem 3.4.2]{NP}, one deduces that 
\begin{align*}
\dk(F, N(0,1)) &\le \sup_{z\in\RR} \| \EE[f_z'(F)] -\EE[Ff_z(F)] \| \\
&\le \EE|1-\langle DF, -L^{-1}DF\rangle| + 2 \EE |\de(DF|DL^{-1}F|)|,
\end{align*}
thus concluding the proof. \qed

\subsection{Key technical estimates}

The proofs of our main results are based on a number of technical estimates, that we gather together in the present subsection. From now on, we work in the general framework of Section \ref{ss:frame}, and use the tools of Malliavin calculus discussed in Section \ref{appendix}; recall also the notation \eqref{e:innp}, and that $H^y := H(\eta+\de_y)$ for $H\in \FNS$ and $y\in\XX$.

The following lemma will be used on several occasions.
\begin{lemma} \label{l:E[DLF^p]} 
Let $F\in L^2(\Omega)$, $y\in\XX$ and $p\ge 1$. Then, for every $x\in \XX$,
\begin{align*}
\EE[ |D_x L^{-1} F|^p] &\le \EE[|D_x F|^p], \\
2 \EE[ |(D_x L^{-1} F)^y|^p] &\le \EE[|D_x F^y|^p] +  \EE[|D_x F|^p].
\end{align*}
\end{lemma} 
\begin{proof}
We can assume that $\EE[|D_x F^y|^p] +  \EE[|D_x F|^p]<\infty$ (if not, then there is nothing to prove). The first inequality was proved in \cite[Lemma 3.4]{LPS16}. Applying Item (v) and Item (iv) of Lemma \ref{l:facts}, and reasoning as in the proof of \cite[Corollary 3.3]{LPS16} in order to exchange integrals and add-one cost operators, we infer that, a.s.-$\PP$, 
\begin{align*}
D_x L^{-1} F = - D_x \int_0^\infty T'_t F dt=  -\int_0^\infty e^{-t} T'_t D_x F dt.
\end{align*}
We now observe that, for all $\xi\in \mathbf{N}_\sigma$ and using the same notation as in \eqref{e:def_Pt}--\eqref{e:teta} for thinned measures,
\begin{align}\label{e:juste}
&\EE[ |D_x F((\eta+\de_y)_{e^{-t}}+\xi)|^p | \eta]  \\
&\quad\quad=  e^{-t} \EE[|D_x F(\eta_{e^{-t}} + \xi + \de_y)|^p | \eta]  +  (1-e^{-t}) \EE[|D_x F(\eta_{e^{-t}} + \xi )|^p | \eta], \notag
\end{align}
which yields, using again \eqref{e:teta},
$$
 \int_0^\infty e^{-t} \EE| (T'_t D_x F)(\eta+\de_y)| dt<\infty,
$$
and finally
\begin{align*}
(D_x L^{-1}F)^y = -\int_0^\infty e^{-t} (T'_t D_x F)(\eta+\de_y) dt, \quad \mbox{a.s.--}\PP.
\end{align*}
Using Jensen's inequality with respect to the probability measure $e^{-t}dt$ and \eqref{e:teta}, one deduces therefore that
\begin{align}\label{e:DLF^p}
|(D_x L^{-1} F)^y|^p \le \int_0^\infty e^{-t} \int_{\mathbf{N}_\sigma} \EE[ |D_x F((\eta+\de_y)_{e^{-t}}+\xi)|^p | \eta]  \Pi_{(1-e^{-t})\la}(d\xi) dt.
\end{align}
Taking expectation in \eqref{e:DLF^p} after having applied \eqref{e:juste} and solving the resulting integrals in $dt$ leads to the desired estimate. 
\end{proof}

The next inequality provides a useful upper bound for the variance of random variables of the type $\langle DF, -DL^{-1}F\rangle$.

\begin{proposition}\label{p:cdc}
Fix $\sB\in\mathcal{X}$ such that $\lambda(\sB)<\infty$. Let $F,G\in \FNS$ be such that $F(\sB), G(\sB)\in L^2(\Omega)$, and let $\{\sA_x : x\in\mathbb \sB\}$ be functionally measurable and such that $F(\sA_x), G(\sA_x)\in L^2(\Omega)$ for every $x$. Assume that, for some $p>4$,
\begin{align*}
\sup_{x\in\sB} \EE[ |D_x F(\sB)|^p]^{\frac 1 p} + \EE[|D_x F(\sA_x)|^p]^{\frac 1 p}+\EE[ |D_x G(\sB)|^p]^{\frac 1 p} + \EE[|D_x G(\sA_x)|^p]^{\frac 1 p} = K_1<\infty.
\end{align*}
Then, $F(\sB), L^{-1}G(\sB)\in \mathbb{D}^{1,2}$, $\langle DF(\sB), -DL^{-1}G(\sB) \rangle\in L^2(\Omega)$ and
\begin{align*}
&\Var[\langle DF(\sB), -DL^{-1}G(\sB) \rangle ]\\
& \le C \la^2(\{(x,y)\in\sB^2: \sA_x\cap\sA_y\neq \emptyset\})  \\
&\quad + C \iint_{\sB^2_\Delta} \Big( \EE[|D_x F(\sB) - D_xF(\sA_x)|]^{1-4/p} 
+ \EE[|D_x G(\sB)-D_xG(\sA_x)|]^{1-4/p} \Big) \la^2(dx,dy),
\end{align*} 
where  $C= 8 \max(1, K_1)^4$ and  $\sB^2_\Delta$ is defined in \eqref{e:bdelta}. 
\end{proposition}


\begin{proof}[Proof of Proposition \ref{p:cdc}] 
In order to simplify the notation, we write $F(\sB) = F, G(\sB) = G$, $L^{-1} G(\sB) = L^{-1} G  $, and so on. Since $\lambda(\sB), \, K_1<\infty$, one deduces immediately the first part of the statement (e.g., applying twice the Cauchy-Schwarz inequality). The finiteness of $K_1$ and $\lambda(\sB)$ can also be invoked to justify the implicit use of Fubini Theorem in the remainder of the proof (details omitted). Exploiting the bi-linearity of covariances together with Points (iv)-(v) of Lemma \ref{l:facts} gives
\begin{align*}
\Var[\langle DF, -DL^{-1}G\rangle] &= \Cov\left(\int D_x F D_xL^{-1}G\la(dx),\int D_y F D_y L^{-1}G\la(dy)\right) \\
&=\iint_{\sB^2} \Cov(D_x F D_xL^{-1}G, D_y F D_yL^{-1}G)\la^2(dx,dy).
\end{align*}
We now write $\iint_{\sB^2} = \iint_{\sB^2\backslash \sB_\Delta^2} +\iint_{\sB_\Delta^2}$ and evaluate separately the two integrals. The first integral only involves pairs $(x,y)$ such that $\sA_x\cap\sA_y\neq \emptyset$. In this case, we first use  H\"older's inequality to infer that
\begin{align*}
&| \Cov(D_x F  D_x L^{-1} G, D_y F  D_y L^{-1} G) | \\
&\le \EE|D_x F  D_x L^{-1} G D_y F  D_yL^{-1} G| + \EE|D_x F D_x L^{-1} G| \EE|D_y F  D_y L^{-1} G|\\
&\le  2(\EE[|D_x F|^4] \EE[|D_x L^{-1} G|^4] \EE[|D_y F|^4] \EE[|D_y L^{-1} G|^4])^{\frac 1 4}.
\end{align*}
Applying Lemma \ref{l:E[DLF^p]}, one therefore deduces the bound
\begin{align}\label{e:cov_bdunif}
| \Cov(D_x F  D_x L^{-1} G, D_y F  D_y L^{-1} G) |  &\le 2\sqrt{\sup_{x\in \sB} \EE[|D_x F|^4] \sup_{x\in \sB} \EE[|D_x G|^4]} \nonumber\\
&\le  \sup_{x\in \sB} \EE[|D_x F|^4] + \sup_{x\in \sB} \EE[|D_x G|^4] \le 2 K_1^4.
\end{align}
To estimate the contribution of the second integral, fix $(x,y)$ such that $\sA_x\cap\sA_y=\emptyset$. We rewrite the corresponding covariance as follows:
\begin{align}\label{e:cov}
&| \Cov(D_x F D_x L^{-1} G, D_y F D_y L^{-1} G)|\\&=  \Cov((D_x F - D_x F(\sA_x)) D_x L^{-1} G, D_y F D_y L^{-1} G)\nonumber\\
 &+ \Cov(D_x F(\sA_x) (D_x L^{-1} G - D_x L^{-1} G(\sA_x)), D_y F D_y L^{-1} G) \nonumber\\
 &+  \Cov(D_x F(\sA_x) D_x L^{-1} G(\sA_x), (D_y F - D_y F(\sA_y)) D_y L^{-1} G) \nonumber\\
 &+\Cov(D_x F(\sA_x) D_x L^{-1} G(\sA_x), D_y F(\sA_y) (D_y L^{-1} G - D_y L^{-1} G(\sA_y)))\nonumber\\
 &+\Cov(D_x F(\sA_x) D_x L^{-1} G(\sA_x), D_y F(\sA_y) D_y L^{-1} G(\sA_y)).\notag
\end{align}
Since $( D_x F(\sA_x) , D_x L^{-1} G(\sA_x))$ and $(D_y F(\sA_y) ,D_y L^{-1} G(\sA_y))$ are, respectively, measurable with respect to $\sigma(\eta |_{\sA_x})$ and with respect to $\sigma(\eta |_{\sA_y})$, they are also independent, implying that the last term on the right-hand side of \eqref{e:cov} vanishes. We now bound the first term on the right-hand side of \eqref{e:cov}, the idea  is to isolate the two-scale add-one-cost discrepancy of $F$ by using the following consequence of H\"older's inequality: for any non-negative random variables $X,Y,Z,W\in L^p(\Omega)$ with $p\in [4,\infty]$, one has that
\begin{align}\label{e:xyzw}
 \EE[XYZW] &=\EE[X^{1 - \frac 4 p} X^{\frac{4}{p}} YZW] \le \EE[X]^{1 - \frac{4}{p}}\EE[X(YZW)^{\frac{p}{4}}]^{\frac{4}{p}} \nonumber\\
 &\le \EE[X]^{1- \frac{4}{p}}  (\EE[X^4]\EE[Y^p]\EE[Z^p]\EE[W^p])^{\frac 1 p}. 
\end{align}
Combining \eqref{e:xyzw}, the trivial bound $\Cov[X_1,X_2]\le \EE[|X_1 X_2|]+\EE[|X_1|]\EE[|X_2|]$, the elementary inequality $(a+b)^4\le 8(a^4+b^4)$,  as well as Lemma \ref{l:E[DLF^p]} gives
\begin{align*}
&\Cov[(D_x F - D_x F(\sA_x))  D_x L^{-1} G, D_y F  D_y L^{-1} G ]\\
&\le  \EE[|D_x F-D_xF(\sA_x)|]^{1-\frac 4 p} \big[ (16 K_1^4 K_1^{3p})^{\frac 1 p} +  (16 K_1^4 K_1^p)^{\frac 1 p} \big] \\
&\le 4 \max(1,K_1)^{4} \EE[|D_x F-D_xF(\sA_x)|]^{1-\frac 4 p}.
\end{align*}
The remaining non-vanishing three terms on the right-hand side of \eqref{e:cov} can be dealt with in a similar way, using (when necessary) Lemma \ref{l:E[DLF^p]} in order to bypass the operator $L^{-1}$. This yields the estimate: for every $(x,y)\in \sB_\Delta$,
\begin{multline}\label{e:cov_bd2scale}
\Cov(D_x F  D_xL^{-1} G, D_y F  D_y L^{-1} G) \le 4 \max(1,K_1)^{4} \Big( \EE[|D_x F - D_xF(\sA_x)|]^{1-\frac{4}{p}} \\
+ \EE[|D_x G-D_xG(\sA_x)|]^{1-\frac{4}{p}} + \EE[|D_y F-D_yF(\sA_y)|]^{1-\frac{4}{p}} +\EE[|D_y G-D_y G(\sA_y)|]^{1-\frac{4}{p}}\Big).
\end{multline}
 Integrating the bounds \eqref{e:cov_bdunif} and \eqref{e:cov_bd2scale} over  $\sB^2\backslash \sB_\Delta^2$ and $\sB_\Delta^2$, respectively, yields the desired conclusion.
\end{proof}

The next statement provides a bound of a similar nature on the variance of Kabanov-Skorohod integrals. 

\begin{proposition}\label{p:sko} Fix $\sB\in\mathcal{U}$ such that $\lambda(\sB)<\infty$ and let $\{\sA_x : x\in\sB\}$ be functionally measurable. Let $F,G\in \FNS$ be such that $F(\sB), G(\sB), F(\sA_x), G(\sA_x)\in L^2(\Omega)$ for every $x$. Suppose that, for some $p>4$,
\begin{align*}
\sup_{x,y\in\sB} \EE[ |D_x F^y(\sB)|^p]^{\frac 1 p} + \EE[ |D_xF^y(\sA_x)|^p]^{\frac 1 p} + \EE[ |D_x G^y(\sB)|^p]^{\frac 1 p} + \EE[ |D_x G^y(\sA_x)|^p]^{\frac 1 p}:=K_2  <\infty,
\end{align*}
Then, $DF(\sB) | DL^{-1}G(\sB)|, \, DF(\sB) DL^{-1}G(\sB) \in \dom \delta$ and 
\begin{align*}
&\EE[\de(DF (\sB) |DL^{-1}G(\sB)|)^2] +\EE[\de(DF(\sB) DL^{-1}G(\sB))^2] \\
&\le C\la(\sB) +  C \la^2(\{(x,y)\in\sB^2: \sA_x\cap\sA_y\neq \emptyset\}) \\ 
&\quad +C\iint_{\sB^2_\Delta} \Big( \EE [|D_yG^x(\sB) -D_yG^x(\sA_y)|]^{1- \frac 4 p} + \EE [|D_yF^x(\sB) -D_yF^x(\sA_y)|]^{1- \frac 4 p}\Big) \la^2(dx,dy)\\
&\quad + {C\iint_{\sB^2_\Delta} \Big( \EE [|D_yG(\sB) -D_yG(\sA_y)|]^{1- \frac 4 p} + \EE [|D_yF(\sB) -D_yF(\sA_y)|]^{1- \frac 4 p}\Big) \la^2(dx,dy)}.
\end{align*}
where $C = 16 \max(1,K_2)^4$.
\end{proposition}
\begin{proof}   
To simplify, we write $F(\sB) = F, G(\sB) = G$, $L^{-1} G(\sB) = L^{-1} G,  F^y(\sB) = F^y, G^y(\sB) = G^y  $, and so on. We only prove the bound on $\EE[\de(DF (\sB) |DL^{-1}G(\sB)|)^2]$; the second summand on the left-hand side of the bound can be dealt with along the same lines. Since $K_2, \lambda(\sB)<\infty$, one checks immediately (with the help of Lemma \ref{l:E[DLF^p]} ) that $u(x) := D_xF | D_xL^{-1}G|$  verifies the assumptions of Lemma \ref{l:SkorohodVariance}, which implies in particular that $DF | DL^{-1}G| \in \dom \delta$ and
\begin{align*}
&\EE[\de(DF|DL^{-1}G|)^2]\\
&= \EE \int_\sB |D_x F D_xL^{-1}G|^2 \la(dx) + \EE \iint_{\sB^2} D_x(D_yF|D_y L^{-1} G|) D_y(D_x F |D_xL^{-1} G|) \la^2(dx,dy) \\
&=: J_1+J_2,
\end{align*}
where the restricted domains of integration are justified by the fact that, for all $H\in L^0(\Omega)$, one has that $
 D_x H(\sB) = 0$ whenever $x\not\in \sB$ (similar facts are exploited without mention throughout the proof). We also recall the relation \eqref{e:claro}, and subdivide the proof into several steps.

\smallskip

\noindent\underline{\it Step 1: Bounding $J_1$.} By H\"older's inequality and Lemma \ref{l:E[DLF^p]}, we have
\begin{align*}
\EE[ |D_x F D_xL^{-1}G|^2] \le \EE[|D_x F|^4]^{\frac 1 2} \EE[|DL^{-1}G|^4]^{\frac 1 2} \le K_2^{4}.
\end{align*}
leading to $J_1\le K_2^{4}\la(\sB)$.

\medskip

\noindent\underline{\it Step 2: On-diagonal contribution to $J_2$.} Denote the integrand in $J_2$ by
$$\scr R = \scr R(x,y)  :=D_x(D_yF|D_y L^{-1} G|) D_y(D_x F |D_xL^{-1} G|)$$ 
We first consider the case where $(x,y)$ is such that $\sA_x\cap\sA_y\neq \emptyset$ and use the crude bound $|D_xH| \le |H^x|+|H|$ for any $H\in L^0(\Omega)$ and $x\in\XX$ to obtain
\begin{align*}
|\scr R|\le \big( |D_yF^x  (D_y L^{-1} G )^x| + |D_yF  D_y L^{-1} G| \big)  \big( |D_x F^y (D_xL^{-1} G)^y | + |D_x F  D_xL^{-1} G|   \big).
\end{align*}
We now assess the expectation of the first among the four products. By H\"older's inequality and Lemma \ref{l:E[DLF^p]}, one has that 
\begin{align*}
&\EE |D_yF^x (D_y L^{-1} G)^x   D_x F^y (D_xL^{-1} G)^y | \\
&\le \Big(\EE[ |D_y F^x|^4] \EE[ |(D_y L^{-1}G)^x|^4] \EE[ |D_x F^y|^4] \EE[ |(D_x L^{-1} G)^y|^4]\Big)^{1/4}\le K_2^4.
\end{align*}
 Handling the expectation of remaining three terms by analogous arguments, implies the bound  $\EE[ |\scr R(x,y) | ]\le 4K_2^4$, for all $x,y\in \sB$.
  
  \smallskip
\noindent\underline{\it Step 3: Off-diagonal contribution to $J_2$.} Now we consider $(x,y)$ such that $\sA_x\cap\sA_y=\emptyset$.  We can thus write
\begin{align*}
\scr R = \scr R(x,y)  
&=D_x\Big( (D_yF-D_yF(\sA_y)) |D_y L^{-1} G|\Big) D_y(D_x F |D_xL^{-1} G|) \\
&\quad + D_x\Big(D_yF(\sA_y) (|D_y L^{-1} G| -|D_yL^{-1}G(\sA_y)|) \Big) D_y(D_x F |D_xL^{-1} G|)\\
&=(D_y F^x - D_y F^x({\sA_y})) |(DL^{-1}G)^x| D_y(D_xF |D_x L^{-1}G|) \\
&\quad- (D_y F - D_y F(\sA_y) ) |DL^{-1}G| D_y(D_xF |D_x L^{-1}G|) \\
&\quad+ D_y F^x({\sA_y})( |( D_y L^{-1}G)^x|-  |(D_yL^{-1}G({\sA_y}))^x| ) D_y(D_x F |D_xL^{-1} G|) \\
&\quad- D_y F(\sA_y)(|D_yL^{-1}G| - |D_yL^{-1}G(\sA_y)|) D_y(D_x F |D_xL^{-1} G|).
\end{align*}
We estimate the expectation of the third term appearing after the last equality, namely $$\scr J = \scr J(x,y) :=\EE D_y F^x({\sA_y})( |( D_y L^{-1}G)^x|-  |(D_yL^{-1}G({\sA_y}))^x| ) D_y(D_x F |D_xL^{-1} G|),$$ the other summands being dealt with by a slight variation of the same argument. Writing
\begin{align*}
D_y(D_x F |D_xL^{-1} G|) = D_x F^y |(D_x L^{-1}G)^y| - D_x F|D_xL^{-1}G|,
\end{align*}
 then applying \eqref{e:xyzw} to 
\begin{align*}
X&= |(D_y L^{-1}G)^x-  (D_yL^{-1}G({\sA_y}))^x|,  \\
Y&= |( D_y F({\sA_y}))^x|,\\
Z&= |D_x F^y| + |D_xF|,\\
W&=|(D_x L^{-1}G)^y| + |D_xL^{-1}G|,
\end{align*} 
 we arrive at 
\begin{align*}
|\scr J|\le \EE[|( D_y L^{-1}G)^x -  ( D_yL^{-1}G({\sA_y})^x|]^{1-\frac 4 p}
\cdot (\EE[X^4]\EE[Y^p]\EE[Z^p]\EE[W^p])^{1/p}.
\end{align*}
Applying Lemma \ref{l:E[DLF^p]} to $F=G-G(\sA_y)$ and $p=1$, we infer that 
\begin{align*}
2\EE[|( D_y L^{-1}G)^x -  ( D_yL^{-1}G({\sA_y}))^x|] \le \EE[|D_y G^x-  D_yG^x({\sA_y})|] +{ \EE[|D_y G-  D_yG({\sA_y})|]}.
\end{align*}
A further application of Lemma \ref{l:E[DLF^p]} also shows that $\EE[X^4]\le 16K_2^4$, $\EE[Y^p]\le K_2^p$ and $\EE[Z^p], \EE[W^p]\le 16 K_2^p$. Hence,
\begin{align*}
|\scr J(x,y)|\le 8 \max(1,K_2)^4 \Big( \EE[|D_y G^x-  D_yG^x({\sA_y})|]^{1-\frac{4}{p}}+ \EE[|D_y G-  D_yG({\sA_y})|]^{1-\frac{4}{p}} \Big).
\end{align*}
After assessing the remaining terms decomposing $\scr R(x,y)$, we deduce that there exists a finite constant $C$ that depends on $K_2$ such that
\begin{align*}
&|\EE \scr R(x,y)| \notag\le 16 \max(1,K_2)^4 \Big(\EE[|D_y G^x-  D_yG^x({\sA_y})|]^{1-\frac{4}{p}} +  \EE[|D_y F^x-  D_y F^x({\sA_y})|]^{1-\frac{4}{p}} \\
&\hspace{4.5cm} + \EE[|D_y G-  D_yG({\sA_y})|]^{1-\frac{4}{p}} +  \EE[|D_y F-  D_y F({\sA_y})|]^{1-\frac{4}{p}}\Big).
\end{align*}
Integrating the estimates at Step 2 and 3 over $\sB^2\backslash \sB^2_\Delta$ and $\sB_\Delta$, respectively, and taking into account the content of Step 1 yields the desired conclusion.

%
\end{proof}

The following result can be regarded as a refinement of Proposition \ref{p:cdc},  displaying an additional indicator, as well as an integral to evaluate. In particular, in contrast to the situation of Propositions \ref{p:cdc} and \ref{p:sko}, the presence of the indicator requires us to pay careful attention to the scaling and centering of the considered random variables.


\begin{proposition}
\label{p:p3}
Fix $\sB\in\mathcal{X}$ such that $\lambda(\sB)<\infty$ and consider a functionally measurable collection of sets $\{\sA_x : x\in \sB\}$. Let $\mathbf{F}=(F_1,...,F_m)$ a vector of elements of $L^0(\Omega)$ such that, for each $i\in [m]$ and every $x\in \sB$, $F_i(\sB), F_i(\sA_x)\in L^2(\Omega)$. Set $\wh F_i (\sB) =(F_i(\sB)-\EE F_i(\sB))/\sigma_i$ with $\sigma_i\in(0,\infty)$, $\forall i\in[m]$ and write $\mathbf{\wh F}(\sB) =(\wh F_1(\sB),...,\wh F_m(\sB))$. Suppose that for some $p>6$, one has 
\begin{align*}
\sup_{\ell\in[m]}\sup_{x\in\sB} \EE[ |D_x F_\ell(\sB)|^p]^{\frac 1 p} + \EE[|D_x F_\ell(\sA_x)|^p]^{\frac 1 p}= K_3<\infty.
\end{align*}
Then for any $j,k\in[m]$,
\begin{align*}
&\int_0^1 \Var[\langle 1_{w\le \norm{D\mathbf{\wh F}}}, |D_x \wh F_j(\sB) D_x L^{-1} \wh F_k(\sB)|\rangle] dw \\
&\le \sum_{i=1}^m \frac{C}{\sigma_i\sigma_j^2\sigma_k^2} \Big(\la^2\{(x,y)\in\sB^2: \sA_x\cap\sA_y\neq \emptyset\} + \iint_{\sB^2_\Delta} \sup_{\ell\in[m]} \EE|D_x F_\ell(\sB) - D_x F_\ell(\sA_x)|]^{1-\frac{5}{p}}\la^2(dx,dy)\Big).
\end{align*}
and
\begin{align*}
&\int_0^1 w \Var[\langle 1_{w\le \norm{D\mathbf{\wh F}}}, |D_x \wh F_j (\sB)D_x L^{-1} \wh F_k(\sB)|\rangle] dw\\
&\le  \sum_{i=1}^m \frac{C}{\sigma^2_i\sigma_j^2\sigma_k^2} \Big(\la^2\{(x,y)\in\sB^2: \sA_x\cap\sA_y\neq \emptyset\} + \iint_{\sB^2_\Delta} \sup_{\ell\in[m]} \EE|D_x F_\ell(\sB) - D_x F_\ell(\sA_x)|]^{1-\frac{6}{p}}\la^2(dx,dy)\Big).
\end{align*}
where $C= 24 \max(1,K_3)^6$.
\end{proposition}

\begin{proof}
As for the other proofs in this section, we will remove all dependencies on $\sB$ from the considered random elements, in order to simplify the notation.
Denote the left-hand side of the two inequalities in the statement by $V^{(1)}_{jk}$ and $V^{(2)}_{jk}$ respectively. We rewrite them in terms of covariances,  as follows
\begin{align*}
V^{(1)}_{jk}&=\iint_{\sB^2} \int_0^1 \Cov[ 1_{w\le \norm{D_x\mathbf F}} |D_x \wh F_j D_x L^{-1}\wh F_k|,1_{w\le \norm{D_y\mathbf F}}|D_y \wh F_j D_y L^{-1}\wh F_k|] dw\la^2(dx,dy),\\
V^{(2)}_{jk}&= \iint_{\sB^2} \int_0^1 w\,\Cov[1_{w\le \norm{D_x\wh{\mathbf F}}} |D_x \wh F_j D_x L^{-1}\wh F_k|,1_{w\le \norm{D_y\mathbf F}}|D_y \wh F_j D_y L^{-1}\wh F_k|] dw\la^2(dx,dy).
\end{align*}
As previously, we subdivide the proof into several steps, according to the relative position of the coordinates of $(x,y)\in \sB^2$. 

\noindent\underline{\it Case 1: $\sA_x\cap\sA_y \neq \emptyset$.} We invoke the trivial bound $\Cov[X,Y]\le \EE[|XY|]+\EE|X|\EE|Y|$ to deduce that
\begin{align*}
&\Cov[1_{w\le \norm{D_x\wh{\mathbf F|}}} |D_x \wh F_j D_x L^{-1}\wh F_k|, 1_{w\le \norm{D_y\wh{\mathbf F}}}|D_y \wh F_j D_y L^{-1}\wh F_k|] \\
&\le \EE [1_{w\le \norm{D_x\wh{\mathbf F}}\wedge \norm{D_y\wh{\mathbf F}}}  |D_x \wh F_j D_x L^{-1}\wh F_k D_y \wh F_j D_y L^{-1}\wh F_k|] \\
&+ \EE[1_{w\le \norm{D_x\wh{\mathbf F}}} |D_x \wh F_j D_x L^{-1}\wh F_k|] \EE[1_{w\le \norm{D_y\wh{\mathbf F}}}|D_y \wh F_j D_y L^{-1}\wh F_k|].
\end{align*}
Bounding the last $1_{w\le \norm{D_y \wh{\mathbf F}}}$ term by $1$, then integrating with respect to $w$ over $[0,1]$, one obtains
\begin{align*}
&\int_0^1 \Cov[1_{w\le \norm{D_x\wh{\mathbf F}}} |D_x \wh F_j D_x L^{-1}\wh F_k|, 1_{w\le \norm{D_y\wh{\mathbf F}}}|D_y \wh F_j D_y L^{-1}\wh F_k|] dw \\
& \le \EE\Big[\norm{D_x\wh{\mathbf F}}\wedge \norm{D_y\wh{\mathbf F}} |D_x \wh F_j D_x L^{-1}\wh F_k D_y \wh  F_j D_y L^{-1}\wh F_k|\Big]\\
&\hspace{4cm} + 
\EE\Big[\norm{D_x\wh{\mathbf F}}|D_x \wh F_j D_x L^{-1}\wh F_k|]\EE[|D_y \wh F_j D_y L^{-1}\wh F_k|]\Big]\\
&\le \sum_{i=1}^m \frac{1}{\sigma_i\sigma_j^2\sigma_k^2} \EE[|D_xF_i D_x F_j D_x L^{-1}F_k D_y F_j D_yL^{-1}F_k|] \\
&\hspace{4cm}+  \sum_{i=1}^m \frac{1}{\sigma_i\sigma_j^2\sigma_k^2} \EE[|D_x F_i D_x F_j D_x L^{-1}F_k|]\EE[|D_yF_j D_yL^{-1}F_k|].
\end{align*}
Applying Lemma  \ref{l:E[DLF^p]}  as we did in the proof of Proposition \ref{p:sko} gives 
\begin{align}\label{e:p6.4_1}
&\int_0^1 \Cov[1_{w\le \norm{D_x\wh{\mathbf F}}} |D_x \wh F_j D_x L^{-1}\wh F_k|, 1_{w\le \norm{D_y\wh{\mathbf F}}}|D_y \wh F_j D_y L^{-1}\wh F_k|] dw  \le \sum_{i=1}^m \frac{2  K_3^5}{\sigma_i\sigma_j^2\sigma_k^2}.
\end{align}
Similarly,
\begin{align*}
&\int_0^1 w\,\Cov[1_{w\le \norm{D_x\wh{\mathbf F}}} |D_x \wh F_j D_x L^{-1}\wh F_k|, 1_{w\le \norm{D_y\wh{\mathbf F}}}|D_y\wh F_j D_y L^{-1}\wh F_k|] dw \\
& \le \frac{1}{2}\EE\Big[\norm{D_x\wh{\mathbf F}}^2\wedge \norm{D_y\wh{\mathbf F}}^2 |D_x \wh F_j D_x L^{-1}\wh F_k D_y \wh F_j D_y L^{-1}\wh F_k|\Big] \\
& \hspace{4cm}+
\frac{1}{2}\EE\Big[\norm{D_x\wh{\mathbf F}}^2|D_x \wh F_j D_x L^{-1}\wh F_k|]\EE[|D_y \wh F_j D_y L^{-1}\wh F_k|]]\\
&\le \frac{1}{2}\sum_{i=1}^m \frac{1}{\sigma^2_i\sigma_j^2\sigma_k^2} \EE[|D_xF_i|^2| D_x F_j D_x L^{-1}F_k D_y F_j D_yL^{-1}F_k|] \\
&\hspace{4cm} +\frac{1}{2}\sum_{i=1}^m \frac{1}{\sigma_i^2\sigma_j^2\sigma_k^2} \EE[|D_x F_i|^2| D_x F_j D_x L^{-1}F_k|]\EE[|D_yF_j D_yL^{-1}F_k|]\\
&\le \sum_{i=1}^m \frac{K_3^6}{\sigma_i^2\sigma_j^2\sigma_k^2} 
\end{align*}

\noindent\underline{\it Case 2:  $\sA_x\cap\sA_y=\emptyset$.} Setting $\wh F_i(\sA_x):= F_i(\sA_x)/\sigma_i$ for $i\in[m]$ and $\wh{\mathbf F}(\sA_x)=(\wh F_1(\sA_x),...,\wh F_m(\sA_x) )$, one has by independence
\begin{align*}
&\Cov[ 1_{w\le \norm{D_x\wh{\mathbf F}}} |D_x \wh F_j D_x L^{-1}\wh F_k|,1_{w\le \norm{D_y\wh{\mathbf F}}}|D_y\wh F_j D_y L^{-1}\wh F_k|] \\
&= \Cov[(1_{w\le \norm{D_x\wh{\mathbf F}}}- 1_{w\le \norm{D_x\wh{\mathbf F}(\sA_x)}})|D_x \wh F_j D_x L^{-1}\wh F_k|,1_{w\le \norm{D_y\wh{\mathbf F}}}|D_y\wh F_j D_y L^{-1}\wh F_k|] \\
&+ \Cov[ 1_{w\le \norm{D_x\wh{\mathbf F}(\sA_x)}}(|D_x \wh F_j| - |D_x \wh F_j(\sA_x)|) |D_xL^{-1}\wh F_k|, 1_{w\le \norm{D_y\wh{\mathbf F}}}|D_y\wh F_j D_y L^{-1}\wh F_k|] \\
&+ \Cov[ 1_{w\le \norm{D_x\wh{\mathbf F}(\sA_x)}} |D_x {\wh F}_j(\sA_x)|(|D_xL^{-1}\wh F_k| - |D_xL^{-1}\wh F_k(\sA_x)|),1_{w\le \norm{D_y\wh{\mathbf F}}}|D_y \wh F_j D_yL^{-1}\wh F_k|] \\
&+ \Cov[ 1_{w\le \norm{D_x\wh{\mathbf F}(\sA_x)}} |D_x {\wh F}_j(\sA_x)D_xL^{-1}\wh F_k(\sA_x)|,(1_{w\le \norm{D_y\wh{\mathbf F}}}- 1_{w\le \norm{D_y\wh{\mathbf F}(A_y)}})|D_y \wh F_j D_yL^{-1}\wh F_k| ]\\
&+ \Cov[1_{w\le \norm{D_x\wh{\mathbf F}(\sA_x)}} |D_x {\wh F}_j(\sA_x)D_xL^{-1}\wh F_k(\sA_x)|, 1_{w\le \norm{D_y\wh{\mathbf F}(A_y)}} (|D_y \wh F_j| - |D_y \wh F_j(A_y)|) |D_yL^{-1}\wh F_k|] \\
&+ \Cov[1_{w\le \norm{D_x\wh{\mathbf F}(\sA_x)}} |D_x {\wh F}_j(\sA_x)D_xL^{-1}\wh F_k(\sA_x)|, 1_{w\le \norm{D_y\wh{\mathbf F}(A_y)}} |D_y \wh F_j(A_y)| (|D_yL^{-1}\wh F_k|-|D_yL^{-1}\wh F_k(A_y)|)].
\end{align*}
Let us denote the summands in the above display by $\frak C_1,..., \frak C_6$. Note that 
\begin{align*}
|1_{w\le \norm{D_x\wh{\mathbf F}}} - 1_{w\le \norm{D_x\wh{\mathbf F}(\sA_x)}}|= 1_{\norm{D_x\wh{\mathbf F}(\sA_x)} <w\le \norm{D_x\wh{\mathbf F}}} + 1_{\norm{D_x\wh{\mathbf F}} <w\le \norm{D_x\wh{\mathbf F}(\sA_x)}}.
\end{align*}
This shows 
\begin{align*}
|\frak C_1| &\le \EE[1_{\norm{D_x\wh{\mathbf F}(\sA_x)} <w\le \norm{D_x\wh{\mathbf F}}}|D_x\wh F_j D_x L^{-1}\wh F_k D_y \wh F_j D_yL^{-1}\wh F_k|] \\
& +  \EE[1_{\norm{D_x\wh{\mathbf F}} < w \le \norm{D_x\wh{\mathbf F}(\sA_x)}} |D_x\wh F_j D_x L^{-1}\wh F_k D_y \wh F_j D_yL^{-1}\wh F_k|] \\
&+ \EE[ 1_{\norm{D_x\wh{\mathbf F}(\sA_x)} <w\le \norm{D_x\wh{\mathbf F}}}|D_x\wh F_j D_x L^{-1}\wh F_k|]\EE[|D_y \wh F_j D_yL^{-1}\wh F_k|]\\
&+ \EE[ 1_{\norm{D_x\wh{\mathbf F}} < w \le \norm{D_x\wh{\mathbf F}(\sA_x)}}|D_x\wh F_j D_x L^{-1}\wh F_k|]\EE[|D_y \wh F_j D_yL^{-1}\wh F_k|].
\end{align*}
Similarly,
\begin{align*}
|\frak C_2| &\le \EE[1_{w\le \norm{D_x \wh{\mathbf F}(\sA_x)}\wedge \norm{D_y\wh{\mathbf F}}}|D_x\wh F_j - D_x {\wh F}_j(\sA_x)| |D_x L^{-1}\wh F_k D_y \wh F_j D_yL^{-1}\wh F_k|] \\
&+ \EE[1_{w\le \norm{D_x \wh{\mathbf F}(\sA_x)}} |D_x\wh F_j - D_x {\wh F}_j(\sA_x)||D_x L^{-1}\wh F_k|]\EE[|D_y \wh F_j D_yL^{-1}\wh F_k|].
\end{align*}
We handle $\frak C_4$ in the same way as we did for $\frak C_1$, and $\frak C_3,\frak C_5, \frak C_6$ the same as $\frak C_2$. We omit the details to avoid repetitions.  Integrating with respect to $w$ over $[0,1]$ gives
\begin{align*}
&\int_0^1 \Cov [1_{w\le \norm{D_x\wh{\mathbf F}}} |D_x\wh F_j D_x L^{-1}\wh F_k|,1_{w\le \norm{D_y\wh{\mathbf F}}}|D_y \wh F_j D_yL^{-1}\wh F_k|] dw\\
&\le 2 \EE[\norm{D_x\wh{\mathbf F} - D_x \wh{\mathbf F}(\sA_x)} |D_x\wh F_j D_x L^{-1}\wh F_k D_y \wh F_j D_yL^{-1}\wh F_k|] \\
&+ 2 \EE[\norm{D_x\wh{\mathbf F} - D_x \wh{\mathbf F}(\sA_x)} |D_x\wh F_j D_x L^{-1}\wh F_k|]\EE[|D_y \wh F_j D_yL^{-1}\wh F_k|]\\
&+\EE[ \norm{D_x \wh{\mathbf F}(\sA_x)}\wedge \norm{D_y\wh{\mathbf F}}|D_x\wh F_j - D_x {\wh F}_j(\sA_x)| |D_x L^{-1}\wh F_k D_y \wh F_j D_yL^{-1}\wh F_k|] \\
&+ \EE[ \norm{D_x \wh{\mathbf F}(\sA_x)}|D_x\wh F_j - D_x {\wh F}_j(\sA_x)||D_x L^{-1}\wh F_k|]\EE[|D_y \wh F_j D_yL^{-1}\wh F_k|] \\
&+ \EE[ \norm{D_x \wh{\mathbf F}(\sA_x)}\wedge \norm{D_y\wh{\mathbf F}} |D_x L^{-1}\wh F_k - D_x L^{-1} F_k(\sA_x)| |D_x {\wh F}_j(\sA_x) D_y \wh F_j D_yL^{-1}\wh F_k|]\\
&+ \EE[ \norm{D_x \wh{\mathbf F}(\sA_x)} |D_x L^{-1}\wh F_k - D_x L^{-1} F_k(\sA_x)||D_x {\wh F}_j(\sA_x)|]\EE[|D_y \wh F_j D_yL^{-1}\wh F_k|] \\
&+ 2 \EE[ |D_x {\wh F}_j(\sA_x)D_xL^{-1}\wh F_k(\sA_x) D_y \wh F_j D_yL^{-1}\wh F_k| \norm{D_y\wh{\mathbf F} - D_y\wh{\mathbf F}(A_y)}] \\
&+2 \EE[|D_x {\wh F}_j(\sA_x)D_xL^{-1}\wh F_k(\sA_x)|] \EE[|D_y \wh F_j D_yL^{-1}\wh F_k|\norm{D_y\wh{\mathbf F} - D_y\wh{\mathbf F}(A_y)}]\\
&+ \EE[ \norm{D_x \wh{\mathbf F}(\sA_x)}\wedge \norm{D_y\wh{\mathbf F}(A_y)}|D_x {\wh F}_j(\sA_x) D_x L^{-1}\wh F_k(\sA_x)D_yL^{-1}\wh F_k| |D_y \wh F_j - D_y \wh F_j(A_y)|] \\
&+ \EE[\norm{D_x \wh{\mathbf F}(\sA_x)} |D_x {\wh F}_j(\sA_x) D_x L^{-1}\wh F_k(\sA_x)|]\EE[|D_y \wh F_j - D_y \wh F_j(A_y)| |D_yL^{-1}\wh F_k|]\\
&+ \EE[ \norm{D_x \wh{\mathbf F}(\sA_x)}\wedge \norm{D_y\wh{\mathbf F}(A_y)}|D_x {\wh F}_j(\sA_x) D_x L^{-1}\wh F_k(\sA_x)D_y \wh F_j(A_y)| |D_yL^{-1}\wh F_k - D_yL^{-1}\wh F_k(A_y)|]\\
&+ \EE[\norm{D_x \wh{\mathbf F}(\sA_x)} |D_x {\wh F}_j(\sA_x) D_x L^{-1}\wh F_k(\sA_x)|]\EE[|D_y \wh F_j(A_y)||D_yL^{-1}\wh F_k -D_yL^{-1}\wh F_k(A_y)|].
\end{align*}
Applying H\"older's inequality as in the proof of Proposition \ref{p:cdc} and Lemma \ref{l:E[DLF^p]}  yields that
\begin{align*}
&\EE[\norm{D_x\wh{\mathbf F} - D_x \wh{\mathbf F}(\sA_x)} |D_x\wh F_j D_x L^{-1}\wh F_k D_y \wh F_j D_yL^{-1}\wh F_k|]\\
&\le \sum_{i=1}^m \frac{2 \max(1,K_3)^5}{\sigma_i\sigma_j^2\sigma_k^2} \EE|D_x F_i - D_x F_i(\sA_x)|]^{1-\frac{5}{p}}.
\end{align*}
Handling analogously the remaining 11 terms, leads to the following estimate: for any $p>5$, 
\begin{align}\label{e:p6.4_2}
&\int_0^1 \Cov[ 1_{w\le \norm{D_x\wh{\mathbf F}}} |D_x\wh F_j D_x L^{-1}\wh F_k|,1_{w\le \norm{D_y\wh{\mathbf F}}}|D_y \wh F_j D_yL^{-1}\wh F_k| ] dw \notag\\
&\le \sum_{i=1}^m \frac{24 \max(1,K_3)^5}{\sigma_i\sigma_j^2\sigma_k^2} \sup_{\ell\in[m]} \EE|D_x F_\ell - D_x F_\ell(\sA_x)|]^{1-\frac{5}{p}}.
\end{align}
Similarly, writing $a^2-b^2=(a+b)(a-b)$, one has that, for any $p>6$, 
\begin{align*}
&\int_0^1 w\,\Cov[1_{w\le \norm{D_x\wh{\mathbf F}}} |D_x\wh F_j D_x L^{-1}\wh F_k|,1_{w\le \norm{D_y\wh{\mathbf F}}}|D_y \wh F_j D_yL^{-1}\wh F_k|] dw\\
&\le 2 \EE[\norm{D_x\wh{\mathbf F} - D_x \wh{\mathbf F}(\sA_x)}(\norm{D_x\wh{\mathbf F}}+\norm{D_x \wh{\mathbf F}(\sA_x)}) |D_x\wh F_j D_x L^{-1}\wh F_k D_y \wh F_j D_yL^{-1}\wh F_k|] \\
&+ 2 \EE[\norm{D_x\wh{\mathbf F} - D_x \wh{\mathbf F}(\sA_x)}(\norm{D_x\wh{\mathbf F}}+\norm{D_x \wh{\mathbf F}(\sA_x)}) |D_x\wh F_j D_x L^{-1}\wh F_k|]\EE[|D_y \wh F_j D_yL^{-1}\wh F_k|]\\
&+\EE[ \norm{D_x \wh{\mathbf F}(\sA_x)}^2\wedge \norm{D_y\wh{\mathbf F}}^2|D_x\wh F_j - D_x {\wh F}_j(\sA_x)| |D_x L^{-1}\wh F_k D_y \wh F_j D_yL^{-1}\wh F_k|] \\
&+ \EE[ \norm{D_x \wh{\mathbf F}(\sA_x)}^2|D_x\wh F_j - D_x {\wh F}_j(\sA_x)||D_x L^{-1}\wh F_k|]\EE[|D_y \wh F_j D_yL^{-1}\wh F_k|] \\
&+ \EE[ \norm{D_x \wh{\mathbf F}(\sA_x)}^2\wedge \norm{D_y\wh{\mathbf F}} |D_x L^{-1}\wh F_k - D_x L^{-1} F_k(\sA_x)| |D_x {\wh F}_j(\sA_x) D_y \wh F_j D_yL^{-1}\wh F_k|]\\
&+ \EE[ \norm{D_x \wh{\mathbf F}(\sA_x)}^2 |D_x L^{-1}\wh F_k - D_x L^{-1} F_k(\sA_x)||D_x {\wh F}_j(\sA_x)|]\EE[|D_y \wh F_j D_yL^{-1}\wh F_k|] \\
&+ 2 \EE[ |D_x {\wh F}_j(\sA_x)D_xL^{-1}\wh F_k(\sA_x) D_y \wh F_j D_yL^{-1}\wh F_k| \norm{D_y \wh{\mathbf F} - D_y \wh{\mathbf F}(A_y)}(\norm{D_y\wh{\mathbf F}}+\norm{D_y \mathbf F(A_y)})] \\
&+2 \EE[|D_x {\wh F}_j(\sA_x)D_xL^{-1}\wh F_k(\sA_x)|] \EE[|D_y \wh F_j D_yL^{-1}\wh F_k|\norm{D_y \wh{\mathbf F} - D_y \wh{\mathbf F}(A_y)}(\norm{D_y\wh{\mathbf F}}+\norm{D_y \mathbf F(A_y)})]\\
&+ \EE[ \norm{D_x \wh{\mathbf F}(\sA_x)}^2\wedge \norm{D_y\wh{\mathbf F}(A_y)}^2|D_x {\wh F}_j(\sA_x) D_x L^{-1}\wh F_k(\sA_x)D_yL^{-1}\wh F_k| |D_y \wh F_j - D_y \wh F_j(A_y)|] \\
&+ \EE[\norm{D_x \wh{\mathbf F}(\sA_x)}^2 |D_x {\wh F}_j(\sA_x) D_x L^{-1}\wh F_k(\sA_x)|]\EE[|D_y \wh F_j - D_y \wh F_j(A_y)| |D_yL^{-1}\wh F_k|]\\
&+ \EE[ \norm{D_x \wh{\mathbf F}(\sA_x)}^2\wedge \norm{D_y\wh{\mathbf F}(A_y)}^2|D_x {\wh F}_j(\sA_x) D_x L^{-1}\wh F_k(\sA_x)D_y \wh F_j(A_y)| |D_yL^{-1}\wh F_k - D_yL^{-1}\wh F_k(A_y)|]\\
&+ \EE[\norm{D_x \wh{\mathbf F}(\sA_x)}^2 |D_x {\wh F}_j(\sA_x) D_x L^{-1}\wh F_k(\sA_x)|]\EE[|D_y \wh F_j(A_y)||D_yL^{-1}\wh F_k -D_yL^{-1}\wh F_k(A_y)|]\\
&\le \sum_{i=1}^m \frac{24 \max(1,K_3)^6}{\sigma_i^2\sigma_j^2\sigma_k^2} \sup_{\ell\in[m]} \EE|D_x F_\ell - D_x F_\ell(\sA_x)|]^{1-\frac{6}{p}}.
\end{align*}

\noindent\underline{\it Conclusion.}  Integrating \eqref{e:p6.4_1} and \eqref{e:p6.4_2} over $\sB^2\setminus\sB^2_\Delta$ and $\sB^2_\Delta$, respectively, gives the estimate for $V^{(1)}_{jk}$. The estimate of $V^{(2)}_{jk}$ follows the same line.
\end{proof}

\subsection{Proof of Theorem \ref{t:bound_w}}

{\bf (i)} For the Wasserstein bound \eqref{e:a1w}, we apply Lemma \ref{l:bound_w} to $F(\sB)$. Since we let $\sigma^2=\Var[F(\sB)]$, the first term in Lemma \ref{l:bound_w} vanishes. We  see by an application of  Lemma \ref{l:duality} that $$\EE[\langle DF(\sB), -DL^{-1} F(\sB)\rangle ] = \Var[F(\sB)].$$ Therefore, applying the Cauchy-Schwarz inequality shows that the second term in Lemma \ref{l:bound_w} is bounded from above by
\begin{align*}
\frac{1}{\sigma^2}\sqrt{\Var[\langle DF(\sB), -DL^{-1}F(\sB)\rangle]}. 
\end{align*}
Applying Proposition \ref{p:cdc}  to $F=G=F(\sB)$ gives the first two terms in \eqref{e:a1w}. Moreover, by H\"older's inequality and Lemma \ref{l:E[DLF^p]}, we have
\begin{align}
 \EE \langle |DF(\sB)|^2,|DL^{-1}F(\sB)|\rangle \le  \int_\sB \EE[|D_xF(B)|^3] \la(dx) \le K_1^3 \la(\sB), \label{e:i3}
\end{align}
finishing the proof. 

\smallskip

\noindent{\bf (ii)} For the Kolmogorov bound \eqref{e:ak}, we apply Theorem \ref{l:bound_kol} to $F(\sB)$. It suffices to consider the last term in Theorem \ref{l:bound_kol}, which, by the Cauchy-Schwarz inequality, is bounded from above by
\begin{align*}
\frac{2}{\sigma^2} \sqrt{\EE[|\de(DF(\sB)|DL^{-1}F(\sB)|)|^2]}.
\end{align*}
Applying Proposition \ref{p:sko} to $F=G=F(\sB)$ ends the proof.

%
%

\subsection{Proof of Theorem \ref{t:d_2}}

The proof is almost identical to that of Theorem \ref{t:bound_w}, with the use of Lemma \ref{l:d_2} in place of Lemma \ref{l:bound_w}. For any $i,j\in[m]$, by Lemma \ref{l:duality}, $\EE[\langle DF_i(\sB), -DL^{-1}F_j(\sB)]=\Cov[F_i(\sB),F_j(\sB)]$. By the triangle inequality, for all $a\in\RR$, 
\begin{align*}
\sqrt{\EE[(a-\langle DF_i(\sB), -DL^{-1}F_j(\sB)\rangle)^2]}\le  |a-\Cov[F_i(\sB),F_j(\sB)]| + \sqrt{\Var[\langle DF_i, -DL^{-1}F_j\rangle]} .
\end{align*}
Choosing $a=\Sigma(i,j)\sigma_i\sigma_j$, then applying Proposition \ref{p:cdc} to $F=F_i,G=F_j$ gives $c(\ga_1+\ga_2^{4,p}+\ga_3)$. The last term $c\ga_4$ arises if one applies as previously H\"older's inequality and Lemma \ref{l:E[DLF^p]}  to the second term in Lemma \ref{l:d_2}, thus ending the proof.

\subsection{Proof of Theorem \ref{t:dc}}
The following powerful result, proved by Schulte and Yukich \cite[Proposition 2.3]{SY18}, will be used repeatedly. It provides uniform upper bound for the second moment (with respect to \emph{any} probability distribution) of the second derivatives of the solution to  Stein's equation with mollified test functions. The estimate is more accurate than  \eqref{e:unif_2nd_partial} when some a priori knowledge on $\dc(\mathbf Y,N_\Sigma)$ is available. 

\begin{lemma}[See Proposition 2.3 in \cite{SY18}] \label{l:SY2.4}Let $\mathbf{Y}$ be an $\RR^m$-valued random vector and $\Sigma$ be an invertible $m\times m$ covariance matrix. Then,
\begin{align*}
\sup_{h\in\mathcal I_m} \EE \sum_{i,j=1}^m |\partial^2_{ij} f_t(\mathbf Y)|^2 \le \op{\Sigma^{-1}}^2 \Big(m^2 (\log t)^2  \dc(\mathbf Y,N_\Sigma)+ 530 m^{17/6}\Big).
\end{align*}
where the left-hand side depends on $h$ through the function $f_t$, solving the Stein's equation associated with the test function $h_t$ given by \eqref{e:h_t}. 
\end{lemma}  

We follow the smoothing approach in the proof of \cite[Theorem 1.2]{SY18}, bearing in mind that, in contrast to \cite{SY18}, we never appeal to the second order add-one-cost operator.  For clarity, we split the proof into several steps. We write $\bF=\bF(\sB)$ for simplicity.

\medskip

\noindent\underline{\it Step 1: smoothing.}  By \cite[Lemma 2.2]{SY18}, for any $t\in(0,1)$,
\begin{align*}
\dc(\wh{\mathbf F},N_\Sigma) \le \frac{4}{3} \sup_{h\in\mathcal{I}_m} |\EE h_t(\wh{\mathbf F}) -\EE h_t(N_\Sigma) | + \frac{20m}{\sqrt{2}}\frac{\sqrt{t}}{1-t}.
\end{align*}

\medskip

\noindent\underline{\it Step 2 Decomposition.} Let $h\in\mathcal{I}_m$. Applying Stein's equation with test function $h_t$ and Lemma \ref{l:duality},  one has
\begin{align}\label{e:step2-1}
\EE h_t(\wh{\mathbf F}) - \EE h_t(N_\Sigma) = \sum_{i,j=1}^m \Sigma(i,j) \EE \partial^2_{ij} f_t(\wh{\mathbf F}) -  \sum_{k=1}^m  \EE \langle D(\partial_k f_t(\wh{\mathbf F})), -DL^{-1} \wh F_k \rangle.
\end{align}
By Taylor's formula, 
\begin{align*}
D(\partial_k f_t(\wh{\mathbf F})) = \int_0^1 \sum_{j=1}^m  \partial_{jk}^2 f_t(\wh{\mathbf F}+ u D\wh{\mathbf F})D\wh F_j du,
\end{align*}
from which one sees that the second term on the right-hand side of \eqref{e:step2-1} writes
\begin{align}
\sum_{k=1}^m \EE \langle D(\partial_k f_t(\wh{\mathbf F})), -DL^{-1} \wh F_k \rangle = \sum_{j,k=1}^m \EE   \partial^2_{jk} f_t(\wh{\mathbf F}) \langle D\wh F_j, -DL^{-1}\wh F_k\rangle \nonumber\\
+ \sum_{j,k=1}^m \EE \left\langle \int_0^1 \left(\partial^2_{jk} f_t(\wh{\mathbf F}+uD\wh{\mathbf F})- \partial^2_{jk}f_t(\wh{\mathbf F}) \right) D\wh F_j du, -DL^{-1}\wh F_k\right\rangle.\label{e:step2-2}
\end{align}
The first term on the right-hand side of \eqref{e:step2-2} will be compared with the first term on the right-hand side of \eqref{e:step2-1}.  Using Taylor's formula for the second term on the right-hand side of \eqref{e:step2-2}, 
one has 
\begin{align*}
&\sum_{j,k=1}^m \EE \left\langle \int_0^1 \left(\partial^2_{jk} f_t(\wh{\mathbf F}+uD\wh{\mathbf F})- \partial^2_{jk}f_t(\wh{\mathbf F}) \right) D\wh F_j du, -DL^{-1}\wh F_k\right\rangle \\
&=\sum_{i,j,k=1}^m \EE \left\langle \int_0^1\int_0^1  \partial^3_{ijk} f_t(\wh{\mathbf F}+vuD\wh{\mathbf F})   uD\wh F_i D\wh F_j dudv, -DL^{-1}\wh F_k\right\rangle \\
&= \sum_{i,j,k=1}^m \EE \left\langle \int_0^1\int_0^1 \partial^3_{ijk} f_t(\wh{\mathbf F}+vD\wh{\mathbf F})  uD\wh F_i D\wh F_j dudv, -DL^{-1}\wh F_k\right\rangle  \\
&+ \sum_{i,j,k=1}^m \EE \left\langle \iint_{[0,1]^2} \left(\partial^3_{ijk} f_t(\wh{\mathbf F}+uv D\wh{\mathbf F}) - \partial^3_{ijk} f_t(\wh{\mathbf F}+v D\wh{\mathbf F})\right)u D\wh F_i D\wh F_j dvdu, -DL^{-1}\wh F_k \right\rangle \\
& =:  J+ L, 
\end{align*}
where, by integrating out the variable $u$ and reverting Taylor's expansion, 
\begin{align*}
J= \frac{1}{2} \sum_{j,k=1}^m \EE \langle  D(\partial^2_{jk}f_t(\wh{\mathbf F}))D\wh F_j, -DL^{-1}\wh F_k\rangle.
\end{align*}
Setting
\begin{align*}
I = \sum_{j,k=1}^m \EE (\Sigma(j,k) - \langle D\wh F_j, -DL^{-1}\wh F_k\rangle)  \partial^2_{jk} f_t(\wh{\mathbf F}),
\end{align*}
one obtains that 
\begin{align*}
|\EE h_t(\wh{\mathbf F}) - \EE h_t(N_\Sigma)| \le |I|+|J|+|L|. 
\end{align*}

\medskip

\noindent\underline{\it Step 3 Estimation of $I$.} We bound the term $I$ following an approach analogous to the one we used to bound one-dimensional Wasserstein distances. By the Cauchy-Schwarz inequality and Lemma \ref{l:SY2.4}, one has that
\begin{align*}
| I |&\le  \sum_{j,k=1}^m  \sqrt{\EE[(\partial^2_{jk} f_t(\wh{\mathbf F}))^2]} \sqrt{\EE[ (\Sigma(j,k) - \langle D\wh F_j, -DL^{-1}\wh F_k\rangle)^2]} \\
&\le \op{\Sigma^{-1}} \Big(m |\log t|  \sqrt{\dc(\mathbf{\wh F},N_\Sigma)}+ 24 m^{17/12}\Big)\sum_{j,k=1}^m \sqrt{\EE[ (\Sigma(j,k) - \langle D\wh F_j, -DL^{-1}\wh F_k\rangle)^2]} \\
&\le 4 \max(1,K')^2 \op{\Sigma^{-1}} \Big(m |\log t|  \sqrt{\dc(\mathbf{\wh F},N_\Sigma)}+ 24 m^{17/12}\Big)  (\ga_1+\ga_2^{4,p}+\ga_3).
\end{align*}
where we used the arguments leading to Theorem \ref{t:d_2} in the last inequality. 

\medskip

\noindent\underline{\it Step 4 Estimation of $J$.} By Lemma \ref{l:duality}, Cauchy-Schwarz's inequality and Lemma \ref{l:SY2.4}, 
\begin{align*}
|J|&=\frac{1}{2}\left| \sum_{j,k=1}^m \EE[\partial^2_{jk}f_t(\wh{\mathbf F})\de( -D\wh F_j DL^{-1}\wh F_k)] \right|  \\
&\le \frac{1}{2}\op{\Sigma^{-1}} \Big(m |\log t|  \sqrt{\dc(\mathbf{\wh F},N_\Sigma)}+ 24 m^{17/12}\Big)  \sum_{j,k=1}^m \frac{1}{\sigma_j\sigma_k} \sqrt{\EE[\de(DF_j DL^{-1}F_k)^2]}
\end{align*}
Applying Proposition \ref{p:sko} as we did in the proof of Theorem \ref{t:bound_w}-(ii) yields that
\begin{align*}
|J|\le 3 \max(1,K')^2 \op{\Sigma^{-1}} \Big(m |\log t|  \sqrt{\dc(\mathbf{\wh F},N_\Sigma)}+ 24 m^{17/12}\Big)  ( \ga_2^{4,p}+\ga'_2+\ga_3+\ga_5).
\end{align*}

\smallskip

\noindent\underline{\it Step 5 Estimation of $L$.} Plugging in the integral representation of $\partial^3_{ijk}f_t$ given by \eqref{e:der3}, applying the Cauchy-Schwarz inequality, then integrating out the $z$ variable shows that (see \cite[p.22]{SY18} the estimate for $J_{2,2}$ therein) 
\begin{align*}
|L| \le \op{\Sigma^{-1}}^{3/2} \frac{\sqrt{6}m^{\frac{3}{2}}}{2\sqrt{t}}\sum_{i,j,k=1}^m U_{ijk},
\end{align*}
where
\begin{multline*}
U_{ijk}:=
\sup_{\overset{z\in \RR^m}{s,u\in[0,1]}} \EE \int_{\sB}\int_0^1 |h(\sqrt{s}z-\sqrt{1-s}(\wh{\mathbf F}+uvD_x \wh{\mathbf F})) - h(\sqrt{s}z-\sqrt{1-s}(\wh{\mathbf F}+vD_x \wh{\mathbf F})) | \\  |D_x\wh F_i D_x \wh F_j D_x L^{-1} \wh F_k| dv \la(dx) 
\end{multline*}
Set $r(D_x \wh\bF):=\frac{1}{\norm{D_x\wh\bF}} D_x\wh\bF$. By a change of variables $w=v\vert|D_x\wh\bF\vert|$  in the first term and the fact that the $h$ is the indicator of a convex set, one has
\begin{align*}
U_{ijk}&\le \sup_{\overset{z\in \RR^m}{s,u\in[0,1]}} \EE \int_{\sB}\int_0^{\norm{D_x\wh\bF}} |h(\sqrt{s}z-\sqrt{1-s}(\wh{\mathbf F}+uw r(D_x \wh{\mathbf F}))) - h(\sqrt{s}z-\sqrt{1-s}(\wh{\mathbf F}+w r(D_x \wh{\mathbf F}))) | \\&\hspace{6cm} 1_{\norm{D_x\wh\bF}\le 1} \frac{|D_x\wh F_i|}{\norm{D_x\wh\bF}} |D_x \wh F_j D_x L^{-1} \wh F_k| dw \la(dx)\\
&\quad + \sup_{\overset{z\in \RR^m}{s,u\in[0,1]}} \EE \int_{\sB}\int_0^1 1_{\norm{D_x\wh\bF}\ge 1}  |D_x\wh F_i D_x \wh F_j D_x L^{-1} \wh F_k| dv \la(dx) \\
&=: U_{ijk}^{(1)} + U_{ijk}^{(2)}.
\end{align*}
It is clear that
\begin{align*}
 U_{ijk}^{(2)}&\le \EE \int_{\sB} \norm{D_x \wh\bF}|D_x\wh F_i D_x \wh F_j D_x L^{-1} \wh F_k| \la(dx)\\
 &\le \sum_{\ell=1}^m \frac{1}{\sigma_i\sigma_j\sigma_k\sigma_\ell} \int_{\sB} \EE[|D_x F_\ell D_\bx F_i D_x F_j D_x L^{-1} F_k|] \la(dx).
\end{align*}
Applying H\"older's inequality and Lemma \ref{l:E[DLF^p]} as before, one has
\begin{align*}
\sum_{i,j,k=1}^m U_{ijk}^{(2)} \le (K')^4 \left(\sum_{i=1}^m \frac{1}{\sigma_i}\right)^4 \la(\sB).
\end{align*}
Repeating the argument of \cite[p.24]{SY18} gives 
\begin{align*}
U^{(1)}_{ijk}\le R^{(1)}_{jk} + R^{(2)}_{jk}
\end{align*}
where 
\begin{align*}
R^{(1)}_{jk} &\le \sqrt{m}\op{\Sigma^{-\frac 1 2}}\left( \frac{1}{2}\sum_{\ell=1}^m\int_{\sB} \EE[|D_x \wh F_\ell|^4] \la(dx) + \frac{m}{4} \int_{\sB} \EE[|D_x \wh F_j|^4 + |D_x \wh F_k|^4]  \la(dx)   \right) \\
&\hspace{2cm}+ 2\dc(\wh \bF, N_\Sigma)\frac{1}{3}\left( \sum_{\ell=1}^m \int_{\sB} \EE[|D_x \wh F_\ell|^3]\la(dx) +  m\int_{\sB} \EE[|D_x\wh F_j|^3+|D_x\wh F_k|^3]\la(dx) \right) \\
&\le \sqrt{m} \op{\Sigma^{-\frac 1 2}} (K')^4 \Big(\frac{1}{2}\sum_{\ell=1}^m \frac{1}{\sigma_\ell^4} + \frac m {4\sigma_j^4} + \frac{m}{4\sigma_k^4} \Big) \la(\sB) \\
&\hspace{2cm} + \frac{2}{3} (K')^3  \,  \Big( \sum_{\ell=1}^m \frac 1	{\sigma_\ell^3} + \frac{m}{\sigma_j^3} + \frac{m}{\sigma_k^3}\Big) \la(\sB)\, \dc(\wh\bF, N_\Sigma)
\end{align*}
by the uniform moment condition of the  add-one cost operators, and 
\begin{align*}
R^{(2)}_{jk}\le \left(2\dc(\wh \bF, N_\Sigma) V^{(1)}_{jk} + 2\sqrt{m}\op{\Sigma^{-\frac 1 2}} V_{jk}^{(2)}\right)^{\frac 1 2}
\end{align*}
with $V^{(1)}_{jk}$ and $V^{(2)}_{jk}$ defined in the proof of Proposition \ref{p:p3}. Applying Proposition \ref{p:p3} implies
\begin{align*}
V^{(1)}_{jk} \le \sum_{i=1}^m \frac{24\max(1,K')^6}{\sigma_i\sigma_j^2\sigma_k^2} \Big(| \la^2(\sB^2\setminus \sB^2_\Delta)+ \iint_{\sB^2_\Delta} \sup_{\ell\in[m]} \EE|D_x F_\ell - D_x F_\ell(\sA_x)|]^{1-\frac{5}{p}} \la^2(dx dy) \Big)
\end{align*}
and 
\begin{align*}
V^{(2)}_{jk}\le  \sum_{i=1}^m \frac{24\max(1,K')^6}{\sigma^2_i\sigma_j^2\sigma_k^2} \Big(|\la^2(\sB^2\setminus \sB^2_\Delta) + \iint_{\sB^2_\Delta}  \sup_{\ell\in[m]} \EE|D_x F_\ell - D_x F_\ell(\sA_x)|]^{1-\frac{6}{p}}\la^2(dx dy)\Big).
\end{align*}
Combining these estimates yields
\begin{align*}
&\frac{2\sqrt{t}}{\sqrt{6}m^{\frac 3 2}} \op{\Sigma^{-1}}^{-\frac 3 2} |L| \\
&\le \sum_{i,j,k=1}^m (U_{ijk}^{(2)} + R^{(1)}_{jk} + R^{(2)}_{jk})\\
&\le (K')^4 \left(\sum_{i=1}^m \frac{1}{\sigma_i}\right)^4 \la(\sB) + (K')^4 m^{\frac 7 2}\op{\Sigma^{-\frac 1 2}} \Big(\sum_{i=1}^m \frac{1}{\sigma_i^4} \Big) \la(\sB)  \\
&\quad + 2(K')^3 m^3 \Big(\sum_{i=1}^m \frac{1}{\sigma_i^3}\Big) \la(\sB) \dc(\wh \bF, N_\Sigma) \\
& \quad + 7\max(1,K')^3 m \Big(\sum_{i=1}^m\frac{1}{\sigma_i}\Big)^{\frac 5 2} \sqrt{ \dc(\wh\bF, N_\Sigma) \la^2(\sB^2\setminus\sB^2_\Delta) } \\
& \quad + 7\max(1,K')^3 m \Big(\sum_{i=1}^m\frac{1}{\sigma_i}\Big)^{\frac 5 2} \sqrt{\dc(\wh\bF, N_\Sigma) \iint_{\sB^2_\Delta} \sup_{\ell\in[m]} \EE|D_x F_\ell - D_x F_\ell(\sA_x)|]^{1-\frac{5}{p}} \la^2(dx,dy)} \\
&\quad + 7\max(1,K')^3 m^{\frac 5 4} \op{\Sigma^{-\frac 1 2}}^{\frac 1 2}  \Big(\sum_{i=1}^m \frac 1 {\sigma_i}\Big)^3 \sqrt{\la^2(\sB^2\setminus\sB^2_\Delta)} \\
&\quad + 7\max(1,K')^3 m^{\frac 5 4} \op{\Sigma^{-\frac 1 2}}^{\frac 1 2}  \Big(\sum_{i=1}^m \frac 1 {\sigma_i}\Big)^3 \sqrt{ \iint_{\sB^2_\Delta} \sup_{\ell\in[m]} \EE|D_x F_\ell - D_x F_\ell(\sA_x)|]^{1-\frac{6}{p}} \la^2(dx,dy) },
\end{align*}
where the sum of $U^{(2)}_{ijk}$ contributes to the first term on the right-hand side, that of $R^{(1)}_{jk}$ to the second and third terms, and that of $R^{(2)}_{jk}$ to the remaining ones.  Alternatively, we established
\begin{align*}
|L| \le \frac{c_0}{\sqrt{t}} \left[(\ga_5)^2 + \ga_4 \dc(\wh\bF, N_\Sigma) + (\ga_3+\ga_2^{5,p})\left(\sum_{i=1}^m \frac{1}{\sigma_i}\right)^{\frac{1}{2}} \dc(\wh\bF, N_\Sigma)^{\frac{1}{2}} + (\ga_3 + \ga_2^{6,p}) \left(\sum_{i=1}^m \frac{1}{\sigma_i}\right)\right],
\end{align*}
where 
\begin{align}\label{e:c_complicated}
c_0= 7\sqrt{6} \max(1,K')^4 \op{\Sigma^{-1}}^{\frac 3 2} \Big(m^{\frac 3 2} + m^2\op{\Sigma^{- \frac 1 2}}^{\frac 1 2}\Big).
\end{align}

\noindent\underline{\it Step 6: Solving a recursive inequality.} Setting $\kappa=\dc(\wh\bF, N_\Sigma)$, Steps 1-5 lead to the recursive inequality
\begin{align}\label{e:recu}
\kappa &\le \frac{40m}{\sqrt{2}}\sqrt{t}+ 168 m^{\frac{17}{12}}\max(1,K')^2 \op{\Sigma^{-1}}  (|\log t| \sqrt{\kappa} + 1)(\ga_1+\ga_2^{4,p}+\ga_2'+\ga_3+\ga_5) \nonumber\\
&\quad +\frac{c_0}{\sqrt{t}} \left[\ga_5^2 + \ga_4 \kappa + (\ga_3+\ga_2^{5,p})\left(\sum_{i=1}^m \frac{1}{\sigma_i}\right)^{\frac{1}{2}} \sqrt{\kappa} + (\ga_3 + \ga_2^{6,p}) \left(\sum_{i=1}^m \frac{1}{\sigma_i}\right)\right] \notag\\
&\le c\sqrt{t} + c (|\log t|\sqrt{\kappa}+1) (\ga_1+\ga_2^{4,p}+\ga_2'+\ga_3+\ga_5) \notag \\
&\quad +\frac{c}{\sqrt{t}} \left[\ga_5^2 + \ga_4 \kappa + (\ga_3+\ga_2^{5,p})\left(\sum_{i=1}^m \frac{1}{\sigma_i}\right)^{\frac{1}{2}} \sqrt{\kappa} + (\ga_3 + \ga_2^{6,p}) \left(\sum_{i=1}^m \frac{1}{\sigma_i}\right)\right]
\end{align}
for all $t\in(0,1/2)$, where $c=20\sqrt{2} m + 4\sqrt{6} c_0$ and the constant $c_0$ is given in \eqref{e:c_complicated}.
Let 
\begin{align*}
\ga = (2c+1)(\ga_1+\ga_2^{4,p} + \ga_2^{5,p}+\ga_2^{6,p} +\ga_2'+\ga_3+\ga_4+\ga_5+ \sum_{i=1}^m \frac{1}{\sigma_i} ). 
\end{align*}
Suppose that $\kappa \ge \sum_{i=1}^m \frac 1 {\sigma_i}$, otherwise one achieves the presumably best rate and there is no need to proceed.  Let $t=\ga^2$, then one has
\begin{align*}
\kappa \le c\ga + c( 2|\log \ga|\sqrt{\kappa}+1) (\ga_1+\ga_2^{4,p}+\ga_2'+\ga_3+\ga_5) + \frac{c}{\ga} \Big(\ga_5^2 + \ga_4\kappa + (\ga_3+\ga_2^{5,p}) \kappa  + (\ga_3+\ga_2^{6,p})\sum_{i=1}^m\frac{1}{\sigma_i} \Big)
\end{align*}
Thanks to our choice of $\ga$,  one has 
\begin{align*}
\frac{c}{\ga} (\ga_4+\ga_3+\ga_2^{5,p}) \le \frac{1}{2},
\end{align*}
yielding
\begin{align*}
\kappa &\le 2c\ga + 2c(|\log \ga|\sqrt{\kappa}+1)(\ga_1+\ga_2^{4,p}+\ga_2'+\ga_3+\ga_5) + \frac{2c}{\ga}\Big(\ga_5^2+(\ga_3+\ga_2^{6,p})\sum_{i=1}^m\frac{1}{\sigma_i}\Big) \\
&\le  2c\ga + 2c(|\log \ga|\sqrt{\kappa}+1)\ga+ 4c \ga.
\end{align*}
Since $\kappa\le 1$ by the definition, we can assume that $\ga\le 1/(6c)$, otherwise the desired bound \eqref{e:kappabound} is trivial for any $c'\ge 0$. Moreover, one has
\begin{align*}
|\log\ga|\sqrt{\kappa} \le |\log \ga| (\sqrt{6c\ga} + \sqrt{2c|\log\ga|+2c}\sqrt{\ga} )\le c':=6\sqrt{c}.
\end{align*} 
 Therefore,
\begin{align}\label{e:kappabound}
\kappa \le (6c+2c(c'+1))\ga \le 20 c^{\frac 3 2} \,  \ga,
\end{align}
ending the proof.

  \appendix

\bibliographystyle{plain}

\begin{thebibliography}{123}



\bibitem{AS92} Aldous, D.; Steele, J. M. Asymptotics for Euclidean minimal spanning trees on random points. \emph{Probab. Theory Related Fields} 92(2) (1992), 247--258.


\bibitem{AT} Adler, R. A.; Taylor, J. E. \emph{Random Fields and Geometry}. Springer-Verlag, Berlin, 2007

\bibitem{Alexander96} Alexander, K. S. The RSW theorem for continuum percolation and the CLT for Euclidean minimal spanning trees. \emph{Ann. Appl. Probab.} 6(2) (1996), 466--494.

\bibitem{AW} Aza\"is, J.-M.; Wschebor, M. \emph{Level Sets and Extrema of Random Processes and Fields}. Wiley, 2009

\bibitem{Bac10a}
Baccelli, F; B{\l}aszczyszyn B. \emph{ Stochastic geometry and wireless networks: Volume I, Theory}. Now Publishers, Inc., 2009.

 {

\bibitem{BaBo} Baccelli, B.; Bordenave, Ch. The radial spanning tree of a Poisson point process. \emph{Ann. Appl. Probab.} 17(1) (2007), 305-359.

}

\bibitem{Barbour88} Barbour, A. D. Stein's method and Poisson process convergence. A celebration of applied probability.  \emph{J. Appl. Probab.} (1988), Special Vol. 25A, 175--184.

\bibitem{Barbour90} Barbour, A. D. Stein’s method for diffusion approximations. \emph{Probab. Theory Related
Fields}, 84(3) (1990), 297–322.

\bibitem{BY05} Baryshnikov, Yu.; Yukich, J. E. Gaussian limits for random measures in geometric probability. \emph{Ann. Appl. Probab.} 15(1A) (2005), 213--253.

\bibitem{Bergeretal} Berger, N.; Bollobas, B.; Borgs, C.; Chayes, J.; Riordan, O. Degree distribution of the FKP model. In: {\emph Automata, Languages and Programming}
(ed. J. Baeten, J. Lenstra, J. Parrow, and G. Woeginger), Volume 2719 of
Lecture Notes in Computer Science, pp. 725--738 (2003). Springer, Heidelberg.

\bibitem{BPSURVEY} Bourguin, S.; Peccati, G. The Malliavin-Stein method on the Poisson space. In:
G. Peccati and M. Reitzner, editors,  \emph{ Stochastic analysis for Poisson point processes},
Chapter 6  (2016), pages 185–228. Bocconi
University Press and Springer.

\bibitem{BdBDE}
Bierm\'e, H; Di Bernardino E.; Duval, C.; Estrade, A.
Lipschitz-killing curvatures of excursion sets for two-dimensional
  random fields.
\emph{  Electron. J. Stat.} 13(1) (2019), 536--581.

\bibitem{BieDes12}
Bierm{\'e} H.; Desolneux, A.
\newblock Crossings of smooth shot noise processes.
\newblock {\em Ann. Appl. Probab.}, 22(6) (2012), 2240--2281.

\bibitem{BST}
Bulinski, A.; Spodarev, E.; Timmermann, F.  Central limit theorems for the excursion set volumes of weakly
  dependent random fields.
{\em Bernoulli}, 18(1):100--118, 2012.

\bibitem{C08} Chatterjee, S. A new method of normal approximation. \emph{Ann. Probab.} 36(4) (2008), 1584--1610.

\bibitem{CS17} Chatterjee, S.; Sen, S. Minimal spanning trees and Stein's method. \emph{Ann. Appl. Probab}. 27(3) (2017), 1588--1645.

\bibitem{CGS} Chen, L. H. Y.; Goldstein, L.; Shao, Q.-M. \emph{Normal approximation by Stein's method}. Probability and its Applications (New York). Springer, Heidelberg, 2011. xii+405 pp. 

\bibitem{DP} D\"obler, C.; Peccati, G. The fourth moment theorem on the Poisson space. \emph{Ann. Probab.} 46(4) (2018), 1878--1916.


\bibitem{DG} Duerinckx, M.; Gloria, A. Multiscale second-order Poincaré  inequalities in probability.  ArXiv preprint: 1711.03158
 (2017).

\bibitem{ET14} Eichelsbacher, P.; Th\"ale, C. New Berry-Esseen bounds for non-linear functionals of Poisson random measures. \emph{Electron. J. Probab.} 19 (2014), no. 102, 25 pp. 

\bibitem{FKP} Fabrikant, A.; Koutsoupias, E.; and Papadimitriou, C.H. Heuristically
optimized trade-offs: a new paradigm for power laws in the internet. In: {\it Au-
tomata, Languages and Programming}, Volume 2380 of Lecture Notes in Com-
puter Science, pp. 110--122 (2002). Springer, Berlin.

\bibitem{Gotze91} Götze, F. On the rate of convergence in the multivariate CLT. \emph{Ann. Probab.} 19(2) (1991), 724--739.

\bibitem{JW} Jordan, J.; Wade, A. R. Phase transitions for random geometric preferential attachment graphs. \emph{Adv. in Appl. Probab.} 47(2) (2015), 565--588.


\bibitem{KL96} Kesten, H.; Lee, S.The central limit theorem for weighted minimal spanning trees on random points.  \emph{Ann. Appl. Probab.} 6(2) (1996), 495--527. 

\bibitem{Lr19} Lachièze-Rey, R. Normal convergence of non-localised geometric functionals and shot-noise excursions. \emph{Ann. Appl. Probab.} 29(5) (2019), 2613--2653.


\bibitem{LrM19} Lachi\`eze-Rey, R.; Muirhead, S. Percolation Of The Excursion Sets Of Planar Symmetric Shot Noise Fields. ArXiv preprint: 1910.14504 (2019).

\bibitem{LRP} Lachi\`eze-Rey, R.; Peccati, G. New Berry-Esséen bounds for functionals of binomial point processes.  \emph{Ann. Appl. Probab.} 27(4) (2017), 1992–2031.

\bibitem{LRSY} Lachièze-Rey, R.; Schulte, M.; Yukich, J. E. Normal approximation for stabilizing functionals. \emph{Ann. Appl. Probab.} 29(2) (2019), 931--993.

\bibitem{Last} Last, G. Stochastic analysis for Poisson processes. In:  G. Peccati and M. Reitzner (Editors) (2016). {\it Stochastic Analysis for Poisson Point Processes.}

\bibitem{LP} Last, G.; Penrose, M. D. \emph{Lectures on the Poisson process}. Institute of Mathematical Statistics Textbooks, 7. Cambridge University Press, Cambridge (2018)

\bibitem{LPPTRF} Last, G.; Penrose, M. D. Poisson process Fock space representation, chaos expansion and covariance inequalities.  \emph{Probab. Theory Related Fields} 150(3-4) (2011), 663--690.

\bibitem{LPS16} Last, G.; Peccati, G.; Schulte, M.  Normal approximation on Poisson spaces: Mehler's formula, second order Poincaré inequalities and stabilization. \emph{Probab. Theory Related Fields} 165(3-4) (2016), 667--723.

\bibitem{Lee97} Lee, S. The central limit theorem for Euclidean minimal spanning trees. I. \emph{Ann. Appl. Probab.} 7(4) (1997), 996--1020.

\bibitem{NP} Nourdin, I.; Peccati, G. \emph{Normal approximations with Malliavin calculus}. From Stein's method to universality. Cambridge Tracts in Mathematics, 192. Cambridge University Press, Cambridge, 2012. xiv+239 pp.

\bibitem{NPY} Nourdin, I.; Peccati, G.; Yang, X. Multivariate normal approximation on the Wiener space: new bounds in the convex distance. ArXiv preprint: 2001.02188 (2020).

\bibitem{PSTU10} Peccati, G.; Solé, J. L.; Taqqu, M. S.; Utzet, F. Stein's method and normal approximation of Poisson functionals. \emph{Ann. Probab.} 38(2) (2010), 443--478. 

\bibitem{PZ10} Peccati, G.; Zheng, C.Multi-dimensional Gaussian fluctuations on the Poisson space.  \emph{ Electron. J. Probab.} 15 (2010), no. 48, 1487--1527. 

\bibitem{PenroseBook} Penrose, M. D. \emph{Random Geometric Graphs}. Oxford University Press, New-York, 2003.

\bibitem{Penrose05} Penrose, M. D. Multivariate spatial central limit theorems with applications to percolation and spatial graphs.  \emph{Ann. Probab.} 33(5) (2005), no. 5, 1945--1991.

\bibitem{PW} Penrose, M. D.; Wade, A. R. Limit theory for the random on-line nearest-neighbor graph. \emph{Random Structures Algorithms} 32(2) (2008), 125--156.

\bibitem{PWsurvey} Penrose, M. D.; Wade, A. R.  Random directed and on-line networks. In: {\it New perspectives in stochastic geometry}, 248--274, Oxford Univ. Press, Oxford, 2010. 

\bibitem{PY01} Penrose, M. D.; Yukich, J. E. Central limit theorems for some graphs in computational geometry. \emph{Ann. Appl. Probab.} 11(4) (2001), 1005--1041. 

\bibitem{PY02} Penrose, M. D.; Yukich, J. E.  Limit theory for random sequential packing and deposition. \emph{Ann. Appl. Probab.} 12(1) (2002), 272--301.

\bibitem{PY05} Penrose, M. D.; Yukich, J. E. Normal approximation in geometric probability. In: \emph{Stein's method and applications}, 37--58, Lect. Notes Ser. Inst. Math. Sci. Natl. Univ. Singap., 5, Singapore Univ. Press, Singapore, 2005.

\bibitem{Reinert} Reinert, G. Three general approaches to Stein's method. In: {\it A Program in Honour of Charles Stein: Tutorial Lecture Notes}. A.D. Barbour, L.H.Y. Chen, eds. World Scientific, Singapore (2005), 183--221.

\bibitem{Schreiber} Schreiber, T. Limit theorems in stochastic geometry. In: \emph{New perspectives in stochastic geometry}, 111--144, Oxford Univ. Press, Oxford, 2010.

\bibitem{Schulte16} Schulte, M. Normal approximation of Poisson functionals in Kolmogorov distance. \emph{J. Theoret. Probab.} 29(1) (2016), 96--117. 
 {
\bibitem{SchTh} Schulte, M.; Th\"ale, Ch. Central limit theorems for the radial spanning tree. {\emph Random Structures and Algorithms}, 50 (2017), 262–286.}

\bibitem{SY18} Schulte, M.; Yukich, J.E.  Multivariate second order Poincar\'e
inequalities for Poisson functionals. \emph{Electron. J. Probab.} 24 (2019), no. 130, 1–42.

\bibitem{SZ19} Shao, Q.-M.; Zhang, Z.-S. Berry-Esseen bounds of normal and nonnormal approximation for unbounded exchangeable pairs. \emph{Ann. Probab.} 47(1) (2019), 61--108.

\bibitem{Steele88} Steele, J. M. Growth rates of Euclidean minimal spanning trees with power weighted edges. \emph{Ann. Probab.} 16(4) (1988), 1767--1787.

\bibitem{steele} Steele, J.M. Cost of sequential connection for points in space. \emph{Oper. Res. Lett.} 8 (1989), 137--142.

\bibitem{Stein} Stein, Ch. \emph{Approximate computation of expectations}. Institute of Mathematical Statistics Lecture Notes—Monograph Series, 7. Institute of Mathematical Statistics, Hayward, CA, 1986. {\rm iv}+164 pp.

\bibitem{wade07} Wade, A. R. Explicit laws of large numbers for random nearest-neighbour-type graphs. \emph{Adv. in Appl. Probab.} 39 (2007), no. 2, 326--342.

\bibitem{wade09} Wade, A. R. Asymptotic theory for the multidimensional random on-line nearest-neighbour graph. \emph{Stochastic Process. Appl.} 119(6) (2009), 1889--1911.

\bibitem{YSA} Yogeshwaran, D.; Subag, E.; Adler, R. J. Random geometric complexes in the thermodynamic regime. \emph{Prob. Th. Rel. Fields} 167 (2017), 107-142

\bibitem{Yukich} Yukich, J. E. \emph{Probability theory of classical Euclidean optimization problems}. Lecture Notes in Mathematics, 1675. Springer-Verlag, Berlin, 1998. x+152 pp.

\end{thebibliography}

\end{document}